\documentclass[lettersize,journal]{IEEEtran}
\usepackage{amsmath,amsfonts,amssymb}
\usepackage{algorithm}
\usepackage{algorithmicx}
\usepackage{array}
\usepackage[caption=false,font=normalsize,labelfont=sf,textfont=sf]{subfig}
\usepackage{textcomp}
\usepackage{stfloats}
\usepackage{url}
\usepackage{verbatim}
\usepackage{graphicx}
\usepackage{cite}

\usepackage{xcolor}
\usepackage{amsthm}
\usepackage{algpseudocode}
\usepackage{enumitem}
\usepackage{comment}
\usepackage{float}
\usepackage{caption}
\usepackage{etoc}
\usepackage{booktabs}

\usepackage{color,hyperref}
\definecolor{darkblue}{rgb}{0,0,0.8}

\newcommand{\order}[1]{\mathcal{O}\left(#1\right)}
\newcommand{\prt}[1]{\left(#1\right)}
\newcommand{\brk}[1]{\left[#1\right]}
\newcommand{\brkn}[1]{[#1]}
\newcommand{\crk}[1]{\left\{#1\right\}}

\newcommand{\inpro}[1]{\left\langle #1 \right\rangle}

\newcommand{\norm}[1]{\left\Vert #1 \right\Vert}
\newcommand{\normn}[1]{\Vert #1 \Vert}
\newcommand{\bs}[1]{\boldsymbol{#1}}

\newcommand{\R}{\mathbb{R}}
\newcommand{\T}{\intercal}
\newcommand{\E}{\mathbb{E}}
\newcommand{\x}{\mathbf{x}}

\newcommand{\1}{\mathbf{1}}
\newcommand{\sumn}{\sum_{i=1}^n}
\newcommand{\cO}{\mathcal O}
\newcommand{\cL}{\mathcal{L}}

\newcommand{\fp}[2]{f_{#1, \pi_{#2}^{#1}}}
\newcommand{\bxs}[1]{\bar{x}_*^{#1}}
\newcommand{\bx}[2]{\bar{x}_{#1}^{#2}}

\newcommand{\xitl}{x_{i,t}^{\ell}}
\newcommand{\svar}{\sigma^2_{\textup{shuffle}}}
\newcommand{\Fp}[1]{F_{\pi_{#1}}}
\newcommand{\Bxt}[1]{\1\,(\bx{t}{#1})^{\T}}
\newcommand{\BFp}[1]{\1\,(\bar{\nabla}\Fp{#1}(\x_t^{#1}))^{\T}}

\newtheorem{theorem}{Theorem}
\newtheorem{lemma}{Lemma}
\newtheorem{assumption}{Assumption}
\newtheorem{definition}{Definition}
\newtheorem{remark}{Remark}
\newtheorem{corollary}{Corollary}

\definecolor{cuhkpl}{RGB}{152,24,147}
\definecolor{bluep}{RGB}{53,130,134}

\def\kh#1{\color{black}{#1}}
\def\sp#1{\color{black}{#1}}
\newcommand{\am}[1]{{\color{black}{#1}}}
\newcommand{\xli}[1]{{\color{black}{#1}}}

\begin{document}

\title{Distributed Random Reshuffling over Networks}

\author{Kun Huang, Xiao Li, \IEEEmembership{Member, IEEE}, Andre Milzarek, Shi Pu, \IEEEmembership{Member, IEEE}, and Junwen Qiu 
\thanks{X. Li was partially supported by the National Natural Science Foundation of China (NSFC) under Grant No. 12201534 and by the Shenzhen Science and Technology Program under Grant No. RCBS20210609103708017. A. Milzarek was partly supported by the Fundamental Research Fund -- Shenzhen Research Institute of Big Data (SRIBD) Startup Fund JCYJ-AM20190601 and by the Shenzhen Science and Technology Program under Grant GXWD20201231105722002-20200901175001001. S. Pu was partially supported by Shenzhen Research Institute of Big Data under Grant T00120220003, by the National Natural Science Foundation of China under Grant 62003287, and by Shenzhen Science and Technology Program under Grant RCYX202106091032290. K. Huang was partially supported by the Internal Program of Shenzhen Research Institute of Big Data (SRIBD) (Grant No. J00220220003). (Corresponding author: Shi Pu.)}
		\thanks{K. Huang, A. Milzarek, S. Pu and J. Qiu are with the School
			of Data Science, Shenzhen Research Institute of Big Data, The Chinese
			University of Hong Kong, Shenzhen, China. 
			X. Li is with the School
			of Data Science, The Chinese
			University of Hong Kong, Shenzhen, China.
			K. Huang and J. Qiu are also with  Shenzhen Institute of Artificial Intelligence and Robotics for Society (AIRS), Shenzhen, China.
			{\tt\small (emails: kunhuang@link.cuhk.edu.cn, lixiao@cuhk.edu.cn, andremilzarek@cuhk.edu.cn, pushi@cuhk.edu.cn, junwenqiu@link.cuhk.edu.cn)} 
			
			This paper has supplementary downloadable material available at http://ieeexplore.ieee.org., provided by the author. The material includes parts of the proofs. This material is 0.3 Mb in size.
		}
	}

\markboth{Journal of \LaTeX\ Class Files,~Vol.~14, No.~8, August~2021}%
{Shell \MakeLowercase{\textit{et al.}}: A Sample Article Using IEEEtran.cls for IEEE Journals}

\IEEEpubid{0000--0000/00\$00.00~\copyright~2021 IEEE}

\maketitle

\begin{abstract}
  In this paper, we consider distributed optimization problems
  where $n$ agents, each possessing a local cost function,
  collaboratively minimize the average of the local cost functions over
  a connected network. To solve the problem, we propose a distributed random reshuffling (D-RR) algorithm  {that invokes the random reshuffling (RR) update in each agent}. We show that D-RR inherits  {favorable characteristics} of RR for both smooth strongly convex and smooth nonconvex objective functions. In particular, for smooth strongly convex objective functions, D-RR achieves  $\cO(1/T^2)$ rate of convergence (where $T$ counts the epoch number) in terms of the squared distance between the iterate and the global minimizer.  
  When the objective function is assumed to be smooth nonconvex, we show that D-RR drives the squared norm of the gradient to $0$ at a rate of $\cO(1/T^{2/3})$. These convergence results match those of centralized RR (up to constant factors)   {and outperform the distributed stochastic gradient descent (DSGD) algorithm if we run a relatively large number of epochs}. Finally, we conduct a set of numerical experiments to illustrate the efficiency of the proposed D-RR method on both strongly convex and nonconvex distributed optimization problems.
\end{abstract}


\section{Introduction}
	
	In this paper, we consider solving the following optimization problem by a group of agents $[n] := \crk{1,2,\dots, n}$ connected over a network:
	\begin{equation}
		\label{eq:P_RR}
		\min_{x\in\R^p} \frac{1}{n}\sum_{i=1}^n f_i(x) \quad \text{{with}} \quad f_i(x)= \frac{1}{m}\sum_{\ell = 1}^m f_{i,\ell}(x),
	\end{equation}
	where each $f_i:\R^p \rightarrow \R$ is a local cost function associated with the local private dataset of agent $i$, and $m$ denotes the number of data points or mini-batches in each local dataset. 
	The finite sum structure of $f_i$ naturally appears in many  machine learning and signal processing problems {that} often involve a large amount of data, i.e., $nm$ can be  {prohibitively} large.
	Designing efficient distributed algorithms to solve Problem \eqref{eq:P_RR} has attracted great interest in recent years. In particular  {and initiated by the work \cite{nedic2009distributed}}, distributed algorithms implemented over networked agents with no central controller have become popular choices.  {In this} setting, the agents only exchange information with their immediate neighbors in the network, which can help avoid the communication bottleneck of centralized protocols  {and} increase algorithmic flexibility as well as the robustness to link and node failures \cite{assran2018stochastic,lu2021optimal}.
	
	Due to the large size of data, distributed stochastic gradient (SG) methods implemented over networks\footnote{We also refer as decentralized stochastic gradient methods.} have been studied extensively to solve Problem \eqref{eq:P_RR}; see, e.g.,  \cite{nedic2018network,pu2021distributed,pu2021sharp,tang2018d,yuan2019performance,lian2017can}. 
	These methods have been shown to be efficient, among which some enjoy the comparable performance to the centralized stochastic gradient descent (SGD) algorithm under certain conditions \cite{chen2015learning2,yuan2020influence,huang2021improving,pu2021sharp,lian2017can,pu2020asymptotic}. Moreover, targeting the finite sum structure of Problem \eqref{eq:P_RR}, various distributed variance reduction (VR)-based methods have been developed to improve the algorithmic performance \cite{xin2019variance,xin2021fast}.
	Nevertheless, despite the existing SG- and VR-based (distributed) optimization schemes, random reshuffling (RR) has been a popular and successful method for solving the finite sum optimization problems {\sp in practice \cite{bertsekas2011,bottou2009,bottou2012,ying2018stochastic,gurbu2019,haochen2019,nguyen2020unified}}.
	Compared to SGD that employs uniform random sampling with replacement at each iteration, RR proceeds in a cyclic sampling fashion. Namely, at each cycle (epoch), the data points or mini-batches are permuted uniformly at random and are then selected sequentially according to the permuted order for gradient computation.
	Under a centralized computation, RR is provably more efficient than SGD in certain situations (see the literature review subsection for further details) and does not require additional storage costs  {compared to} VR-based methods; see, e.g., \cite{mishchenko2020random,nguyen2020unified}. {\kh Intuitively, RR allows to utilize all data points in every epoch which can lead to better theoretical and empirical performance. However}, the development and study of distributed RR methods over networks seem to be fairly limited and less advanced.
	This observation motivates the following question: 
	{\em Can we design an efficient distributed RR algorithm over networks with similar convergence guarantees as centralized RR?}
	
	In this paper, we give an affirmative answer to the above question. To solve Problem \eqref{eq:P_RR}, we design a novel algorithm termed distributed random reshuffling (D-RR) that  {invokes the RR update in each agent}. We will show that D-RR has comparable convergence properties to RR for both smooth strongly convex and smooth nonconvex objective functions. Here, the term `smooth'  {refers to objective functions with} Lipschitz continuous gradient. 
	
	\subsection{Related Work}
	
	There is a vast literature on solving Problem \eqref{eq:P_RR} with distributed gradient or stochastic gradient methods; see,  {e.g.}, \cite{tsitsiklis1986distributed,nedic2009distributed,nedic2010constrained,lobel2011distributed,jakovetic2014fast,xu2015augmented,kia2015distributed,shi2015extra,di2016next,qu2017harnessing,nedic2017achieving,xu2017convergence,pu2020push}. Among the existing methods, the distributed gradient descent (DGD) algorithm considered in \cite{nedic2009distributed} has drawn  {remarkable} attention due to its simplicity and robust performance.
	When the exact  {full} gradient is not available or hard to evaluate, stochastic gradient methods provide an alternative to reduce the per-iteration sampling cost for solving large-scale machine learning problems.
	The distributed  {implementations} of stochastic gradient methods over networks, including vanilla distributed stochastic gradient descent (DSGD) and more advanced methods,   {have} been shown to achieve comparable performance to  {the centralized counterparts} \cite{chen2012limiting,chen2015learning,chen2015learning2,lian2017can,pu2021sharp,tang2018d,yuan2020influence,huang2021improving,yuan2021removing,pu2021distributed,xin2021improved,alghunaim2021unified}. Particularly, recent efforts have been focusing on reducing the \emph{transient times} required by distributed algorithms to obtain the same convergence rate as centralized SGD.  {E.g.}, for strongly convex and smooth objective functions, the works \cite{huang2021improving,yuan2021removing}  {have so far achieved} the shortest transient time to  {match} the $\cO({1}/{nT})$ convergence rate of SGD, which behaves as $\cO({n}/{(1-\lambda)})$ with $1-\lambda$ denoting the spectral gap related to the mixing matrix among the agents.
	
	{\sp It is worth noting that algorithms based on stochastic gradients also work with online streaming data, which is different from the finite-sum (offline) setting we consider in this work.}

	
	

 RR is widely utilized in practice for tackling large-scale machine learning problems, such as the training of deep neural networks \cite{bertsekas2011,bottou2009,bottou2012,gurbu2019,haochen2019,nguyen2020unified}. Experimental evidence \cite{bottou2009curiously,bottou2012stochastic} indicates that  RR often has better empirical performance than SGD. Under the assumptions that the objective function is strongly convex and has Lipschitz Hessian, and the iterates are uniformly bounded, the work \cite{gurbu2019} establishes $\cO({1}/{T^2})$  {asymptotic} rate of convergence of RR with high probability in terms of the squared distance between the iterate and the unique optimal solution.  {Based on these motivating observations,} a series of works  {have started} to study the convergence behavior of RR; see \cite{haochen2019,nagaraj2019sgd,mishchenko2020random,nguyen2020unified,tran2021smg}.  {For instance, the work \cite{mishchenko2020random} establishes $\cO(1/mT^2)$ convergence rate of RR under the assumptions that each component function in the finite-sum is smooth and strongly convex, where $m$ represents the number of training samples. The authors claimed that this rate outperforms the rate of SGD under a similar setting; see Section 3.1 in \cite{mishchenko2020random}.}
 When each component function in the finite-sum is smooth nonconvex and a certain bounded variance-type assumption holds, the works \cite{mishchenko2020random,nguyen2020unified} derive a  {$\mathcal{O} (1/m^{1/3}T^{2/3})$}  rate of convergence of RR  in expectation in terms of the squared norm of gradient.  {This rate is  superior to that of SGD under a similar setting (i.e., $\mathcal{O} (1/m^{1/2}T^{1/2})$ \cite{nguyen2020unified})  after a  number of epochs related to the sample size $m$; see the  \cite[Remark 2]{nguyen2020unified}}. Very recently, the work \cite{lix2021convergence} establishes strong limit-point convergence results of RR for smooth nonconvex minimization under the Kurdyka-{\L}ojasiewicz inequality. 
	
There are also recent works considering implementing RR  over networked agents \cite{yuan2018variance,jiang2021distributed}.  
In \cite{yuan2018variance}, a distributed variance-reduced RR method was introduced and it was shown to enjoy linear convergence for smooth and strongly convex objective functions. The authors in  \cite{jiang2021distributed} considered a convex, structured problem and showed that the proposed algorithm converges to
a neighborhood of the optimal solution in expectation at a sublinear rate. Though both the two works consider distributed RR-type methods, the superiority of RR over distributed SGD-type methods in certain settings was not demonstrated. {\sp Table \ref{tab:comp} compares the theoretical results of the related works with those of D-RR.} 
\begin{table}[]
\setlength{\tabcolsep}{4pt}
	\begin{tabular}{@{}cccc@{}}
	\toprule
			 & \multicolumn{2}{c}{Strongly Convex}                                                                                   & Nonconvex                                                                             \\ \midrule
	Stepsize & Constant                                              & Decreasing                                                         & Constant                                                                                \\ \midrule
	SGD/DSGD     & $\order{\frac{\alpha}{n}}$\cite{chen2015learning}                     & $\order{\frac{1}{mnT}}$\cite{pu2021sharp}                                  & $\order{\frac{1}{\sqrt{mnT}}}$\cite{lian2017can}                                 \\
	CRR      & $\order{m\alpha^2}$\cite{mishchenko2020random} & $\tilde{\cO}\prt{\frac{1}{mT^2}}$\cite{nguyen2020unified} & $\order{\frac{1}{m^{\frac{1}{3}} T^{\frac{2}{3}}}}$\cite{mishchenko2020random,nguyen2020unified} \\
	D-RR     & $\order{m\alpha^2}$                            & $\order{\frac{1}{mT^2}}$                                  & $\order{\frac{1}{T^{\frac{2}{3}}}}$                                              \\ \bottomrule
	\end{tabular}
	\caption{{\sp A summary of related theoretical results. 
	For strongly convex objective functions with constant stepsize $\alpha$, we show the size of the final error bounds. The others are complexity results. The notion $\tilde{\cO}\prt{\cdot}$ additionally hides the logarithm factors compared to $\order{\cdot}$.}}
 	\label{tab:comp}
	\end{table}

	\subsection{Main Contributions}
	 {
	 {In this work, we propose an} efficient algorithm termed distributed random reshuffling (D-RR) for solving distributed optimization problems over networks (see Algorithm \ref{alg:DGD-RR}), which invokes the RR update in each agent. 
	
	For smooth strongly convex objective function, we conduct a non-asymptotic analysis for D-RR with both constant and decreasing stepsizes. 
	We show that with a decreasing stepsize, D-RR achieves $\cO(1/mT^2)$ rate of convergence after $T$ epochs in terms of the squared distance between the iterate and the  {global} minimizer (see Theorem \ref{thm:combined1}), where $m$ is defined in \eqref{eq:P_RR}. Note that this result is comparable to results known for centralized RR algorithms (up to constant factors depending on the network) \cite{nguyen2020unified}. In addition, under a constant stepsize $\alpha$, the expected error of D-RR decreases exponentially fast to a neighborhood of $0$ with size being of order { $\cO(m\alpha^2)$} (see Theorem \ref{cor:combined1}). If the constant is appropriately chosen, D-RR has the same $\cO(1/mT^2)$ rate of convergence {   up to logarithmic factors} as that of decreasing stepsize (see Corollary \ref{cor:scvx_complexity}). 
	The obtained results using constant stepsize are also comparable to centralized RR algorithms (up to constant factors) \cite{mishchenko2020random}.
	
	For the smooth nonconvex objective function, 
    we show that D-RR with a properly chosen constant stepsize drives the squared norm of gradient to $0$ at a rate of $\cO(1/T^{2/3})$. Such a convergence result matches that of centralized RR (up to constant factors related to the network structure and sample size $m$ in each agent). {\sp In addition, the derived result only relies on the smoothness and lower boundedness of the component functions without any bounded gradient/variance-type assumption.} 

	We now compare the theoretical convergence speed of D-RR and DSGD. For smooth strongly convex objective function, DSGD with decreasing stepsize has convergence rate $\cO(1/mnT)$ in expectation \cite{pu2021sharp}. Thus, if $T$ is relatively large (related to the number of agents $n$), then D-RR with decreasing stepsize outperforms DSGD with decreasing stepsize. A similar conclusion applies to the case where an appropriate constant stepsize is utilized.  For the smooth nonconvex case, DSGD with constant stepsize has $\cO(1/(mn)^{1/2}T^{1/2})$ rate of convergence \cite{lian2017can}. Thus, D-RR outperforms DSGD if $T$ is relatively large (related to the sample size $m$ in each agent and the number of agents $n$). We remark that even though D-RR may not have superior theoretical performance than  DSGD in certain situations,  we conduct experiments to show the empirical superiority of D-RR to DSGD in Section \ref{sec:sims}. 
	
	Finally, we believe that the RR updating strategy has great potential to be widely utilized in distributed optimization due to the popularity of centralized RR in centralized settings. 
	}
	

	
	

	\subsection{Notation}
	\label{sec:notations}
	
	Throughout this paper, we use column vectors if not otherwise specified. We use $x_{i, t}^{\ell}\in\R^p$ to denote the iterate of agent $i$ at the $t-$th epoch during the $\ell-$th inner loop. For the ease of presentation, we define stacked variables as follows:
	\begin{align*}
		\x_t^{\ell} &:= \prt{x_{1, t}^{\ell}, \dots, x_{n,t}^{\ell}}^{\T}\in\R^{n\times p}\\
		\nabla F_{\pi_{\ell}}(\x_t^{\ell}) &:= \prt{\nabla \fp{1}{\ell}(x_{1, t}^{\ell}), \dots, \nabla \fp{n}{\ell}(x_{n, t}^{\ell})}^{\T}\in\R^{n\times p}
	\end{align*}
	
	We use $\bar{x}\in\R^p$ to denote the averaged variables (among agents),  {e.g., 
	we set} $\bar{x}_t^{\ell}:= \frac{1}{n}\sumn x_{i, t}^{\ell}$ and $\bar{\nabla}\Fp{\ell}(\x_t^{\ell}):= \frac{1}{n}\sumn \nabla \fp{i}{\ell}(\xitl)$ as the average of all the agents' iterates and shuffled gradients at the $\ell-$th inner loop during the $t-$th epoch. 
	We omit the superscript $\ell$ when  {working with iterates from different epochs and} if it is clear from the context. For example, $x_{i,t}\in\R^p$ denotes the iterate of agent $i$ at the beginning of the $t-$th epoch. 
	
	We use $\norm{\cdot}$ to denote the Frobenius norm for a matrix $A\in\R^{n\times p}$ and the $\ell_2$ norm for a vector $a\in\R^{p}$.  {The term} $\inpro{a, b}$  {denotes} the inner product of two vectors $a, b\in\R^{p}$. For two matrices $A, B\in\R^{n\times p}$, $\inpro{A, B}$ is defined as
	
	\begin{equation*}
		\inpro{A, B} := \sum_{i=1}^n\inpro{A_i, B_i},
	\end{equation*}
	where $A_i$ (and $B_i$) represents the $i-$th row of $A$ (and $B$).
	
	\subsection{Organization} The rest of this paper is organized as follows. In Section \ref{sec:alg DRR}, we introduce the D-RR algorithm  along with the roadmaps for the analysis under both the strongly convex and the nonconvex setting. We then present the convergence analysis for the strongly convex case in Section \ref{sec:ana_scvx} and the nonconvex case  {is covered} in Section \ref{sec:ana_noncvx}. Numerical simulations are provided in Section \ref{sec:sims} and we conclude the paper in Section \ref{sec:conclusions}.

  \section{A Distributed Random Reshuffling Algorithm}
	\label{sec:alg DRR}
	
	In this section, we introduce the distributed random reshuffling (D-RR) algorithm and present the roadmaps for the analysis under both strongly convex and nonconvex objectives. 
	
	We start with stating the assumptions regarding the multi-agent network structure. Assume the agents are connected  {via} a graph $\mathcal{G}=(V,E)$,  with $V=[n]$ representing the set of agents, and $E\subset V\times V$  {denotes} the edge set.
	Let $w_{ij}$ represent the $(i, j)$ element of the mixing matrix $W\in\mathbb{R}^{n\times n}$ compliant with the graph $\mathcal{G}$.  The following condition is considered.
	\begin{assumption}\label{ass:W}
		The graph $\mathcal{G}$ is undirected and strongly connected. There exists a link from $i$ and $j$ ($i\neq j$) in $\mathcal{G}$ if and only if $w_{ij} >0$ and $w_{ji}>0$; otherwise, $w_{ij}=w_{ji}=0$. The mixing matrix $W$ is nonnegative, symmetric and stochastic, i.e., $W\1 =\1$. 
	\end{assumption}
	Assumption \ref{ass:W} is standard in the distributed optimization literature. Given a static and undirected graph $\mathcal{G}$, it is convenient to construct a mixing matrix $W$ satisfying the above condition. The following result introduces the contractive property of the mixing matrix, which plays a central role in controlling the consensus errors for different agents. 
	\begin{lemma}
		\label{lem:rhow}
		Let Assumption \ref{ass:W} hold and $\rho_w$ be the spectral norm of the matrix $W - \frac{1}{n}\1\1^{\T}$. Then, $\rho_w<1$ and 
		\begin{equation*}
			\norm{W\bs{\omega} - \1\bar{\omega}^{\T}} \leq \rho_w\norm{\bs{\omega}-\1\bar{\omega}^{\T}} \leq \rho_w\norm{ {\bs{\omega}}}< \norm{\bs{\omega}},
		\end{equation*}
		for any $\bs{\omega}\in\R^{n\times p}$. 
	\end{lemma}
	\begin{proof}
		See Appendix \ref{app:lem_rhow} {    in Supplementary Material}.
	\end{proof}
	
	In what follows, we let $\mathcal{N}_i:=\{j:(i,j)\in E\}$ denote the set of neighbors for agent $i$.
	
	\subsection{Algorithm}
	\label{sec:alg}
	
	We propose the distributed random reshuffling (D-RR) algorithm in Algorithm \ref{alg:DGD-RR} to solve Problem \eqref{eq:P_RR}. D-RR can be viewed as a combination of DGD and RR, where the local full gradient descent steps in DGD are replaced by local gradient descent  {steps that utilize only} one of the local (permuted) component functions $f_{i,{\kh \ell}}$ at a time. 
	Specifically, in each epoch $t$, agent $i$  {first generates a random permutation $\crk{\pi_0^i, \pi_1^i, \dots,\pi_{m-1}^i}$ of $[m]$ and then performs $m$ stochastic gradient steps accessing the local component functions $f_{i,\pi^i_{\kh \ell}}$, ${\kh \ell} \in [m]$, consecutively in a shuffled order. Hence, in contrast to SGD, each agent has guaranteed access to its full local data in every epoch. Notice that} such a sampling scheme leads to a biased stochastic gradient estimator{\kh , see, e.g., \cite{ying2018stochastic}}. 
	
	 {After} agent $i$  {has performed a local stochastic} gradient descent step,  {it} sends the intermediate result to its direct neighbors in Line \ref{line:gd} of Algorithm \ref{alg:DGD-RR}. The received information is then combined in Line \ref{line:combine}. Lines \ref{line:gd}-\ref{line:combine} are similar to the routine of DGD. 
	From an optimization perspective, Line \ref{line:combine} plays the role of a ``projection'' for the consensus constraint. Compared to the work  \cite{jiang2021distributed}, where Line \ref{line:combine} is only performed after each epoch, D-RR has better control over the consensus errors.    
	
	\begin{algorithm}
		\caption{Distributed Random Reshuffling (D-RR)}
		\label{alg:DGD-RR}
		\begin{algorithmic}[1]
			\Require Initialize $x_{i,0}$ for each agent $i\in[n]$. Determine $W = [w_{ij}]\in\R^{n\times n}$ and the stepsize sequence $\{\alpha_t\}$.
			\For{Epoch $t = 0, 1, 2,\dots, T-1$}
			\For{Agent $i$ in parallel}
			\State Independently sample a permutation $\crk{\pi_{0}^i, \pi_{1}^i,\dots, \pi_{m-1}^i}$ of $[m]$ 
			\State Set $x_{i,t}^0 = x_{i,t}$
			\For{$\ell = 0, 1,\dots, m-1$}
			\State Agent $i$ updates $x_{i,t}^{\ell + \frac{1}{2}} = x_{i, t}^{\ell} - \alpha_t \nabla f_{i,\pi_{\ell}^i}(x_{i,t}^\ell)$ and sends $x_{i,t}^{\ell + \frac{1}{2}}$ to its neighbors $j\in\mathcal{N}_i$.\label{line:gd}
			\State Agent $i$ receives $x_{j, t}^{\ell + \frac{1}{2}}$ from its neighbors and updates $x_{i, t}^{\ell + 1} = \sum_{j\in\mathcal{N}_i}w_{ij}x_{j,t}^{\ell + \frac{1}{2}}$.\label{line:combine}
			\EndFor
			\State Set $x_{i, t+ 1} = x_{i, t}^m$.
			\EndFor
			\EndFor
			\State Output $x_{i, T}$.
		\end{algorithmic}
	\end{algorithm}
	
	\begin{remark}
		\label{rem:intui}
		We present an intuitive idea why D-RR works as well as C-RR. Based on Assumption \ref{ass:W}, we have the following relation for the averaged iterates over the network agents:
		\begin{equation}
			\label{eq:RR_avg}
			\bar{x}_t^{\ell + 1}  = \bar{x}_t^{\ell} -  { \frac{\alpha_t}{n}}\sum_{i=1}^n\nabla f_{i,\pi_\ell^i}(x_{i,t}^\ell).
		\end{equation}
		Notice that Problem \eqref{eq:P_RR} can also be written as 
		\begin{equation}
			\label{eq:P_RR2}
			\min_{x\in\R^p} \frac{1}{m}\sum_{\ell={\kh 0}}^{m{\kh -1}} g_{\ell}(x) \quad \text{where} \quad  g_{\ell}(x) = \frac{1}{n}\sum_{i = 1}^n f_{i,\pi_\ell^i}(x).
		\end{equation}
		Therefore, \eqref{eq:RR_avg} can be viewed as approximately implementing the centralized RR method for solving Problem \eqref{eq:P_RR2}, since $\frac{1}{n}\sum_{i=1}^n\nabla f_{i,\pi_\ell^i}(x_{i,t}^\ell)$ is close to $\frac{1}{n}\sum_{i=1}^n\nabla f_{i,\pi_\ell^i}(\bar{x}_t^{\ell})$ when all $x_{i,t}^\ell$ are close to $\bar{x}_t^{\ell}$. {\sp To achieve this objective, the consensus error $\sum_{\ell = 0}^{m-1} \sumn \norm{\xitl - \bx{t}{\ell}}^2$ needs to be handled carefully, and thus we implement Line \ref{line:combine} (consensus step) in the inner loop to better control the aforementioned error term.}
		Such an observation is critical for our analysis for D-RR {\sp and also explains why D-RR can work.}
	\end{remark}
	
	%
	
	As a benchmark for the performance of D-RR, we consider a centralized counterpart of Algorithm \ref{alg:DGD-RR} in Algorithm \ref{alg:GD-RR}.
	\begin{algorithm}
		\caption{Centralized Random Reshuffling (C-RR)}
		\label{alg:GD-RR}
		\begin{algorithmic}[1]
			\Require Initialize $x_{0}$ and stepsize $\alpha_t$.
			\For{Epoch $t = 0, 1, 2,\dots, T-1$}
			\State Sample $\crk{\pi_{0}, \pi_{1},\dots, \pi_{m-1}}$ of $[m]$
			\State Set $x_{t}^0 = x_{t}$
			\For{$\ell = 0, 1,\dots, m-1$}
			\State Update $x_{t}^{\ell + 1} = x_t^{\ell} -  {\frac{\alpha_t}{n}}\sumn \nabla f_{i, \pi_\ell}(x_t^{\ell})$.
			\EndFor
			\State Set $x_{t+ 1} = x_{t}^m$.
			\EndFor
			\State Output $x_{i, T}$.
		\end{algorithmic}
	\end{algorithm}
	
	Using  {the notations from} Section \ref{sec:notations}, Algorithm \ref{alg:DGD-RR} can be written in a compact form \eqref{eq:comp}: 
	\begin{equation}
		\label{eq:comp}
		\x_t^{\ell + 1} = W\prt{\x_t^{\ell} - \alpha_t \nabla \Fp{\ell}(\x_t^{\ell})}.
	\end{equation}
	In the rest of this section, we introduce the roadmaps for studying the convergence properties for D-RR under both strongly-convex objectives and nonconvex objectives. Parts of the analysis follow those in \cite{mishchenko2020random}. 
	
	\subsection{Roadmap: Strongly-Convex Case}
	\label{sec:roadmap_scvx}
	
	We first consider $f_{i,{\kh \ell}}$ satisfying the following assumption.
	
	\begin{assumption}
		\label{ass:fij}
		Each $f_{i,{\kh \ell}}: \R^p\rightarrow \R$ is $\mu-$strongly convex and $L-$smooth, i.e.,  {for all $x, x'\in\R^p$ and $i, {\kh \ell}$, we have}
		\begin{align*}
			\langle\nabla f_{i,\kh \ell}(x)-\nabla f_{i,\kh \ell}(x'), x-x'\rangle&\geq \mu \Vert x-x'\Vert^2,\\
			\Vert \nabla f_{i,\kh \ell}(x)-\nabla f_{i,\kh \ell}(x')\Vert&\leq L\Vert x-x'\Vert.
		\end{align*}
	\end{assumption}    
	Under Assumption \ref{ass:fij}, there exists a unique solution $x^*\in\R^p$ to the Problem \eqref{eq:P_RR}. Moreover, for any $f$ satisfying Assumption \ref{ass:fij}, we have the following lemma. 
	
	\begin{lemma}
		\label{lem:f}
		Let $f:\R^p\rightarrow \R$ satisfy Assumption \ref{ass:fij} {and let us set  $D_f(y,x):=$ $f(y) - f(x) - \inpro{\nabla f(x), y - x}$}. Then 
		\begin{equation}
			\label{eq:muL}
      \frac{\mu}{2}\norm{x-y}^2\leq D_f(y, x)\leq \frac{L}{2}\norm{x-y}^2, \quad   {\forall x, y\in\R^p},
		\end{equation}
		\begin{equation}
			\label{eq:L}
      \frac{1}{2L}\norm{\nabla f(x) - \nabla f(y)}^2\leq D_f(y,x), \quad  {\forall x, y\in\R^p}.
		\end{equation}
	\end{lemma}
	
	\begin{proof}
		The right-hand side of \eqref{eq:muL}  {and \eqref{eq:L} are shown in \cite[Theorem 2.1.5]{nesterov2003introductory}. The left-hand side of \eqref{eq:muL} follows from the definition of $\mu-$strong convexity.}
	\end{proof}
	
	We outline the procedures of the analysis under Assumption \ref{ass:fij}. According to Remark \ref{rem:intui} and the discussions of Section 3.1 in \cite{mishchenko2020random}, given a permutation $\pi$, for Problem \eqref{eq:P_RR2}, the real limit points for $\ell = 1,\dots, m$ are defined as 
	\begin{equation}\small
		\label{eq:limit_avg}
		\begin{aligned}
      \bar{x}_*^{\ell} &:= x^* - \alpha_t\sum_{k=0}^{\ell - 1}\nabla {g_{k}}( {x^*})
      = x^* - \alpha_t\sum_{k=0}^{\ell - 1}\prt{\frac{1}{n}\sumn \nabla \fp{i}{k}(x^*)}{\kh .}
    \end{aligned}
	\end{equation}\normalsize
	
	Since $x^*$ is the solution to Problem \eqref{eq:P_RR}, we obtain
	\begin{align*}
		\bar{x}_*^{m} &= x^* -  {\frac{\alpha_t}{n}}\sum_{k=0}^{m-1}\sumn\nabla \fp{i}{k}(x^*)\\
    &= x^* -  {\frac{\alpha_t}{n}}\sumn\sum_{\ell = 1}^m \nabla f_{i,\ell}(x^*) = x^*,
	\end{align*}
	which is consistent with  {the observations in} \cite{mishchenko2020random}.  {Our next steps are now based on the following principal ideas.} 
	\begin{enumerate}
		\item \label{step:decomp} For any $\ell,\ t\geq 0$, decompose the errors for the inner loop \eqref{eq:decomp_tl} and the outer loop \eqref{eq:decomp_t} respectively:
		\small 
		\begin{align}
			\frac{1}{n}\sumn\norm{\xitl - \bxs{\ell}}^2 &= \norm{\bx{t}{\ell} - \bxs{\ell}}^2 + \frac{1}{n}\sumn \norm{\xitl - \bx{{\kh t}}{\ell}}^2,\label{eq:decomp_tl}\\
			\frac{1}{n}\sumn \norm{x_{i,t}^0 - x^*}^2 &= \norm{\bar{x}_t^0 - x^*}^2 + \frac{1}{n}\sumn \norm{x_{i,t}^0 - \bar{x}_t^0}^2.\label{eq:decomp_t}
		\end{align}\normalsize
		{\sp The first term in \eqref{eq:decomp_t} is comparable to the error term when studying the performance of centralized RR. The second term is caused by decentralization, this is the consensus error. Such arguments also apply to \eqref{eq:decomp_tl}. Dealing with the consensus error is critical and nontrivial.}
		\item Treat the inner loop and outer loop separately as it can be seen from Algorithm \ref{alg:DGD-RR} that $x^0_{i,t} = x_{i,t}$ and $x^m_{i,t} = x_{i,t + 1}$.
	\end{enumerate}
	
	Specifically, we perform the following four steps to derive the results for the strongly convex case:
	\begin{enumerate}[label=(\roman*)]
		\item We first construct two coupled recursions for  {the terms $\E[\|{\bx{t}{\ell} - \bxs{\ell}}\|^2]$ and $\E[\|{\x_t^{\ell} - \Bxt{\ell}}\|^2]$} in Lemmas \ref{lem:xbar0} and \ref{lem:cons0} respectively and introduce a Lyapunov function $H_t^{\ell}$ to decouple the two error terms.  Then, we relate the results  of the inner loop and the outer loop in Lemma \ref{lem:lya}.\label{enum:s1}
		\item Using a decreasing stepsize policy $\alpha_t = \frac{\theta}{\mu(t + K)}$, we obtain the upper bounds for $H_t$ in Lemma \ref{lem:At_ds} and $H_t^{\ell}$ in Lemma \ref{lem:Atl_ds} which are in the order of  {$\cO({1}/{(t + K)^2})$}. {\kh In addition, the expected errors of $H_t$ and $H_t^\ell$ decrease exponentially fast to a neighborhood of $0$ with size being of order $\order{m\alpha^2}$ when using a constant stepsize.} {\sp Such results are also stated in Lemmas \ref{lem:At_ds} and \ref{lem:Atl_ds}, respectively.} \label{enum:s2}
		\item Noting that the bounds in Step \ref{enum:s2} can also be applied to $\E[\|{\bx{t}{\ell} - \bxs{\ell}}\|^2]$ and $\E[\|{\bar{x}_t^0 - x^*}\|^2]$ {\sp for the two stepsize choices}, we  {utilize} them in Lemma \ref{lem:cons0} to obtain a decoupled bound for $\E[\|{\x_t^0 - \Bxt{0}}\|^2]$ in Lemma \ref{lem:cons1}. Invoking Lemma \ref{lem:cons1} in Lemma \ref{lem:xbar0}, we obtain a decoupled and refined bound for $\E[\|{\bx{t}{\ell} - \bxs{\ell}}\|^2]$ in Lemma \ref{lem:opt1}. \label{enum:s3}
		\item Finally, combining \eqref{eq:decomp_t} and  Lemmas \ref{lem:cons1} and \ref{lem:opt1}, we prove {\kh the main results, i.e., Theorem \ref{thm:combined1} for decreasing stepsizes and Theorem \ref{cor:combined1} under a constant stepsize.} \label{enum:s4}
	\end{enumerate}

	{\sp We highlight the main technical challenges of analyzing D-RR compared to the analysis of SGD-type decentralized methods \cite{pu2021sharp, huang2021improving} and that of centralized RR \cite{mishchenko2020random}. Note that the decomposition in Step \ref{step:decomp} is different from that of SGD-type decentralized methods due to the existence of the real limit point $\bx{*}{\ell}$ and the two-loop structure of D-RR. Although the analysis of the term $\E\brk{\norm{\bx{t}{\ell} - \bx{*}{\ell}}^2}$ borrows ideas from \cite{mishchenko2020random}, the extra consensus term due to decentralization imposes further challenges for constructing the Lyapunov function $H_t^{\ell}$. Utilizing the bound of $H_t^\ell$ in both the outer and the inner loops also differs from the analysis of previous SGD-type decentralized methods.}
	
	\subsection{Roadmap: Nonconvex Case}
	\label{sec:roadmap_noncvx}
	
	In the following, we formalize the {\sp assumption} for analyzing D-RR when it is utilized to solve smooth nonconvex optimization problems over networks. {\sp Basically, we only require smoothness and lower boundedness of the cost functions.}
	

	{\kh~
	\begin{assumption}
		\label{as:comp_fun}	
		Each $f_{i,\ell}:\R^p\to\R$ is $L$-smooth and bounded from below, i.e.,  {for all $x,x^\prime \in \R^p$ and $i,\ell$, we have}
    \begin{align*}
      \norm{\nabla f_{i,\ell}(x)- \nabla f_{i,\ell}(x^\prime)} &\leq L\norm{x-x^\prime} \; \text{and} \; f_{i,\ell}(x)\geq \bar{f}_{i,\ell}.
    \end{align*}
	\end{assumption}}
	
	{\sp From Assumption \ref{as:comp_fun}, we can obtain the following lemma. A similar result can be found in, e.g., \cite[Proposition 2]{mishchenko2020random}).}

	{\sp~
	\begin{lemma}
		\label{as:bounded_var}
		Let Assumption \ref{as:comp_fun} hold. Then, there exist nonnegative constants $A,B\geq 0$ such that for any $x\in\R^p$, we have 
		\begin{align}
			\label{eq:bv}
			\frac{1}{mn}\sum_{i=1}^n\sum_{\ell=1}^m\norm{\nabla f_{i,\ell}(x) - \nabla f(x)}^2\leq 2A\prt{f(x) - \bar{f}} + B^2,
		\end{align}
		where $\bar{f}:= \inf_{x\in\R^p}f(x)$, $A= 2L$, and $B^2 = 2L\cdot \prt{\bar{f} - \frac{1}{mn}\sumn\sum_{\ell = 1}^m \bar{f}_{i,\ell}}$.
	\end{lemma}
	}
        {\kh~
        \begin{proof}
            See Supplementary Material \ref{app:lem_bounded_var}.
        \end{proof}
        }
	{\sp {\kh On the one hand, as mentioned in \cite[Section 3.3]{mishchenko2020random},} Lemma \ref{as:bounded_var} bounds the variance of the gradient $\nabla f(x)$. Such a result generalizes the bounded gradient assumption $\norm{\nabla f_{i,\ell}(x)}\leq G$ and the uniformly bounded variance assumption, which is equivalent to \eqref{eq:bv} when $A = 0$. {\kh On the other hand,} Lemma \ref{as:bounded_var} also includes the so-called bounded gradient dissimilarity assumption in the distributed setting; see, e.g., in \cite[A3]{zhang2021fedpd}. From such a perspective, the result also characterizes the non-i.i.d. level among the local datasets.}
	
	
	
	Under Assumption \ref{as:comp_fun}, we are able to establish the convergence result for D-RR.  The main result is given in Theorem \ref{thm:noncvx-complexity}. {\sp The central idea is to construct a novel Lyapunov function $Q_t$ as in \eqref{eq:Qt} so that we can utilize the technique in Lemma \ref{lem:ncvx-rate}. Compared with the analysis in \cite{mishchenko2020random} that directly deals with the term $f(\bx{t}{0}) - \bar{f}$, we apply Lemma \ref{lem:ncvx-rate} to the Lyapunov function $Q_t$ which is nontrivial because of the extra terms related to the consensus error. The core steps for showing Theorem \ref{thm:noncvx-complexity} are given as follows:}
	\begin{enumerate}[label=(\roman*)]
		\item  {Based on standard analysis techniques for} optimization algorithms, we first establish an approximate descent property for D-RR in Lemma \ref{lem:descent-property}.  D-RR does not have exact descent at each epoch due to two types of errors: the consensus errors and the algorithmic errors. 
		\item To derive convergence from the approximate descent property, we further provide upper bounds for these two types of errors in {Lemma \ref{lem:noncvx-con-err}}. 
		\item {\sp 
		By carefully checking the relationship between Lemmas \ref{lem:descent-property} and \ref{lem:noncvx-con-err}, we construct a novel Lyapunov function $Q_t$ in Lemma \ref{lem:Qt}.}
		\item {\kh Finally, applying Lemma \ref{lem:ncvx-rate} to the recursion of $Q_t$ in Lemma \ref{lem:Qt} yields the complexity result of D-RR in Theorem \ref{thm:noncvx-complexity}.}
	\end{enumerate}
	
	
	%
	%
	%

	\section{Convergence Analysis: Strongly-Convex Case}
	\label{sec:ana_scvx}
	
	In this section, we analyze D-RR  {for} smooth and strongly convex objective functions and present the main convergence result.
	We first derive Lemmas \ref{lem:xbar0} and \ref{lem:cons0}, which introduce coupled recursions for two decomposed expected error terms and serve as the cornerstones for the convergence analysis.  {A novel} Lyapunov function $H_t^{\ell}$ is constructed in Lemma \ref{lem:lya} to decouple these two recursions. In Section \ref{sec:pre_scvx}, we bound the Lyapunov function $H_t^{\ell}$ and then obtain the recursion for the two decomposed errors. With all the preliminary results in hand, we are able to show the convergence results for the strongly convex case in Theorem \ref{thm:combined1} (decreasing stepsizes) and {   Theorem} \ref{cor:combined1} (constant stepsize).

	The following technical result is used repeatedly for unrolling the recursion when decreasing stepsizes are employed. 
	\begin{lemma}
		\label{lem:prod}
		For all $1<a<k$, $a\in\mathbb{N}$, and $1<\gamma{\ \leq a}/2$,  {we have} 
		\begin{equation*}
			\frac{a^{2\gamma}}{k^{2\gamma}} \leq \prod_{t=a}^{k-1}\prt{1 - \frac{\gamma}{t}}\leq \frac{a^{\gamma}}{k^{\gamma}}. 
		\end{equation*}
	\end{lemma}
	
	\begin{proof}
		See Lemma 11 in \cite{pu2021sharp}. 
	\end{proof}
	
	\subsection{Supporting Lemmas}
	\label{sec:supp_scvx}
	The contents of this subsection correspond to Step \ref{enum:s1} of Section \ref{sec:roadmap_scvx}.
	In Lemma \ref{lem:xbar0}, we follow the intuition in Remark \ref{rem:intui} and the arguments in \cite{mishchenko2020random} to construct a bound for $\E\brkn{\normn{\bx{t}{\ell} - \bxs{\ell}}^2}$. First, we  define the shuffling variance $\svar$ in \eqref{eq:shuffle_var} as a distributed counterpart of  {corresponding variance} for C-RR  {defined} in \cite{mishchenko2020random}.
	
	\begin{definition}[Shuffling Variance] \label{def:svar}
		Given 
		a permutation $\pi$ of  {$[m]$}, let $\bxs{\ell}$ be defined as in \eqref{eq:limit_avg}. The shuffling variance of agent $i$ is defined  {via} 
		\begin{equation}
			\label{eq:shuffle_var}
			\begin{aligned}
				\svar &:= \max_{\ell = 0, \dots, m-1 }\E\brk{\frac{1}{n}\sumn \fp{i}{\ell}(\bxs{\ell}) - \frac{1}{n}\sumn \fp{i}{\ell}(x^*)\right.\\
				&\quad\left. - \inpro{\frac{1}{n}\sumn\nabla \fp{i}{\ell}(\bxs{\ell}), \bxs{\ell} - x^*}}.
			\end{aligned}
		\end{equation}
	\end{definition}
	
	\begin{lemma}
		\label{lem:xbar0}
		Under Assumptions \ref{ass:W} and \ref{ass:fij}, {    let $\alpha_t \leq \frac{1}{2L}$.} We have 
		\begin{align*}
			& \E[\|{\bx{t}{\ell + 1} - \bxs{\ell + 1}}\|^2] \leq \left(1-\frac{\alpha_t\mu}{2}\right) \E[\|{\bx{t}{\ell} - \bxs{\ell}}\|^2] + 2\alpha_t\svar\\
			&\quad + \frac{2\alpha_t L^2}{n}\prt{\frac{ {1}}{\mu} + \alpha_t}\E[\|{\x_t^{\ell} - \Bxt{\ell}}\|^2].
		\end{align*}
	\end{lemma}
	
	\begin{proof}
		See Appendix \ref{app:lem_xbar0}.
	\end{proof}
	
	Compared to \cite[Theorem 1]{mishchenko2020random}, one more term $\E\brkn{\normn{\x_t^{\ell} - \Bxt{\ell}}^2}$ appears in Lemma \ref{lem:xbar0} which is related to the expected consensus error of the decision variables among different agents. Therefore, if the consensus error decreases fast enough, we can expect that Algorithm \ref{alg:DGD-RR} achieves a similar convergence rate compared to  {the} centralized RR Algorithm \ref{alg:GD-RR}. In fact, we will show in Lemma \ref{lem:cons0} that $\E\brkn{\normn{\x_t^{\ell} - \Bxt{\ell}}^2}$ decreases in the order of $\cO(\alpha_t^2)$. 
	
	Before we proceed to derive the recursion for $\E\brkn{\normn{\x_t^{\ell} - \Bxt{\ell}}^2}$, we establish a relation between the shuffling variance $\svar$ and $\sigma^2_*:= \frac{1}{mn}\sumn\sum_{\ell = 1}^m \norm{\nabla f_{i,\ell}(x^*)}^2$ in Lemma \ref{lem:sigma}, which shows that $\svar\sim \cO(m\alpha_t^2\sigma^2_*)$. 
	The term $\sigma^2_*$ is similar to the variance of the gradient noises in SGD for solving finite sum problems. Hence, it makes sense that we use it as a baseline in our analysis.
	
	\begin{lemma}
		\label{lem:sigma}
		Under Assumption \ref{ass:fij}, 
		we have 
		\begin{equation*}
			\frac{\alpha_t^2\mu m}{8}\cdot\sigma_*^2 \leq \svar\leq \frac{\alpha_t^2 L m}{4}\cdot\sigma_*^2,
		\end{equation*}
		where $\sigma^2_*:= \frac{1}{mn}\sumn\sum_{\ell = 1}^m \norm{\nabla f_{i,\ell}(x^*)}^2$.
	\end{lemma}
	
	\begin{proof}
		See Appendix \ref{app:lem_sigma} {in the Supplementary Material}.
	\end{proof}
	
	Lemma \ref{lem:cons0} presents the recursion for $\E\brkn{\normn{\x_t^{\ell} - \Bxt{\ell}}^2}$. 
	\begin{lemma}
		\label{lem:cons0}
		 {Let Assumption \ref{ass:W} hold and assume} 
		\begin{equation*}
			{\ \alpha_t \leq \sqrt{\frac{2-\rho_w^2}{24\rho_w^2(5-\rho_w^2)}}\frac{(1-\rho_w^2)}{L}}
		\end{equation*}
		 {for all $t$. Then, for all $\ell, t\ge 0$, we have}
		\begin{align*}
			&\E\brkn{\normn{\x_t^{\ell + 1} - \Bxt{\ell + 1}}^2} \leq \frac{(1+\rho_w^2)}{2}\E\brkn{\normn{\x_t^{\ell} - \Bxt{\ell}}^2}\\
			&\quad + \frac{30\alpha_t^2 nL^2}{1-\rho_w^2}\E\brkn{\normn{\bx{t}{\ell} - \bxs{\ell}}^2} + \frac{15n\rho_w^2\alpha_t^2}{1-\rho_w^2}\prt{\sigma_*^2 + 2L\svar}{\color{blue}.}
		\end{align*}
	\end{lemma}
	
	\begin{proof}
		See Appendix \ref{app:lem_cons0}.
	\end{proof}
	 {If the error term $\E\brkn{\normn{\bx{t}{\ell} - \bxs{\ell}}^2}$ is assumed to be bounded, Lemma \ref{lem:cons0} implies} that the expected consensus error $\E\brkn{\normn{\x_t^{\ell} - \Bxt{\ell}}^2}$ decreases as fast as $\cO(\alpha_t^2)$ given a stepsize sequence $\{\alpha_t\}$. Combining this result with Lemma \ref{lem:xbar0}, we can  {then} obtain convergence of D-RR  {with an overall complexity similar} to C-RR. This observation corroborates our intuitive idea in Remark \ref{rem:intui}.  {Our} remaining discussions will make this argument rigorous.
	
	The main difficulty for formalizing our previous discussions about Lemmas \ref{lem:xbar0} and \ref{lem:cons0} is that these two recursions are coupled. As a result, it is hard to unroll them directly to relate the errors corresponding to the inner loop and the outer loop of Algorithm \ref{alg:DGD-RR}. To handle this issue, we first define a Lyapunov function $H_t^{\ell}$ in \eqref{eq:lya_atl} based on the decomposition \eqref{eq:decomp_tl}: 
	\begin{equation}
		\label{eq:lya_atl}
		H_t^{\ell}:= \E\brkn{\normn{\bx{t}{\ell} - \bxs{\ell}}^2} + \omega_t\E\brkn{\normn{\x_t^{\ell} - \Bxt{\ell}}^2},
	\end{equation}
	where $\omega_t$ is specified in \eqref{eq:wt}. 
	 {Constructing an appropriate} recursion for $H_t^{\ell}$ in Lemma \ref{lem:lya}  {allows to finish} Step \ref{enum:s1}. Note that the proof of Lemma \ref{lem:lya} is similar to \cite[Lemma 12]{huang2021improving}.
	\begin{lemma}
		\label{lem:lya}
		Under Assumption \ref{ass:W} and \ref{ass:fij}, let  
		\begin{equation}
			\label{eq:wt}
			\omega_t := \frac{16\alpha_tL^2}{n\mu(1-\rho_w^2)}
		\end{equation}
		and suppose $\alpha_t$ satisfies
		\begin{equation}
			\label{eq:alphat}
			\alpha_t \leq \min\crk{{\ \sqrt{\frac{2-\rho_w^2}{24\rho_w^2(5-\rho_w^2)}}\frac{(1-\rho_w^2)}{L}}, \frac{1-\rho_w^2}{{   2}\mu}, \frac{(1-\rho_w^2)\mu}{8\sqrt{30} L^2}}.
		\end{equation}
		We have the following relation between $H_t^{\ell}$ and $H_t=H_t^0$ :
		\begin{equation}
			\label{eq:Atl}
			\begin{aligned}
				& H_t^{\ell} \leq \prt{1 - \frac{\alpha_t\mu}{4}}^{\ell} H_t^0+ 2\brk{\alpha_t\svar\prt{1 + \frac{240\alpha_t^2\rho_w^2L^3}{\mu(1-\rho_w^2)^2}} \right.\\
				&\left.\quad + \frac{120\alpha_t^3\rho_w^2 L^2}{\mu (1-\rho_w^2)^2}\sigma_*^2}\brk{\sum_{k=0}^{\ell - 1}\prt{1 - \frac{\alpha_t\mu}{4}}^k}.
			\end{aligned}
		\end{equation}
		
		{   In addition,}
		\begin{equation}
			\label{eq:At}
			\begin{aligned}
				& H_{t + 1} \leq \prt{1 - \frac{\alpha_t\mu}{4}}^{m} H_t
				+ 2\left [\alpha_t\svar\prt{1 + \frac{240\alpha_t^2\rho_w^2L^3}{\mu(1-\rho_w^2)^2}}\right.\\
				&\left.\quad + \frac{120\alpha_t^3\rho_w^2 L^2}{\mu (1-\rho_w^2)^2}\sigma_*^2\right] \brk{\sum_{k=0}^{m - 1}\prt{1 - \frac{\alpha_t\mu}{4}}^k},
			\end{aligned}
		\end{equation}
		where $H_{t + 1}$ is \am{defined} as 
        \begin{align*}
            & H_{t+1} := \E[{\norm{\bar{x}_t^m - x^*}^2}] + \omega_t \E[{\norm{\x_t^m - \Bxt{m}}^2]}\nonumber\\
            &= \E[{\norm{\bar{x}_{t + 1} - x^*}^2}] + \omega_t \E[{\|{\x_{t + 1} - \1\bar{x}_{t + 1}^{\T}}\|^2}].
        \end{align*}
	\end{lemma}
	
	\begin{proof}
		See Appendix \ref{app:lem_lya}. 
	\end{proof}
	
	Lemma \ref{lem:lya} plays a key role in decoupling the recursions in Lemmas \ref{lem:xbar0} and \ref{lem:cons0}. It can also be used to bound $\E[\|{\bx{t}{\ell} - \bxs{\ell}}\|^2]$ from the definition of $H_t^{\ell}$ in \eqref{eq:lya_atl}.
	
	\subsection{Preliminary Results}
	\label{sec:pre_scvx}
	In this section, we consider {\sp two specific stepsize choices to finish Step \ref{enum:s2}-\ref{enum:s3}: a decreasing stepsize sequence and a constant stepsize}. {\sp Specifically, the decreasing stepsize is given by}
	\begin{equation}
		\label{eq:de_a}
		\alpha_t = \frac{\theta}{{    m}\mu(t + K)},\quad \forall { t >0},
	\end{equation}
	{   for some $\theta,K>0$.}
	{ \begin{remark}
		The chosen stepsize policy \eqref{eq:de_a} is common in centralized RR algorithms when the objective function is smooth and strongly convex; see for example, \cite{nguyen2020unified}. 
	\end{remark} 
	}

	Note that relation \eqref{eq:At} provides a recursion with respect to the epoch-wise error for $H_t$. We unroll the inequality in light of Lemma \ref{lem:prod} to obtain Lemma \ref{lem:At_ds}.
	
	\begin{lemma}
		\label{lem:At_ds}
		Under Assumption { \ref{ass:W} and \ref{ass:fij}, let $K$ be chosen such that $K\geq \frac{\theta}{2m}$ and $\alpha_t = \frac{\theta}{m\mu(t + K)}$} satisfies \eqref{eq:alphat} for all $t\geq 0$, and $\theta >12$. Then we have  
		\begin{align*}
			H_t &\leq \prt{\frac{K}{t + K}}^{\frac{\theta }{4}} H_0 + \prt{mL + \frac{240\rho_w^2L^2}{\mu(1 - \rho_w^2)^2}}\\
			&\quad \cdot\frac{{   8}\theta^3 \sigma^2_*}{{    m^2}\mu^3(\theta -8)}\frac{1}{(t + K)^2}.
		\end{align*}

		{\sp
		In addition, under the constant stepsize $\alpha_t = \alpha$ that satisfies \eqref{eq:alphat}, we have 
		\begin{align*}
			H_t &\leq \prt{1 - \frac{\alpha\mu}{4}}^{mt} H_0 + \frac{4\alpha^2}{\mu} \prt{mL + \frac{240\rho_w^2L^2}{\mu(1 - \rho_w^2)^2}} \sigma_*^2.
		\end{align*}
		}
	\end{lemma}
	
	\begin{proof}
		See Appendix \ref{app:lem_At_ds} {in the Supplementary Material}.
	\end{proof}
	
	We also obtain the bound for $H_t^{\ell}$ according to Lemmas \ref{lem:lya} and \ref{lem:At_ds} by repeating the procedures in the proof of Lemma \ref{lem:At_ds}.
	\begin{lemma}
		\label{lem:Atl_ds}
		Let the conditions in Lemma \ref{lem:At_ds} hold. {\sp Under the decreasing stepsize policy, we have} 
		\begin{align*}
			H^{{   \ell}}_t &\leq \prt{\frac{K}{t + K}}^{\frac{\theta }{4}} H_0 \\
			&\quad + {\ \biggl[\frac{{\kh28}\theta^2L\sigma^2_*}{m^2\mu^3}\prt{m + \frac{240\rho_w^2L}{\mu(1-\rho_w^2)^2}}}\frac{1}{(t+ K)^2}\biggr],\quad \forall t, \ \ell.
		\end{align*}
		{\sp Under the constant stepsize, we obtain
		\begin{align*}
			H_t^{\ell} &\leq \prt{1 - \frac{\alpha\mu}{4}}^{mt} H_0 + \frac{8\alpha^2}{\mu} \prt{mL + \frac{240\rho_w^2L^2}{\mu(1 - \rho_w^2)^2}} \sigma_*^2.
		\end{align*}}
	\end{lemma}
	
	\begin{proof}
		See Appendix \ref{app:lem_Atl_ds} {in the Supplementary Material}.
	\end{proof}
	
	From the definition of $H_t^{\ell}$, we have $\E[{\|{\bx{t}{\ell} - \bxs{\ell}}\|^2}] \leq H_t^{\ell}.$ Noticing that the {\sp right-hand sides of the inequalities in Lemma \ref{lem:Atl_ds} do} not involve $\ell$, we can directly substitute the above bound into Lemma \ref{lem:cons0} and unroll the recursion with respect to $\ell$. Then, we can obtain a decoupled recursion for the consensus error in the following lemma.
	
	\begin{lemma}
		\label{lem:cons1}
		Let conditions in Lemma \ref{lem:At_ds} hold and define
		\begin{align*}
			\hat{X}_0&:= H_0 + \frac{{\kh28}(m\mu + L)\sigma^2_*}{\mu L^2},\\
			\hat{X}_1 &:= \frac{30 nL^2}{1-\rho_w^2}\hat{X}_0 + \frac{15 n\rho_w^2}{1-\rho_w^2}\sigma^2_* + \frac{mn\mu(1-\rho_w^2)}{8L}\sigma^2_*.
		\end{align*}
		{\sp Under both constant and decreasing stepsizes, we have for all $t,\ell$ that}
		\begin{align*}
			& \E\brk{\norm{\x_t^{\ell + 1} - \Bxt{\ell + 1}}^2}\\
			&\leq \prt{\frac{1+\rho_w^2}{2}}^{\ell + 1} \E\brk{\norm{\x_t^0 - \Bxt{0}}^2} + \frac{\alpha_t^2\hat{X}_1}{1-\rho_w^2}.
		\end{align*}
		
		Moreover, 
		\begin{align*}
			& \E\brk{\norm{\x_t^0 - \Bxt{0}}^2}\\
			&\leq \prt{\frac{1+\rho_w^2}{2}}^{mt}\norm{\x_0^0 - \1(\bar{x}_0^0)^{\T}}^2 + \frac{4\hat{X}_1}{(1-\rho_w^2)^2}\alpha_t^2.
		\end{align*}
	\end{lemma}
	
	\begin{proof}
		See Appendix \ref{app:lem_cons1}.
	\end{proof}
	
	Lemma \ref{lem:cons1} verifies our previous discussion that the consensus error $\E[{\|{\x_t^0 - \Bxt{0}}\|^2}]$ decreases as fast as $\order{\alpha_t^2}$. Combining Lemma \ref{lem:cons1} and Lemma \ref{lem:xbar0}, we can also derive a decoupled and refined bound for the optimization error $\E[{\|{\bar{x}_{t}^0 - x^*}\|^2}]$ in Lemma \ref{lem:opt1}.
	
	\begin{lemma}
		\label{lem:opt1}
		Let the conditions in Lemma \ref{lem:At_ds} hold. {\sp Under the decreasing stepsize policy,} we have
		\begin{align*}
			& \E\brk{\norm{\bar{x}_{t}^0 - x^*}^2} \leq \prt{\frac{K}{t + K}}^{\frac{\theta}{2}}{\norm{\bar{x}_0^0 - x^*}^2}\\
			&\quad + {\ \frac{2\theta^3L\sigma^2_*}{m\mu^3(\theta - 4)}\frac{1}{(t + K)^2}}\\
			&\quad + \frac{{   96} \theta^3 L^2\hat{X}_1}{nm^2\mu^5(1-\rho_w^2)^2(\theta - 4)}\frac{1}{(t + K)^2}\\
			&\quad + \prt{\frac{K}{t + K}}^{\frac{\theta}{2}} \frac{{   96} L^2 }{n\mu^3(1-\rho_w^2)}{\norm{\x_0^0 - \1(\bar{x}_0^0)^{\T}}^2}.
		\end{align*}
	\end{lemma}
	
	\begin{proof}
		See Appendix \ref{app:lem_opt1} {in the Supplementary Material}.
	\end{proof}
	
	\subsection{Main Results: Strongly-Convex Case}
	\label{sec:res_scvx}
	We now present the main convergence results {of D-RR}  for the strongly convex case under both decreasing stepsizes \eqref{eq:de_a} in Theorem \ref{thm:combined1} and a constant stepsize $\alpha$ in {Theorem} \ref{cor:combined1} by combining Lemmas \ref{lem:cons1} and \ref{lem:opt1}.
	
	Under decreasing stepsizes, we have the following theorem.
	\begin{theorem}
		\label{thm:combined1}
		Under Assumptions \ref{ass:W} and \ref{ass:fij}, let $\alpha_t = \frac{\theta}{{    m}\mu(t + K)}$ with $\theta > {   12}$ and $K$ be chosen as
		\begin{equation}
			\label{eq:K}
			\begin{aligned}
				K\geq \max& \crk{{\ \frac{\theta}{2},\sqrt{\frac{24 \rho_w^2 (5-\rho_w^2)L^2\theta^2}{(2-\rho_w^2)(1-\rho_w^2)^2m^2\mu^2}},\frac{2\theta}{m(1-\rho_w^2)}},\right.\\
				&\left.\quad {\ \frac{8\sqrt{30} L^2\theta}{(1-\rho_w^2)m\mu^2}}}{\kh .}
			\end{aligned}
		\end{equation}
		
		
		Then for Algorithm \ref{alg:DGD-RR}, we have 
		\begin{align*}
			& \frac{1}{n}\sumn \E\brk{\norm{x_{i, t}^0 - x^*}^2}
			\leq \prt{\frac{K}{t + K}}^{\frac{\theta}{2}}{\norm{\bar{x}_0^0 - x^*}^2}\\
			&\quad \prt{C_1 + \prt{\frac{1+\rho_w^2}{2}}^{mt}}\frac{\norm{\x_0^0 - \1(\bar{x}_0^0)^{\T}}^2}{n} + \frac{C_2}{(t + K)^2},
		\end{align*}
		where 
		{\ 
		\begin{align*}
			C_1 &:= \prt{\frac{K}{t + K}}^{\frac{\theta}{2}} \frac{{   96} L^2 }{\mu^3(1-\rho_w^2)},\\
			C_2 &:= \frac{2\theta^3L\sigma^2_*}{m\mu^3(\theta - 4)} + \frac{96 \theta^3 L^2\hat{X}_1}{nm^2\mu^5(1-\rho_w^2)^2(\theta - 4)}\\
			&\quad + \frac{4\theta^2\hat{X}_1}{n\mu^2m^2(1-\rho_w^2)^2},
		\end{align*}
		and $\hat{X}_1$ is defined in Lemma \ref{lem:cons1}.
		}
	\end{theorem}
	
	\begin{proof}
		Combining Lemmas \ref{lem:cons1} and \ref{lem:opt1} leads to the result.
	\end{proof}
	
	{ \begin{remark}
	    From Theorem \ref{thm:combined1}, D-RR enjoys the $\mathcal{O}(1/mT^2)$ rate of convergence under a decreasing stepsize policy. Compared with DSGD whose convergence rate is $\mathcal{O}(1/mnT)$, D-RR is more favorable when $T$ is relatively large compared to the number of agents $n$.
	\end{remark}
	}
	
	The convergence result {   of D-RR} under the constant stepsize $\alpha$ is stated in the next theorem.
	\begin{theorem}
		\label{cor:combined1}
		Under Assumptions \ref{ass:W} and \ref{ass:fij}, let $\alpha_t = \alpha$ satisfy \eqref{eq:alphat} and $\hat{X}_1$ be defined as in Lemma \ref{lem:cons1}. We have for Algorithm \ref{alg:DGD-RR} that
		\begin{align*}
			&\frac{1}{n}\sumn\E\brk{\norm{x_{i,t}^0 - x^*}^2} \leq \prt{1 - \frac{\alpha\mu}{4}}^{mt} H_0\\
			&\quad + \frac{4\alpha^2}{\mu} \prt{mL + \frac{240\rho_w^2L^2}{\mu(1 - \rho_w^2)^2}} \sigma_*^2\\
			&\quad + \prt{\frac{1+\rho_w^2}{2}}^{mt}\frac{\norm{\x_0^0 - \1(\bar{x}_0^0)^{\T}}^2}{n} + \frac{4\hat{X}_1}{n(1-\rho_w^2)^2}\alpha^2.
		\end{align*}
	\end{theorem}
	
	\begin{proof}
		
		Let $\alpha_t = \alpha$ in \eqref{eq:ind1}, we obtain
		\begin{equation*}
			\sum_{k=0}^{t - 1}\prt{\frac{1 + \rho_w^2}{2}}^{m(t- 1- k)}\alpha^2\leq \frac{2\alpha^2}{1-\rho_w^2}.
		\end{equation*}
		
		Then, applying Lemma \ref{lem:cons1}, it follows
		
		\begin{equation}
			\label{eq:cons_gamma}
			\begin{aligned}
				\E\brk{\norm{\x_t^0 - \Bxt{0}}^2}&\leq \prt{\frac{1+\rho_w^2}{2}}^{mt}\norm{\x_0^0 - \1(\bar{x}_0^0)^{\T}}^2\\
				&\quad + \frac{4\hat{X}_1}{(1-\rho_w^2)^2}\alpha^2.
			\end{aligned}
		\end{equation}
		
		According to \eqref{eq:decomp_t} and \eqref{eq:lya_atl}, we have 
		\begin{equation}
			\label{eq:opt_gamma}
			\begin{aligned}
				\E\brk{\norm{\bar{x}_t^0 - x^*}^2}&\leq \prt{1 - \frac{\alpha\mu}{4}}^{mt} H_0\\
				&\quad + \frac{4\alpha^2}{\mu} \prt{mL + \frac{240\rho_w^2L^2}{\mu(1 - \rho_w^2)^2}} \sigma_*^2.
			\end{aligned}
		\end{equation}
		
		Combining \eqref{eq:cons_gamma} and \eqref{eq:opt_gamma} finishes the proof.
	\end{proof}
	
		{ \begin{remark}
	    In light of Theorem \ref{cor:combined1}, under a constant stepsize policy, the expected error of D-RR decreases exponentially fast to a neighborhood of $0$ with size being of order $\mathcal{O}(m\alpha^2)$. By comparison, the expected error of DSGD decreases to a neighborhood of $0$ with size being of order $\mathcal{O}(\alpha/n)$ \cite{morral2017success}. Therefore, if $\alpha$ is relatively small, e.g., when higher accuracy is desirable, then D-RR is more favorable than DSGD.
	\end{remark}
	}
	
	To better compare the convergence results of D-RR with those of C-RR, we present the convergence result of C-RR (Algorithm \ref{alg:GD-RR}) {\sp presented in \cite{mishchenko2020random} in Theorem \ref{thm:crr}} which considers a constant stepsize depending on the number of epochs $T$.
	{\ 
	\begin{theorem}
	    \label{thm:crr}
	    {\sp (\cite[Corollary 1]{mishchenko2020random})} {Consider Algorithm \ref{alg:GD-RR},} and let Assumption \ref{ass:fij} hold and the stepsize be chosen as 
	    \begin{equation*}
	        \alpha \leq \min\crk{\frac{1}{L}, \frac{1}{\mu m T}\log\frac{\norm{x_0 - x^*}^2 \mu^2 m T^2}{\kappa \sigma^2_*}}, \text{ where }\kappa = \frac{L}{\mu}, 
	    \end{equation*}
	    then the final iterate of Algorithm \ref{alg:GD-RR} $x_{T}$ satisfies 
	    \begin{align}
	    \label{eq:scvx_com_crr}
	        \E[{\norm{x_T - x^*}^2}] \leq \exp\prt{-\alpha\mu m T}\|x_0-x^*\|^2 +  \tilde{\cO}\prt{\frac{\kappa \sigma^2_*}{\mu^2 m T^2}}.
	    \end{align}
	   
	\end{theorem}
	}
	In Corollary \ref{cor:scvx_complexity}, {we state convergence of D-RR under a similar setting as in Theorem \ref{thm:crr}. It can be seen that the first two terms in \eqref{eq:scvx_com} are comparable with those in \eqref{eq:scvx_com_crr}.} 
	In particular, if we consider a complete graph where $1-\rho_w^2 = 1$ and use the same initialization for all the agents, then the convergence result of D-RR reduces to that given in Theorem \ref{thm:crr}. 

	\begin{corollary}
		\label{cor:scvx_complexity}
		Let the conditions in {Theorem} \ref{cor:combined1} hold. Furthermore, let the stepsize $\alpha_t = \alpha$ satisfy
		\begin{equation*}
			\alpha \leq \frac{4}{{    L} m T}\log\frac{H_0\mu^2{    m}T^2}{\kappa \sigma^2_*},\quad \text{ where } \kappa = \frac{L}{\mu}.
		\end{equation*}
		Then, the final iterate $x_{i,T}^0$, for all $i\in [n]$, satisfies

		 \begin{equation}
			\label{eq:scvx_com}
			\begin{aligned}
				& \frac{1}{n}\sumn\E\brk{\norm{x_{i,T}^0 - x^*}^2} = \exp\prt{-\frac{\alpha\mu m T}{4}} H_0 \\
				& + \tilde{\cO}\prt{\frac{\kappa \sigma^2_*}{\mu^2 m T^2}}  + \tilde{\cO}\prt{\frac{{n\norm{\bar{x}_0^0 - x^*}^2 + \norm{\x_0^0 - \1(\bar{x}_0^{0})^{\T}}^2}}{(1-\rho_w^2)^3nm^2T^2}}\\
				&+ \tilde{\cO}\prt{\frac{\sigma^2_*}{m L^2(1-\rho_w^2)^3T^2}} + \tilde{\cO}\prt{\frac{\sigma^2_*}{(1-\rho_w^2)^2\mu^2m^2T^2}}\\
				&+ \exp\prt{-\frac{1-\rho_w^2}{2}mT}\frac{\norm{\x_0^0 - \1(\bar{x}_0^0)^{\T}}^2}{n}.
			\end{aligned}
		\end{equation}

    In addition, if the network topology is a complete graph and we initialize $x_{i,0}^0 = x_0$, for all $i\in[n]$, {with the further assumption $\norm{x_0 - x^*}^2 = \order{m}$}, then it holds that
    \begin{equation}
		\label{eq:scvx_rec}
		\begin{aligned}
		        & \frac{1}{n}\sumn\E\brk{\norm{x_{i,T}^0 - x^*}^2}\\
				&\quad = \exp\prt{-\frac{\alpha\mu m T}{4}} {\ \norm{x_0 - x^*}^2} + \tilde{\cO}\prt{\frac{\kappa \sigma^2_*}{\mu^2 m T^2}}.
		\end{aligned}
	\end{equation}
	\end{corollary}
	{\ \begin{remark}
		The assumption $\norm{x_0 - x^*}^2 = \order{m}$ is not restrictive. In fact, $\norm{x_0 - x^*}^2$ is usually far less than $\order{m}$.
	\end{remark}
	}
	\begin{proof}
	See Appendix \ref{app:cor_scvx_complexity} in the Supplementary Material.
	\end{proof}

	\section{Convergence Analysis: Nonconvex Case}
	\label{sec:ana_noncvx}
	
	In this section, we consider the case where the objective functions are smooth nonconvex. Under Assumption \ref{as:comp_fun}, we derive a convergence rate result of D-RR, which is comparable to that of centralized RR. Note that Assumption \ref{as:comp_fun} is standard for studying distributed nonconvex optimization algorithms; see, e.g., \cite{zeng2018nonconvex}.
	
	
	We first provide Lemma \ref{lem:descent-property}, which states an approximate descent property of  D-RR under the general smooth nonconvex setting. The error terms in this approximate descent property consist of the  consensus errors and the algorithmic errors.  In order to establish iteration complexity of D-RR, we have to further bound these two types of errors. In Lemma \ref{lem:noncvx-con-err}, we present an upper bound for the consensus errors in terms of the graph structure $\rho_w$ and the stepsize. The algorithmic errors can be easily bounded under Assumption \ref{as:comp_fun}. Finally, by invoking the bounds for these two types of errors in Lemma \ref{lem:descent-property}, we can derive the  convergence result for D-RR; see Theorem \ref{thm:noncvx-complexity}.
	
	\subsection{Supporting Lemmas}
	\label{sec:supp_bg}
	\begin{lemma}
		\label{lem:descent-property}
		Let Assumptions \ref{ass:W} and \ref{as:comp_fun} be valid. Suppose further $\alpha_t = {\kh \alpha}\leq 1/{\kh m}L$.  Then, the following holds for all $t\geq1$:
		\begin{align*}
			&f(\bx{t+1}{0}) \leq f(\bx{t}{0})  - \frac{\alpha{\kh m}}{2}\norm{\nabla f(\bx{t}{0})}^2\\
			&\quad + \frac{L^2\alpha}{n}\sum_{\ell=0}^{m-1}\norm{\x_t^\ell-\1^\T(\bx{t}{\ell})}^2 + \alpha L^2 \sum_{\ell=0}^{m-1} \norm{\bx{t}{\ell}-\bx{t}{0}}^2.
		\end{align*}
	\end{lemma}
	
	\begin{proof}
		See Appendix \ref{app:lem_descent-property} {in the Supplementary Material}.
	\end{proof}

	Lemma \ref{lem:descent-property} is an approximate descent property for D-RR. Next, we estimate the consensus error $\sum_{\ell=0}^{m-1}\|{\x_t^\ell-\1(\bx{t}{\ell})^\T}\|^2$ and the algorithmic error  $\sum_{\ell=0}^{m-1} \norm{\bx{t}{\ell}-\bx{t}{0}}^2$.  In Lemma \ref{lem:noncvx-con-err}, {\sp we bound the last two terms in the inequality of Lemma \ref{lem:descent-property}}. 
	
	\begin{lemma}
		\label{lem:noncvx-con-err}
		Suppose Assumption{\kh s} \ref{ass:W} and \ref{as:comp_fun} are valid. Let the stepsize $\alpha_t = \alpha$ satisfy
		\begin{align*}
			\alpha \leq \min\crk{\frac{1}{2\sqrt{6}mL}, \frac{1-\rho_w^2}{4\sqrt{6}L}}.
		\end{align*}
		
		Then, the following holds for all $t\geq1$:
		{\sp~
		\begin{align*}
			&\cL_{t}\leq \frac{4}{n(1-\rho_w^2)}\norm{\x_t^0 - \1(\bx{t}{0})^{\T}}^2\\
			&\quad + \frac{6m^2\alpha^2B^2(m + 4)}{(1-\rho_w^2)^2} + \frac{6\alpha^2m(4 + m^2)}{(1-\rho_w^2)^2}\norm{\nabla f(\bx{t}{0})}^2\\
			&\quad + \frac{12m^2\alpha^2 A(4 + m)}{(1-\rho_w^2)}\prt{f(\bx{t}{0})  - \bar{f}},
		\end{align*}
		where $\cL_t:=\frac{1}{n}\sum_{\ell=0}^{m-1}\norm{\x_t^{\ell} - \1(\bx{t}{\ell})^{\T}}^2 + \sum_{\ell = 0}^{m-1}\norm{\bx{t}{\ell} - \bx{t}{0}}^2$.
		
		In addition, we have
		\begin{align*}
			&\norm{\x_{t+1}^0 - \1(\bx{t+1}{0})^{\T}}^2 \leq \prt{\frac{1 + \rho_w^2}{2}}^m\norm{\x_t^0 - \1(\bx{t}{0})^{\T}}^2\\
			&\quad + \frac{12\alpha^2 nL^2}{1-\rho_w^2}\cL_t + \frac{6\alpha^2mn B^2}{1-\rho_w^2}  +  \frac{6\alpha^2mn}{1-\rho_w^2}\norm{\nabla f(\bx{t}{0})}^2\\
			&\quad + \frac{12A\alpha^2mn}{1-\rho_w^2}\prt{f(\bx{t}{0}) - \bar{f}}.
		\end{align*}
		}

	\end{lemma}
	\begin{proof}
		See Appendix \ref{app:lem_noncvx-con-err} in the Supplementary Material.
	\end{proof}

	{\sp~
	The extra term $\norm{\x_t^0 -\1(\bx{t}{0})^{\T}}^2$ in $\cL_t$ inspires us to consider the Lyapunov function $Q_t$ in \eqref{eq:Qt}:}
	{\kh~
	\begin{align}
		\label{eq:Qt}
		Q_t:= f(\bx{t}{0}) - \bar{f} + \frac{16\alpha L^2}{n(1-\rho_w^2)^2}\norm{\x_{t}^0 - \1(\bx{t}{0})^{\T}}^2.
	\end{align}
	\begin{lemma}
		\label{lem:Qt}
		Suppose Assumptions \ref{ass:W} and \ref{as:comp_fun} are valid. Let the stepsize $\alpha_t = \alpha$ satisfy
		\begin{align*}
			\alpha\leq \min\crk{\frac{1-\rho_w^2}{4\sqrt{3}L(m+2)}, \frac{(1-\rho_w^2)^{3/2}}{16\sqrt{6}L}}.
		\end{align*}
		Then, we have 
		\begin{align*}
			&Q_{t + 1} \leq \brk{1 + \frac{12m^2\alpha^3L^2 A(4 + m)}{(1-\rho_w^2)} + \frac{384A\alpha^3L^2m}{(1-\rho_w^2)^3}} Q_t \\
			&\quad - \frac{m\alpha}{4}\norm{\nabla f(\bx{t}{0})}^2 + \frac{6m\alpha^3 L^2B^2[m(m + 4) + 32]}{(1-\rho_w^2)^3},
		\end{align*}
	\end{lemma}
	\begin{proof}
		See Appendix \ref{app:lem_Qt} in the Supplementary Material.
	\end{proof}
	}

	{\sp Lemma \ref{lem:ncvx-rate} in \cite[Lemma 6]{mishchenko2020random} provides a direct link connecting Lemma \ref{lem:Qt} to Theorem \ref{thm:noncvx-complexity} in the next subsection.}

\xli{
\begin{lemma}\label{lem:ncvx-rate}
	(\cite[Lemma 6]{mishchenko2020random}) Suppose that there exist constants $a, b, c \geq 0$ and nonnegative sequences $\left(s_t\right)_{t=0}^T,\left(q_t\right)_{t=0}^T$ such that for any $t$ satisfying $0 \leq t \leq T$, we have the recursion
	$$
	s_{t+1} \leq(1+a) s_t-b q_t+c .
	$$
	Then, the following holds:
	$$
	\min _{t=0, \ldots, T-1} q_t \leq \frac{(1+a)^T}{b T} s_0+\frac{c}{b}.
	$$
\end{lemma}

}

	
	\subsection{Main Results: Nonconvex Case}
	\label{sec:res_bg}
	Equipped with Lemmas \ref{lem:Qt} and \ref{lem:ncvx-rate}, we are ready to derive the convergence result of D-RR for smooth nonconvex optimization problems over the networks.  The main result is given in the following theorem. 
	\xli{
	\begin{theorem}
		\label{thm:noncvx-complexity}
		{\sp Suppose Assumptions \ref{ass:W} and \ref{as:comp_fun} are valid. Suppose further that $\alpha_t = \alpha = \frac{\eta}{mT^{\gamma}}$ with $\gamma \in (0,1)$, where $\eta>0$ is some constant such that
		\begin{align*}
		   0<\alpha &\leq \min\Bigg\{\frac{1-\rho_w^2}{4\sqrt{3}L(m+2)}, \frac{(1-\rho_w^2)^{3/2}}{16\sqrt{6}L}, \\&\prt{\frac{12m^2L^2A(4 + m)}{1-\rho_w^2} + \frac{384L^2Am}{(1-\rho_w^2)^3}}^{-1/3} \frac{1}{T^{1/3}} \Bigg\},
		\end{align*}
		 where}
		$T$ denotes the total number of iterations. Then, we have
		\begin{align*}
	&\min _{t=0, \ldots, T-1} \norm{\nabla f(\bx{t}{0})}^2 \\ & \leq \frac{12}{ \eta T^{1-\gamma}} \prt{\prt{f(\bx{0}{0}) - \bar{f} } + \frac{16\alpha L^2}{n(1-\rho_w^2)^2}\norm{\x_{0}^0 - \1(\bx{0}{0})^{\T}}^2}\\
		&\quad + \frac{18\eta^2B^2L^2 [m(m + 4) + 32]}{m^2(1-\rho_w^2)^3 T^{2\gamma} }.
	\end{align*}
		Consequently, the optimal rate is attained when $\gamma=1/3$ and it follows that
		\begin{align*}
			&\min _{t=0, \ldots, T-1} \norm{\nabla f(\bx{t}{0})}^2 \\ & \leq \frac{12}{ \eta T^{2/3}} \prt{\prt{f(\bx{0}{0}) - \bar{f} } + \frac{16\alpha L^2}{n(1-\rho_w^2)^2}\norm{\x_{0}^0 - \1(\bx{0}{0})^{\T}}^2}\\
		&\quad + \frac{18\eta^2B^2L^2 [m(m + 4) + 32]}{m^2(1-\rho_w^2)^3 T^{2/3} }.
		\end{align*}
	\end{theorem}
		
	\begin{proof}
		{\sp Set} 
		\begin{align*}
			b&:= \frac{m\alpha}{4},\\ 
			a&:=  \frac{12m^2\alpha^3L^2A(4 + m)}{1-\rho_w^2} + \frac{384\alpha^3L^2Am}{(1-\rho_w^2)^3}, \\
			c&:= \frac{6m\alpha^3B^2L^2 [m(m + 4) + 32]}{(1-\rho_w^2)^3}, \\
			s_t&:=\prt{f(\bx{t}{0}) - \bar{f} } + \frac{16\alpha L^2}{n(1-\rho_w^2)^2}\norm{\x_{t}^0 - \1(\bx{t}{0})^{\T}}^2\\
			q_t&:=\norm{\nabla f(\bx{t}{0})}^2.
		\end{align*}
	
	Invoking Lemma \ref{lem:ncvx-rate}, using the inequality $(1+a)^T \leq \exp(aT) \leq 3$ for all $\alpha$ satisfying
	\[ 0 < \alpha \leq \prt{\frac{12m^2L^2A(4 + m)}{1-\rho_w^2} + \frac{384L^2Am}{(1-\rho_w^2)^3}}^{-1/3} \frac{1}{T^{1/3}} , \]
	{\sp it follows that}
	\begin{align*}
		&\min _{t=0, \ldots, T-1} \norm{\nabla f(\bx{t}{0})}^2 \\ & \leq \frac{12}{m\alpha T} \prt{\prt{f(\bx{0}{0}) - \bar{f} } + \frac{16\alpha L^2}{n(1-\rho_w^2)^2}\norm{\x_{0}^0 - \1(\bx{0}{0})^{\T}}^2}\\
		&\quad + \frac{18\alpha^2B^2L^2 [m(m + 4) + 32]}{(1-\rho_w^2)^3}.
	\end{align*}
	Substituting $\alpha = \frac{\eta}{mT^\gamma}$, we have
	\begin{align*}
	&\min _{t=0, \ldots, T-1} \norm{\nabla f(\bx{t}{0})}^2 \\ & \leq \frac{12}{ \eta T^{1-\gamma}} \prt{\prt{f(\bx{0}{0}) - \bar{f} } + \frac{16\alpha L^2}{n(1-\rho_w^2)^2}\norm{\x_{0}^0 - \1(\bx{0}{0})^{\T}}^2}\\
		&\quad + \frac{18\eta^2B^2L^2 [m(m + 4) + 32]}{m^2(1-\rho_w^2)^3 T^{2\gamma} }.
	\end{align*}
	This completes the proof.
	\end{proof}
	}

	{ \begin{remark}
	    It can be seen from Theorem \ref{thm:noncvx-complexity} that D-RR enjoys the $\order{1/T^{2/3}}$ rate of convergence for solving smooth nonconvex problems.
	    Similar to the strongly convex case, noticing that DSGD with constant stepsize has $\cO(1/(mn)^{1/2}T^{1/2})$ rate of convergence \cite{lian2017can}, D-RR outperforms DSGD if $T$ is relatively large (related to the sample size $m$ in each agent and the number of agents $n$).
	\end{remark}}

	\section{Experimental Results}
		\label{sec:sims}
 		In this section, we provide two numerical examples which illustrate the performance of D-RR. For both the strongly convex problem \eqref{eq:logistic} and the nonconvex problem \eqref{eq:ncvx_logistic}, we show the proposed D-RR algorithm outperforms SGD and DSGD when $T$ is large enough. All the results in the following experiments are averaged over ten repeated runs {\kh if not otherwise specified}. {\sp We consider the \textit{heterogeneous} data setting for all the experiments, where the data samples are first sorted according to their labels and are then partitioned among the agents. Some codes are from \cite{qureshi2020s}.} 
		
		\begin{figure*}[htbp]
			\centering
			\subfloat[Grid graph, $n = 16$.]{\includegraphics[width=0.33\textwidth]{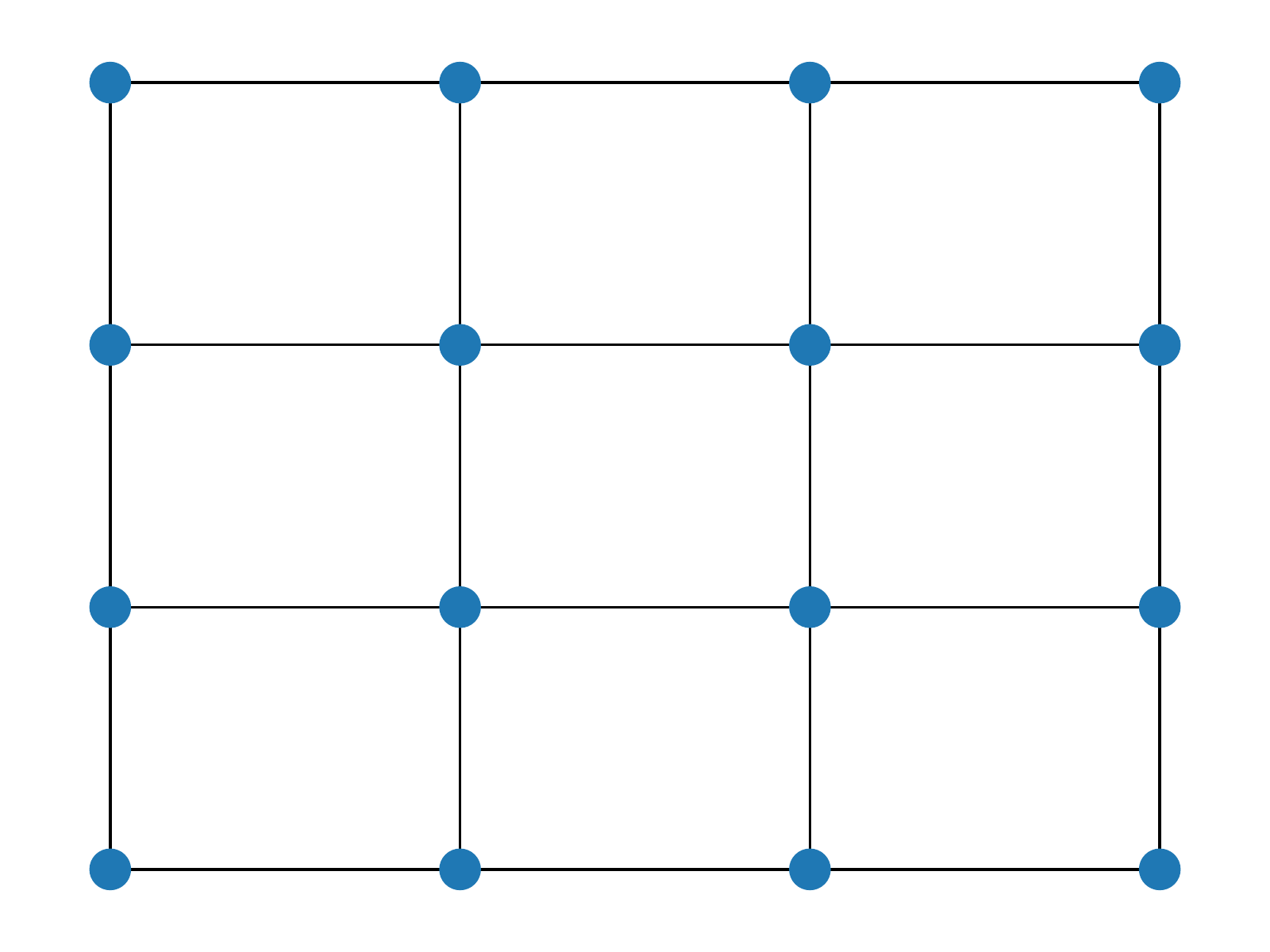}}
			\subfloat[Exponential graph, $n = 16$.]{\includegraphics[width=0.33\textwidth]{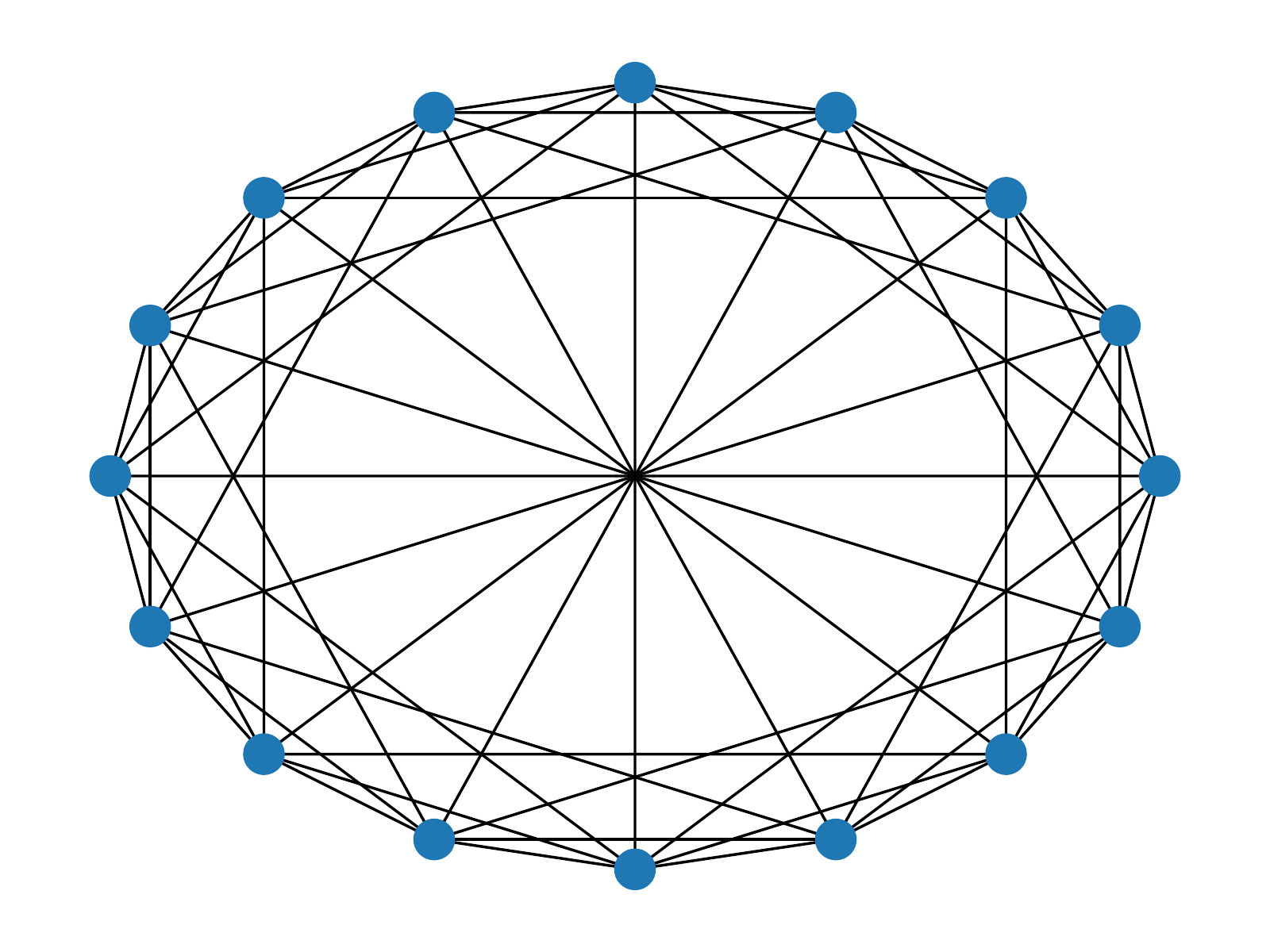}}
			\subfloat[Erd\H{o}s-R\'enyi graph, $n=16$.]{\includegraphics[width=0.33\textwidth]{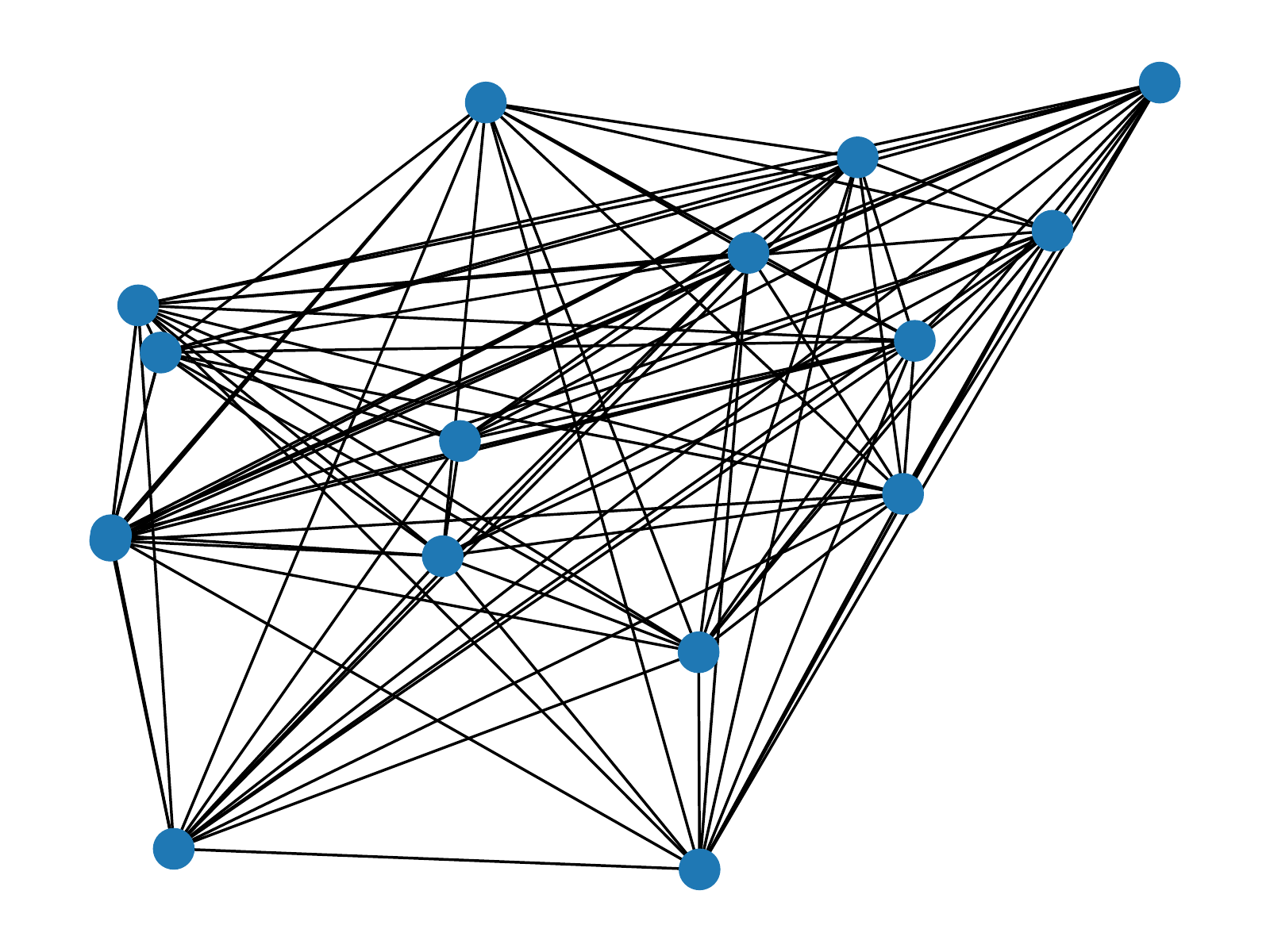}}
			\caption{Illustration of three graph topologies. {\sp The spectral gaps for the three graphs increase from left to right. Each node in the exponential graph is connected to its $2^0, 2^1, \dots$ neighbors, and the probability for edge creation in the Erd\H{o}s-R\'enyi graph is set as $0.8$.}}
			\label{fig:graph}
		\end{figure*}

		\subsection{Logistic regression}
		\label{sec:logistic}
		We consider a binary classification problem using logistic regression \eqref{eq:logistic} and the MNIST dataset \cite{mnist}.
		Each agent possesses a distinct local dataset $\mathcal{S}_i$ selected from the whole dataset $\mathcal{S}$. The classifier can then be obtained by solving the following optimization problem using all the agents' local datasets $\mathcal{S}_i, i=1,2,...,n$:
		\begin{subequations}
			\label{eq:logistic}
			\begin{align}
				&\min_{x\in\R^{p}} f(x) = \frac{1}{n}\sum_{i=1}^n f_i(x),\\
				&f_i(x) := \frac{1}{|\mathcal{S}_i|} \sum_{j\in\mathcal{S}_i} \log\left[1 + \exp(-x^{\T}u_jv_j)\right] + \frac{\rho}{2}\norm{x}^2,
			\end{align}
		\end{subequations}
		{\kh where $\rho$ is set as $\rho = 0.2$.}

		We compare D-RR (Algorithm \ref{alg:DGD-RR}) with DSGD, SGD, {\kh DPG-RR \cite{jiang2021distributed}}, and centralized RR (Algorithm \ref{alg:GD-RR}) for classifying handwritten digits $2$ and $6$ on the MNIST dataset over {\sp a grid graph, an exponential graph, and an Erd\H{o}s-R\'enyi graph (all with $n = 16$), respectively}. 
		We consider {\kh both constant stepsizes (Fig. \ref{fig:mnist_const}) and decreasing stepsizes (Fig. \ref{fig:mnist_de})} for all the methods. {\kh The methods use the same initialization in each figure.}

		\begin{figure*}[htbp]
			\centering
			\subfloat[Grid graph, $n = 16$.]{\includegraphics[width = 0.33\textwidth]{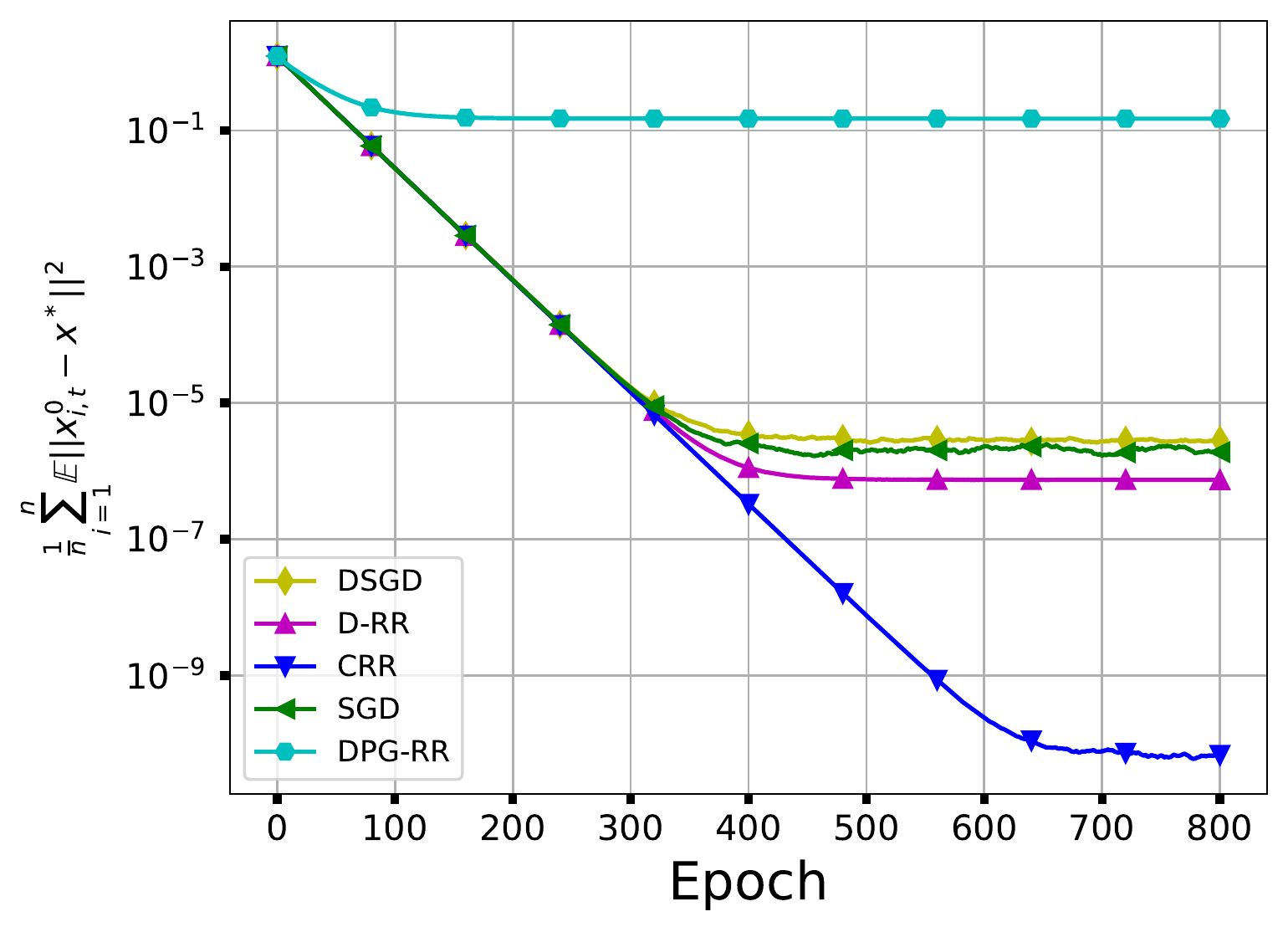}\label{fig:mnist_const_grid}}
			\subfloat[Exponential graph, $n = 16$.]{\includegraphics[width = 0.33\textwidth]{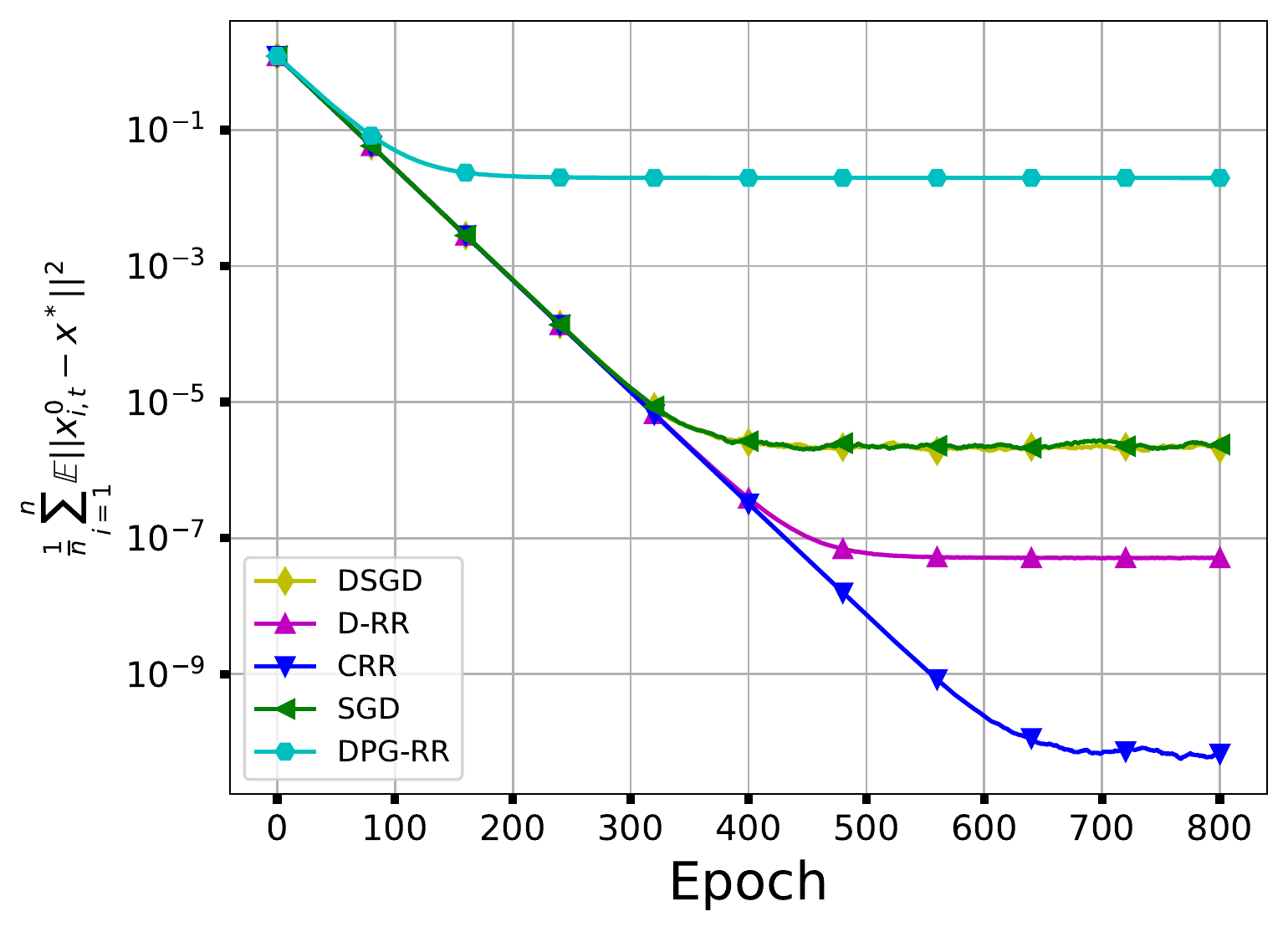}\label{fig:mnist_const_exp}}
			\subfloat[Erd\H{o}s-R\'enyi graph, $n = 16$.]{\includegraphics[width = 0.33\textwidth]{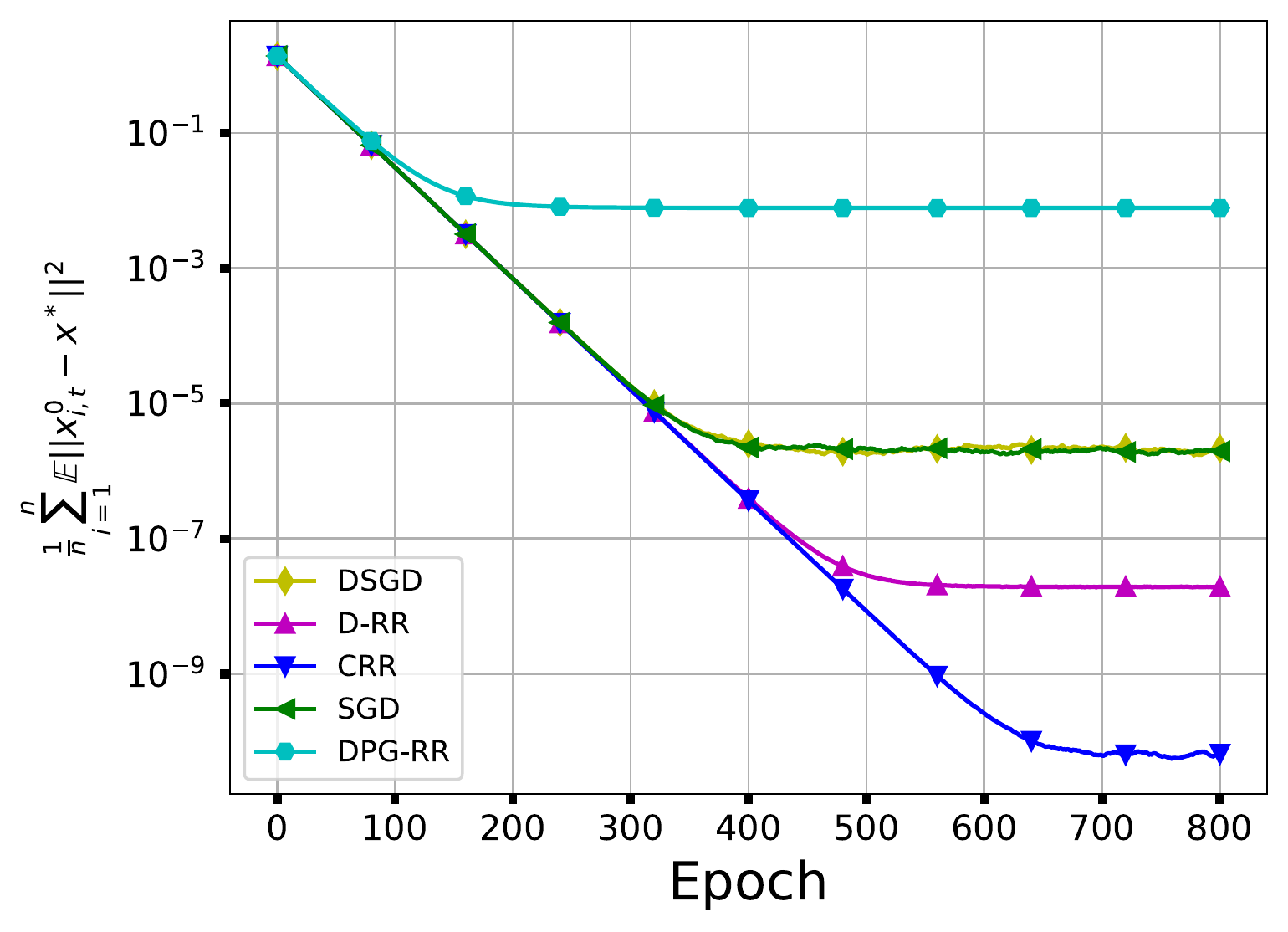}\label{fig:mnist_const_er}}
			\caption{{\sp Comparison among D-RR, DSGD, SGD, DPG-RR, and centralized RR for solving Problem \eqref{eq:logistic} on the MNIST dataset using constant stepsize. The stepsize is set as $1/8000$ for all the methods.}}
			\label{fig:mnist_const}
		\end{figure*}
		
		\begin{figure*}[htbp]
			\centering
			\subfloat[Grid graph, $n = 16$.]{\includegraphics[width = 0.33\textwidth]{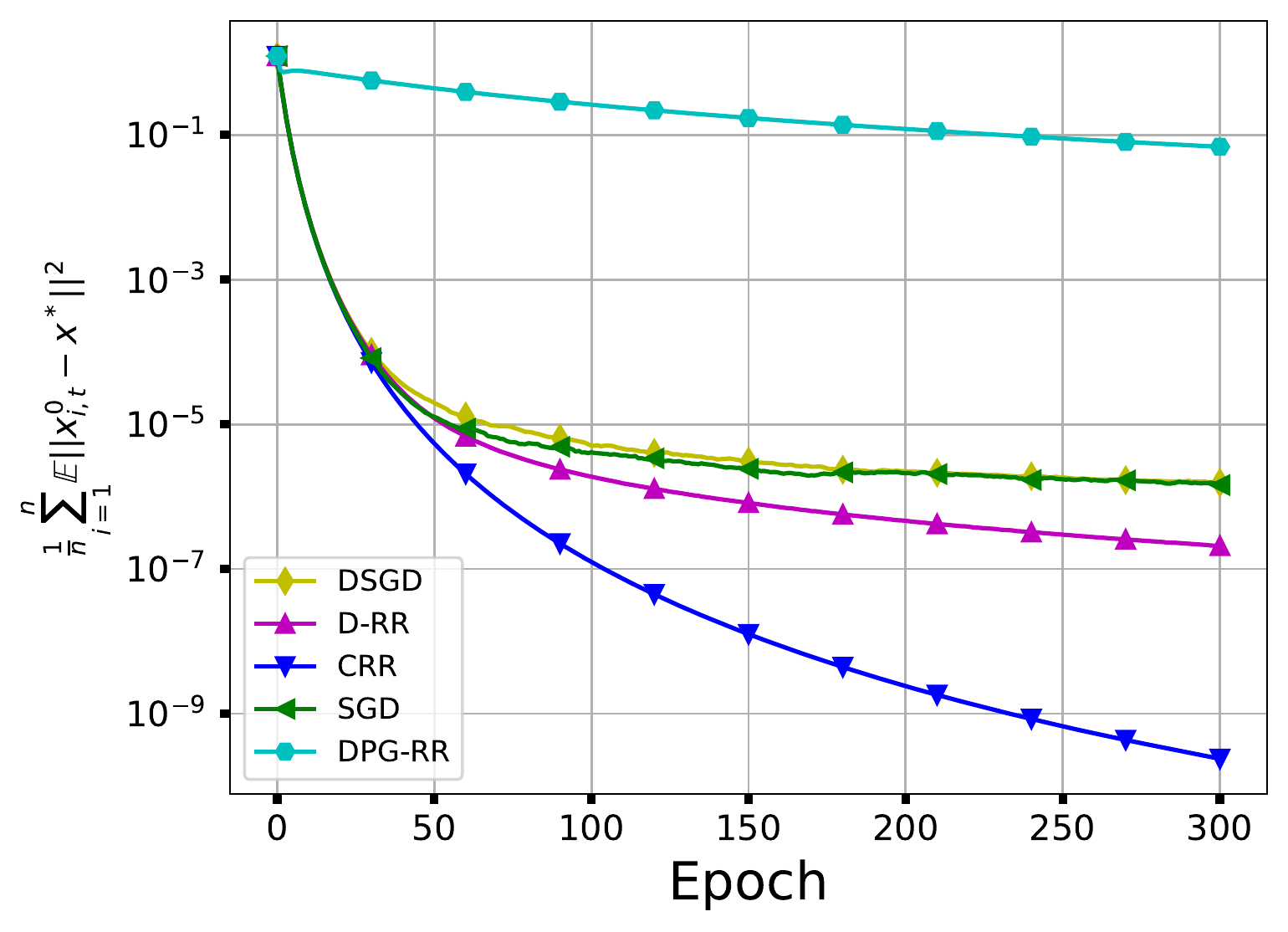}\label{fig:mnist_de_grid}}
			\subfloat[Exponential graph, $n = 16$.]{\includegraphics[width = 0.33\textwidth]{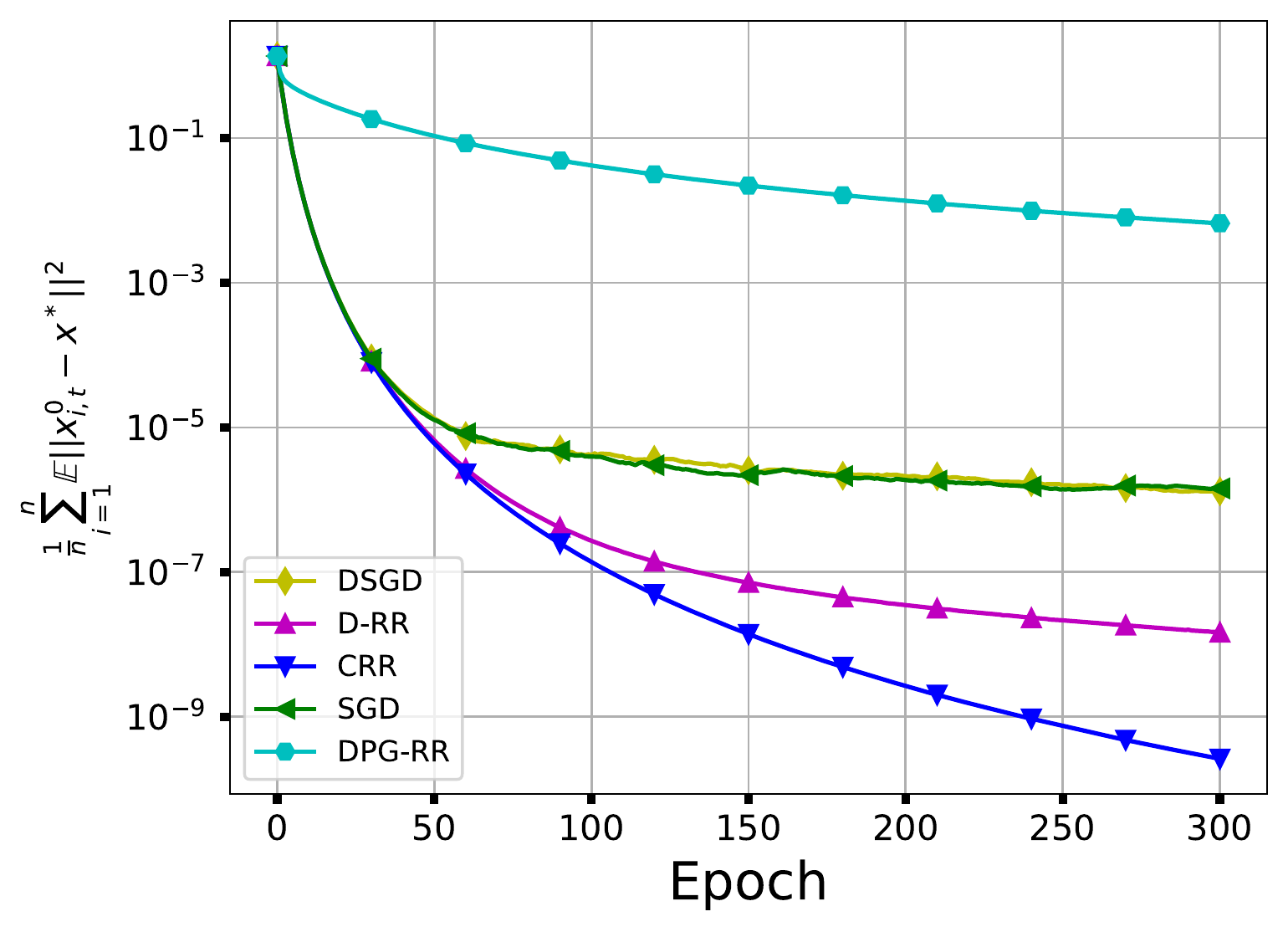}\label{fig:mnist_de_exp}}
			\subfloat[Erd\H{o}s-R\'enyi graph, $n = 16$.]{\includegraphics[width = 0.33\textwidth]{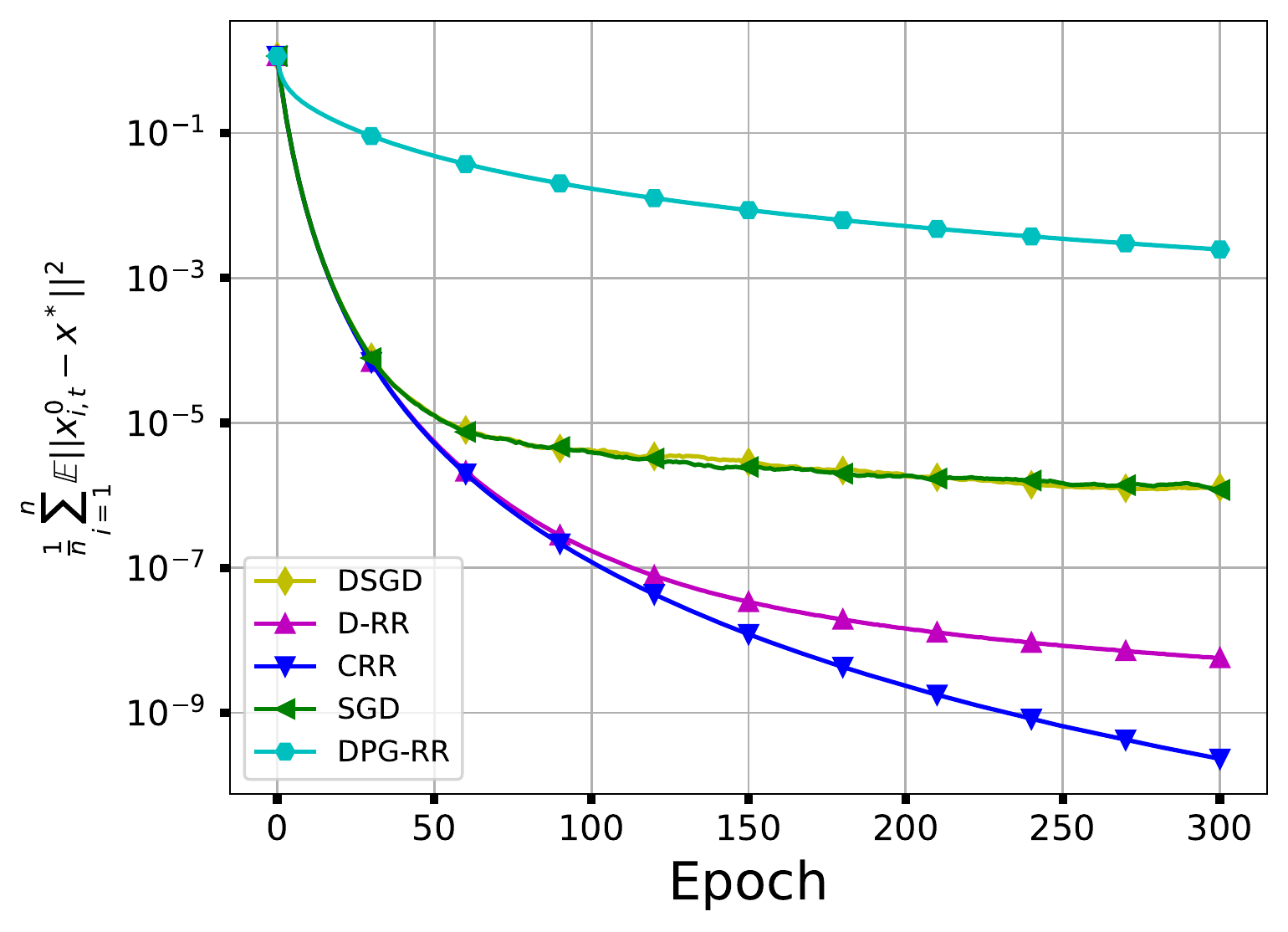}\label{fig:mnist_de_er}}
			\caption{{\sp Comparison among D-RR, DSGD, SGD, DPG-RR, and centralized RR for solving Problem \eqref{eq:logistic} on the MNIST dataset using decreasing stepsizes. The stepsize is set as $\alpha_t = 1/(50t + 400)$ for all the methods.}}
			\label{fig:mnist_de}
		\end{figure*}


		For both {\kh constant (Fig. \ref{fig:mnist_const}) and decreasing (Fig. \ref{fig:mnist_de}) stepsizes}, the errors decay at the same rate for all the algorithms during the starting epochs. After the starting epochs, DSGD and SGD achieve less accuracy compared to D-RR and C-RR. Comparing the two random reshuffling methods, the performance of D-RR is worse than C-RR since the convergence result of D-RR is affected by the connectivity of the graph topology. When the network topology becomes {\sp better-connected} ({\kh from left to right}), the performance of D-RR tends to be more comparable to that of C-RR. {\sp Regarding the stepsize policy, decreasing stepsizes are more favorable which allows larger stepsizes at the starting epochs.}
		

		{\sp We also compare the performance of {\kh D-RR and DPG-RR} with respect to the number of communication rounds for each node in Fig. \ref{fig:mnist_comm}. Both methods utilize the same stepsize and the underlying graph is a grid graph with $n=16$.} 
        {\sp Note that D-RR conducts one round of communication per gradient computation while DPG-RR only communicates epoch-wisely. However, as can be seen from Fig. \ref{fig:mnist_comm}, although DPG-RR saves communication cost and proceeds faster at the beginning, the error can not be controlled as well as in D-RR; see Remark \ref{rem:intui} for discussion.}

		\begin{figure}[htbp]
			\centering
			\includegraphics[width = 0.4\textwidth]{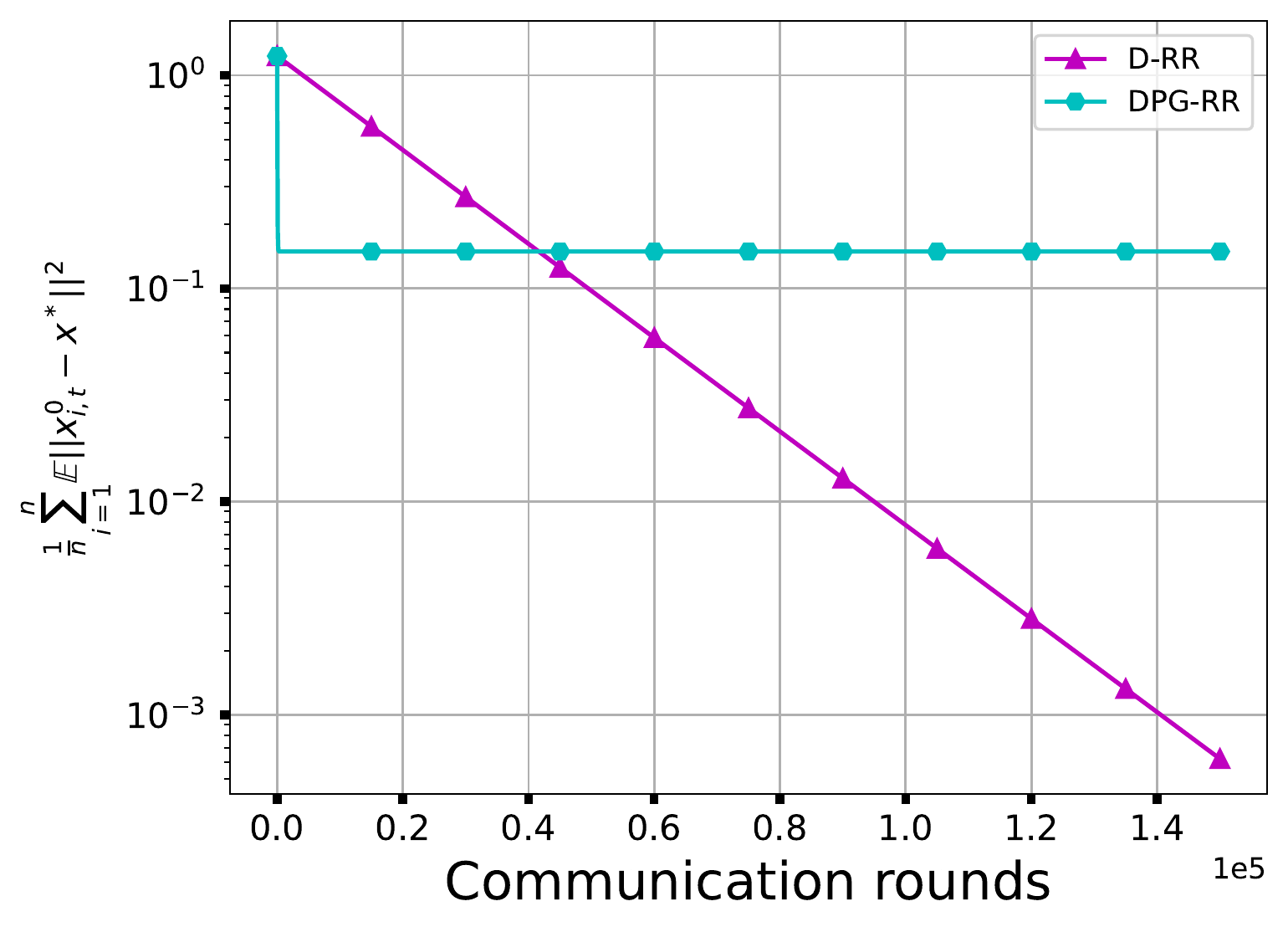}
			\caption{{\sp Comparison between D-RR and DPG-RR for solving Problem \eqref{eq:logistic} on the MNIST dataset with respect to the number of communication rounds. The stepsize is set as $1/8000$ for both methods, and the results are averaged over $2$ repeated runs.}}
			\label{fig:mnist_comm}
		\end{figure}

	\subsection{Nonconvex logistic regression}
	\label{sec:ncvx_logistic}

	Nonconvex regularizers are also widely used in statistical learning such as approximating sparsity. In this part, we consider a nonconvex binary classification problem \eqref{eq:ncvx_logistic} {\kh classifying airplanes and trucks in CIFAR-10 \cite{krizhevsky2009learning} dataset} and compare the proposed D-RR method (Algorithm \ref{alg:DGD-RR}) with DSGD, SGD, and centralized RR (Algorithm \ref{alg:GD-RR}) over {\sp a grid graph, an exponential graph, and an Erd\H{o}s-R\'enyi graph,} respectively. 
	{The optimization problem is}
	\begin{subequations}
		\label{eq:ncvx_logistic}
		\begin{align}
			& \min_{x\in\R^{p}} f(x) = \frac{1}{n}\sum_{i=1}^n f_i(x),\\
			& f_i(x) := \frac{1}{|\mathcal{S}_i|} \sum_{j\in\mathcal{S}_i} \log\left[1 + \exp(-x^{\T}u_jv_j)\right] + \frac{\eta}{2}\sum_{q=1}^p \frac{x_q^2}{1 + x_q^2}.
		\end{align}
	\end{subequations}
	Here, $x_q$ denotes the $q-$th element of $x\in\R^p$. We choose $\eta = 0.2$ and use constant stepsize for all the epochs. All the methods use the same initialization {\sp with the same stepsize}.
	
	From Fig. \ref{fig:cifar10_ncvx}, we also observe that the performance of the two random reshuffling methods outperform DSGD and SGD and achieve higher accuracy after the starting epochs. {\sp By comparing the performance from Fig. \ref{fig:ncvx_grid} to Fig. \ref{fig:ncvx_er} where the spectral gap increases, we can infer that D-RR performs better when the spectral gap becomes larger (i.e., the graph connectivity becomes better).}

	\begin{figure*}[htbp]
		\centering
		\subfloat[Grid graph, $n = 16$.]{\includegraphics[width = 0.33\textwidth]{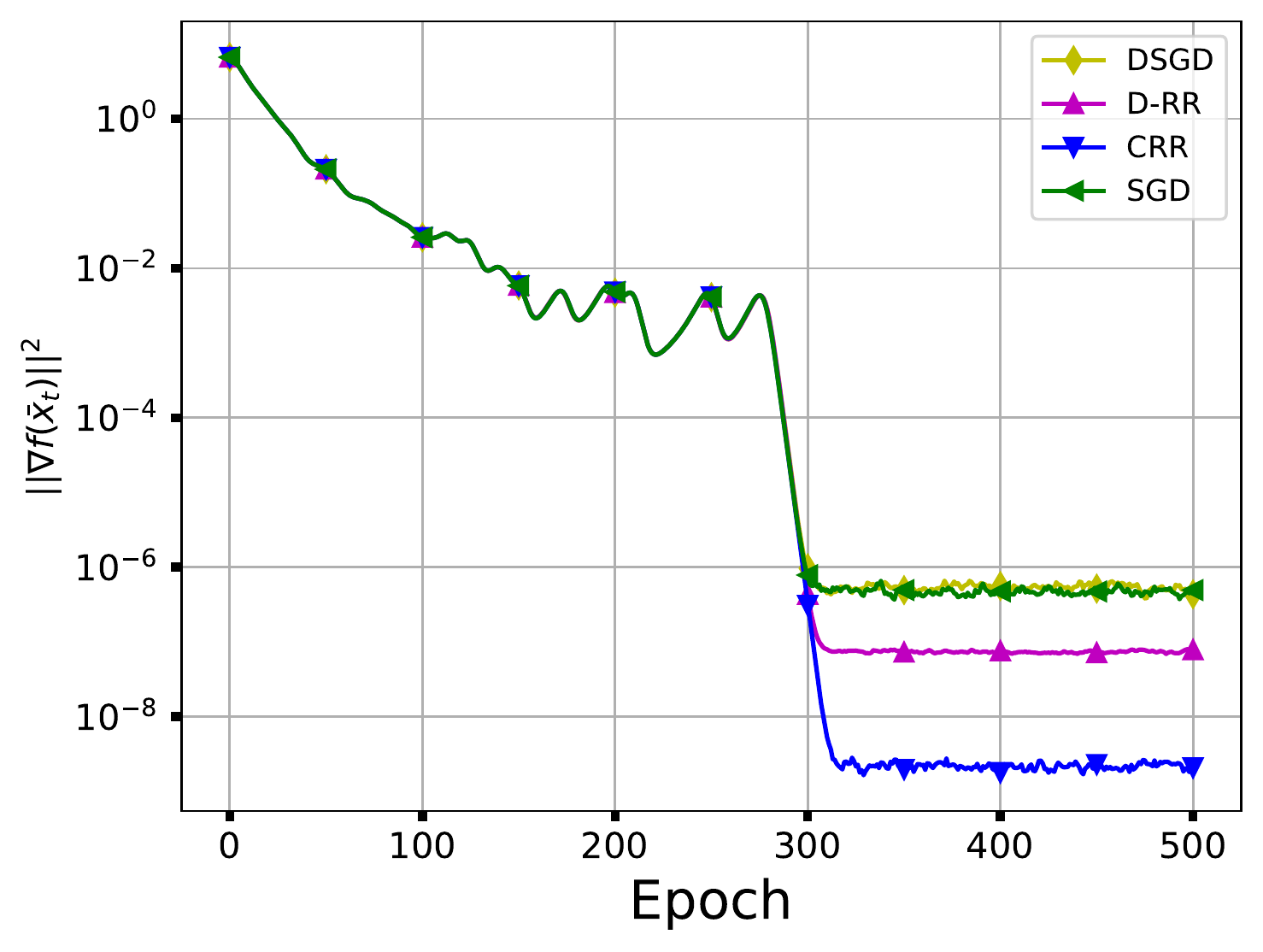}\label{fig:ncvx_grid}}
		\subfloat[Exponential graph, $n = 16$.]{\includegraphics[width = 0.33\textwidth]{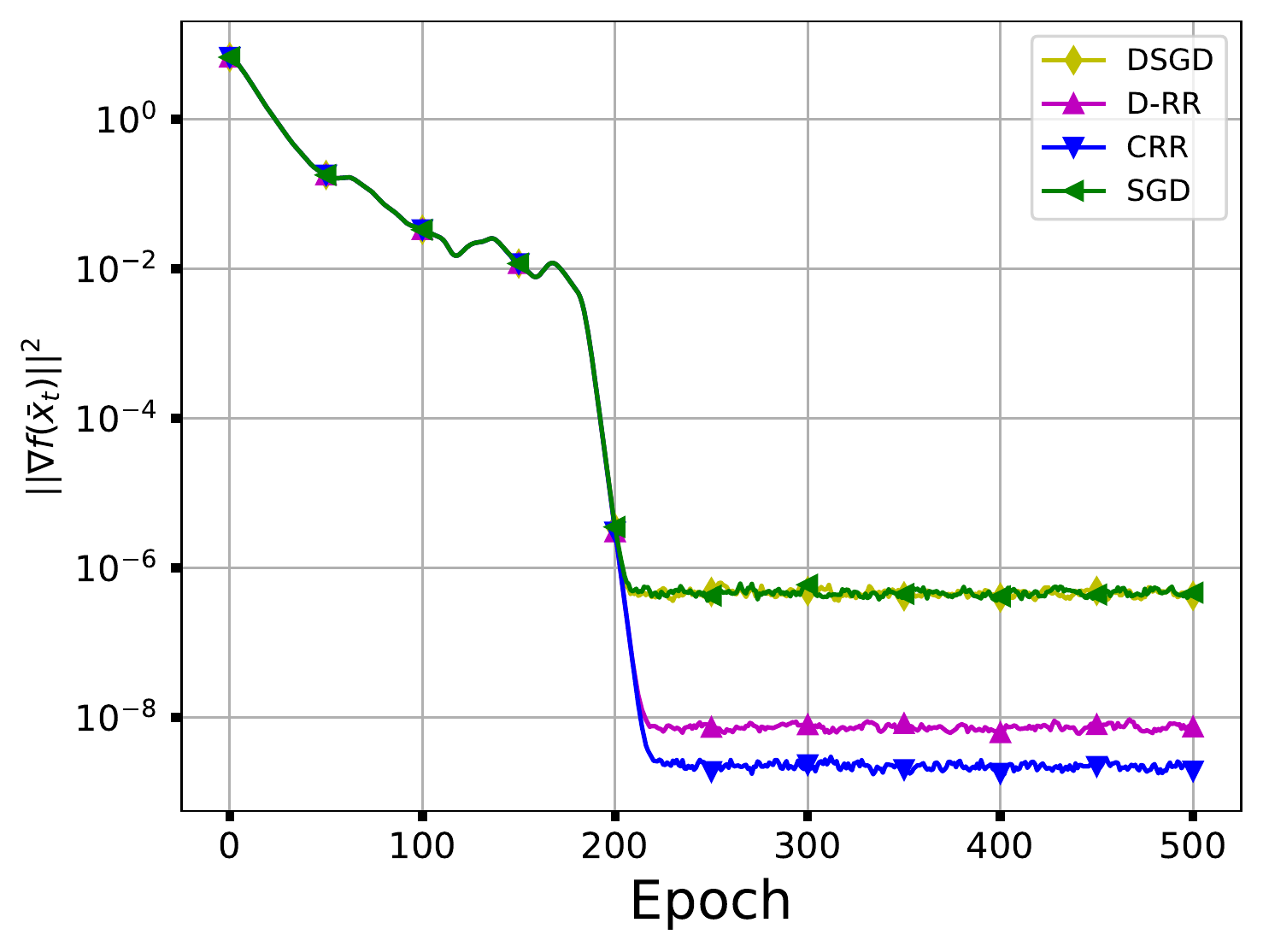}\label{fig:ncvx_exp}}
		\subfloat[Erd\H{o}s-R\'enyi graph, $n = 16$.]{\includegraphics[width = 0.33\textwidth]{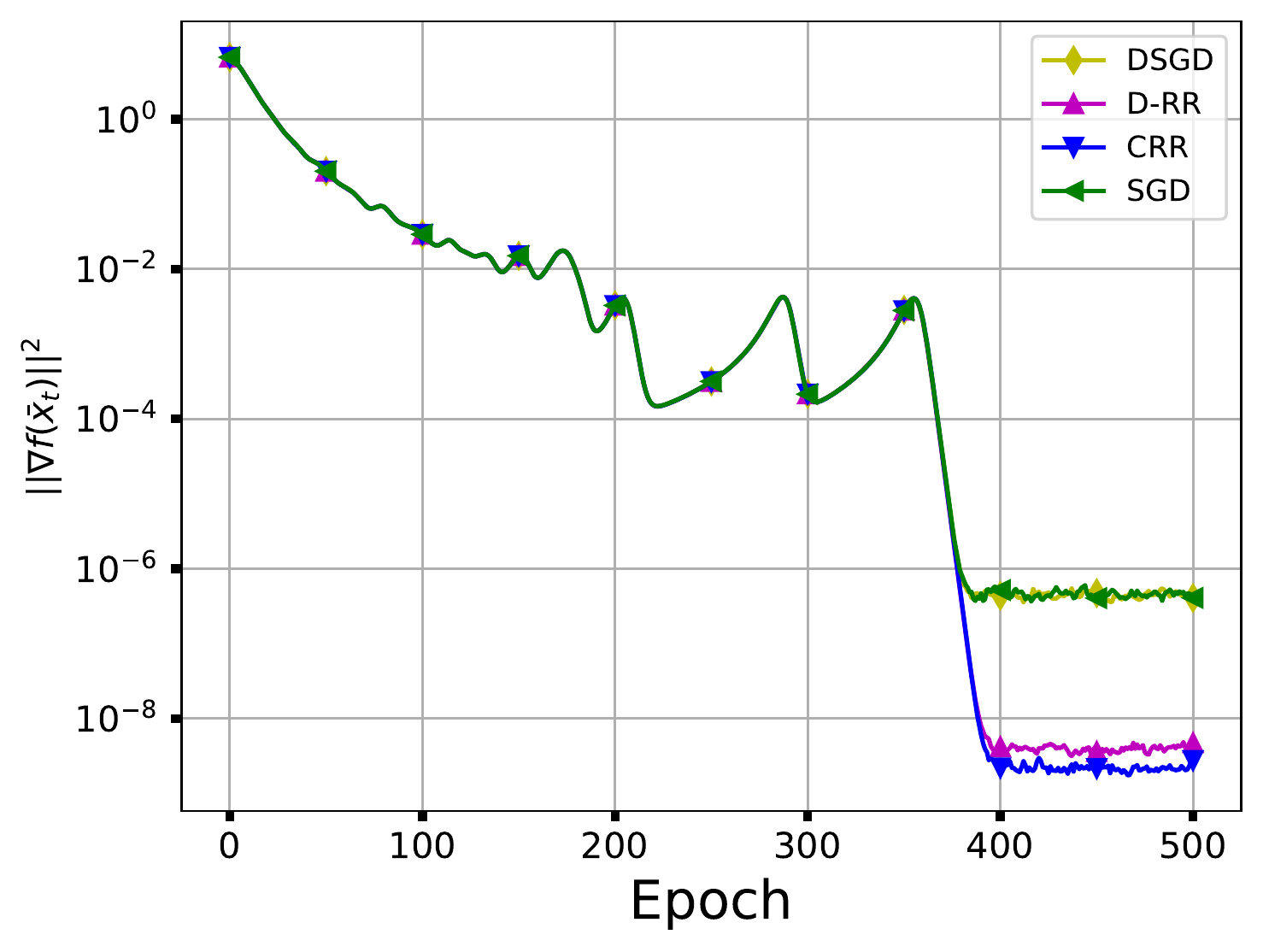}\label{fig:ncvx_er}}
		\caption{{\kh Comparison among D-RR, DSGD, SGD, and centralized RR for solving Problem \eqref{eq:ncvx_logistic} on the CIFAR-10 dataset using constant stepsize. The stepsize is set as $1/550$ for all the methods.}}
		\label{fig:cifar10_ncvx}
	\end{figure*}

 \begin{remark}
     Note that for both problems above, the performance gap between D-RR and DSGD becomes obvious only when the optimization errors are small, which may lead to similar testing performance for the two algorithms. Thus it is of future interest to further explore the conditions under which D-RR outperforms DSGD in the testing accuracy, especially for training large-scale machine learning models.
 \end{remark}
		
		\section{Conclusions}
		\label{sec:conclusions}
This paper is concerned with solving the distributed optimization problem over networked agents. Inspired by the classical distributed gradient descent (DGD) method and Random Reshuffling (RR), we propose a distributed random reshuffling (D-RR) algorithm and show the convergence results of D-RR  match those of centralized RR (up to constant factors) for both smooth strongly convex  and smooth nonconvex objective functions.

		\appendices
		\section{Parts of Proofs for the Strongly-Convex Case}
		
		\subsection{Proof of Lemma \ref{lem:xbar0}}
		\label{app:lem_xbar0}
		\begin{proof}
			As discussed in Remark \ref{rem:intui}, $\frac{1}{n}\sum_{i=1}^n\nabla f_{i,\pi_\ell^i}(x_{i,t}^\ell)$ is an approximation of $\frac{1}{n}\sum_{i=1}^n\nabla f_{i,\pi_\ell^i}(\bar{x}_t^{\ell})$, hence, the  {core difference} between our analysis and  {the one} in \cite{mishchenko2020random} mainly lies in this approximation. According to \eqref{eq:limit_avg}, we have 
			\begin{equation*}
				\bxs{\ell + 1} = \bxs{\ell} -  {\frac{\alpha_t }{n}}\sumn\nabla \fp{i}{\ell}(x^*).
			\end{equation*}
			
		     {This yields}
			\small
				\begin{align}
					& \hspace{.5ex} \E\brkn{\normn{\bx{t}{\ell + 1} - \bxs{\ell + 1}}^2} \nonumber \\
					& = \E\brk{\norm{ \bx{t}{\ell} - \bxs{\ell} -  {\prt{\frac{\alpha_t}{n}\sumn \brk{\nabla \fp{i}{\ell}(\xitl) -  \nabla \fp{i}{\ell}(x^*)}}}}^2}\nonumber\\
				&=\E\brk{\normn{\bx{t}{\ell} - \bxs{\ell}}^2 + \alpha_t^2\norm{\frac{1}{n}\sumn\brk{\nabla \fp{i}{\ell}(\xitl) - \nabla \fp{i}{\ell}(x^*)}}^2\right.\nonumber\\
					&\quad \left. - 2\alpha_t\inpro{\bx{t}{\ell} - \bxs{\ell}, \frac{1}{n}\sumn \brk{\nabla \fp{i}{\ell}(\xitl) - \nabla \fp{i}{\ell}(x^*)}}}. \label{eq:xbar_s1}
				\end{align}
			\normalsize
			
			 {We now divide the the inner product in \eqref{eq:xbar_s1} into two parts}: 
			\begin{align*}
				&\quad \inpro{\bx{t}{\ell} - \bxs{\ell}, \frac{1}{n}\sumn \nabla \fp{i}{\ell}(\xitl) - \frac{1}{n}\sumn\nabla \fp{i}{\ell}(x^*)}\\
				&= \underbrace{\inpro{\bx{t}{\ell} - \bxs{\ell}, \frac{1}{n}\sumn \nabla \fp{i}{\ell}(\bx{t}{\ell}) - \frac{1}{n}\sumn\nabla \fp{i}{\ell}(x^*)}}_{A}\\
				&\quad + \underbrace{\inpro{\bx{t}{\ell} - \bxs{\ell}, \frac{1}{n}\sumn \nabla \fp{i}{\ell}(\xitl) - \frac{1}{n}\sumn\nabla \fp{i}{\ell}(\bx{t}{\ell})}}_{B}.
			\end{align*}
			
			
			 Introducing $\bar{s}_\ell:= \frac{1}{n}\sumn f_{i, \pi^i_\ell}$ and recalling $D_{\bar{s}_\ell}(y,x) = \bar{s}_\ell(y) - \bar{s}_\ell(x) - \inpro{\nabla \bar{s}_\ell(x), y-x}$, we have
			\begin{align*}
				A &= D_{\bar{s}_\ell}(\bx{*}{\ell}, \bx{t}{\ell}) + D_{\bar{s}_\ell}(\bx{t}{\ell}, x^*) - D_{\bar{s}_\ell}(\bx{*}{\ell},x^*).
			\end{align*}
			
			According to Assumption \ref{ass:fij}, $\bar{s}_\ell = \frac{1}{n}\sumn\fp{i}{\ell}$ is also $\mu-$strongly convex and $L-$smooth,  {Thus, applying} \eqref{eq:muL} and \eqref{eq:L}, we  {obtain}
			\begin{equation}
				\label{eq:cmu}
				\begin{aligned}
					\frac{\mu}{2}\norm{\bxs{\ell} - \bx{t}{\ell}}^2 &\leq D_{\bar{s}_\ell}(\bxs{\ell} ,\bx{t}{\ell}),
				\end{aligned}
			\end{equation}
			\begin{equation}
				\label{eq:cL}
				\begin{aligned}
					& {\frac{1}{2L}\norm{\frac{1}{n}\sumn[\nabla \fp{i}{\ell}(x^*) - \fp{i}{\ell}(\bx{t}{\ell})]}^2}
					 \leq D_{\bar{s}_\ell}(\bx{t}{\ell}, x^*).
				\end{aligned}
			\end{equation}
			
			The last  {term} in $A$ can be bounded by shuffling variance $\svar$  {introduced in Definition \ref{def:svar}.} 
			%
			{Note that  {the} definition  {of} $\svar$ is different from that in \cite{mishchenko2020random}  {and it does not include the factor ${1}/{\alpha_t}$ (since we use decreasing stepsizes).}} We have 
			\begin{equation}
				\label{eq:bound_svar}
				 {\E\brkn{D_{\bar{s}_\ell}(\bxs{\ell}, x^*)}} \leq  \svar.
			\end{equation}
			
			For $B$, we apply Cauchy's inequality, Young's inequality and invoke \eqref{eq:muL},
			\begin{equation}
				\label{eq:bound_B}
				|B|\leq  {\frac{c}{2}}\norm{\bx{t}{\ell} - \bxs{\ell}}^2 +  {\frac{L^2}{2nc}}\sumn\norm{\xitl - \bx{t}{\ell}}^2 \quad \forall c>0.
			\end{equation}
			
			Next, we bound the gradient term in \eqref{eq:xbar_s1}. 
			\small
			\begin{align*}
				&\frac{1}{2}\norm{\frac{1}{n}\sumn \nabla \fp{i}{\ell}(\xitl) - \frac{1}{n}\sumn\nabla \fp{i}{\ell}(x^*)}^2\\
				&\leq \norm{\frac{1}{n}\sumn \nabla \fp{i}{\ell}(\xitl) - \frac{1}{n}\sumn \nabla \fp{i}{\ell}(\bx{t}{\ell})}^2\\
				&\quad + \norm{\frac{1}{n}\sumn \nabla \fp{i}{\ell}(\bx{t}{\ell}) - \frac{1}{n}\sumn\nabla \fp{i}{\ell}(x^*)}^2\\
				&\leq \frac{L^2}{n}\sumn\norm{\xitl - \bx{t}{\ell}}^2+ \norm{\frac{1}{n}\sumn \prt{\nabla \fp{i}{\ell}(\bx{t}{\ell}) - \nabla \fp{i}{\ell}(x^*)}}^2.
			\end{align*}\normalsize
			
			The second term would get absorbed combining \eqref{eq:cL} when the stepsize is small, i.e., $\alpha_t \leq {1}/{(2L)}$. Finally, choosing $c = {\mu}/{ {2}}$ in \eqref{eq:bound_B}, we obtain the result.
		\end{proof}
		
		\subsection{Proof of Lemma \ref{lem:cons0}}
		\label{app:lem_cons0}
		\begin{proof}
			Lemma \ref{lem:Fell} first bounds a specific term in our derivation for Lemma \ref{lem:cons0}.
			\begin{lemma}
				\label{lem:Fell}
				We have
				\begin{align*}
					& \E\brkn{\normn{\nabla \Fp{\ell}(\x_t^{\ell})}^2} \leq 6L^2 n \E\brkn{\normn{\bx{t}{\ell} - \bxs{\ell}}^2}\\
					&\quad  + 6L^2\E\brkn{\normn{\x_t^{\ell} - \Bxt{\ell}}^2} + 3n\sigma^2_*  + 6nL\svar.
				\end{align*}
			\end{lemma}

\begin{proof}
     {It holds that} 
    \begin{align}
        &\E\brkn{\normn{\nabla \Fp{\ell}(\x_t^{\ell})}^2} \leq 3\E\brkn{\normn{\nabla \Fp{\ell}(\x_t^{\ell}) - \nabla \Fp{\ell}(\1(\bxs{\ell})^{\T})}^2}\nonumber\\
        & + 3\E\brkn{\normn{\nabla \Fp{\ell}(\1(x^{*})^{\T})}^2}\nonumber\\
        & + 3\E\brkn{\normn{\nabla \Fp{\ell}(\1(\bxs{\ell})^{\T}) - \nabla \Fp{\ell}(\1(x^{*})^{\T})}}\nonumber\\
        &\leq 6L^2 n \E\brkn{\normn{\bx{t}{\ell} - \bxs{\ell}}^2} + 6L^2{\sum}_{i=1}^n\E\brkn{\normn{\xitl - \bx{t}{\ell}}^2}\nonumber\\
        &\quad + 3{\sum}_{i=1}^n \E\brkn{\normn{\nabla \fp{i}{\ell}(x^*)}^2}\nonumber\\
        &\quad + 3\E\brkn{\normn{\nabla \Fp{\ell}(\1(\bxs{\ell})^{\T}) - \nabla \Fp{\ell}(\1(x^{*})^{\T})}^2}\label{eq:fell1}
    \end{align}
    
    We use $\svar$ to bound the last term in \eqref{eq:fell1}. From \eqref{eq:L}: 
    \small
    \begin{align*}
        &\hspace{0.5ex}\E\brkn{\normn{\nabla \Fp{\ell}(\1(\bxs{\ell})^{\T}) - \nabla \Fp{\ell}(\1(x^{*})^{\T})}^2}\\
        & = \E\brk{\sumn\normn{\nabla \fp{i}{\ell}(\bxs{\ell}) - \nabla \fp{i}{\ell}(x^*)}^2}\\
        &\leq \E\brk{2L\sumn (\fp{i}{\ell}(\bxs{\ell}) - \fp{i}{\ell}(x^*) - \langle{\nabla \fp{i}{\ell}(x^*), \bxs{\ell} - x^*}\rangle)}\\
        &= 2nL \E\brkn{D_{\bar{s}_\ell}(\bxs{\ell}, x^*)}
        \leq 2nL\svar.
    \end{align*}\normalsize
    
     {Next, we} bound $\sumn \E\brkn{\normn{\nabla \fp{i}{\ell}(x^*)}^2}$ using $\sigma^2_*$: 
    \begin{align*}
        \sumn \E\brkn{\normn{\nabla \fp{i}{\ell}(x^*)}^2} &= \sumn\frac{1}{m}\sum_{j=1}^m\norm{\nabla f_{i,j}(x^*)}^2 = n\sigma^2_*.
    \end{align*}
    
     {Combining the last} two steps and \eqref{eq:fell1} finishes the proof of Lemma \ref{lem:Fell}. 
\end{proof}
			
			With the help of Lemma \ref{lem:Fell}, we prove Lemma \ref{lem:cons0}.  {Let us set} $\bar{\nabla}\Fp{\ell}(\x_t^{\ell}):= \frac{1}{n}\sumn \nabla \fp{i}{\ell}(\xitl)$. By Lemma \ref{lem:rhow}, we have 
			\begin{align*}
				& \E\brkn{\normn{\x_t^{\ell + 1} - \Bxt{\ell + 1}}^2} \\
				&\leq \rho_w^2 \E\brk{\normn{\x_t^{\ell} - \Bxt{\ell}}^2 + \alpha_t^2\normn{\nabla\Fp{\ell}(\x_t^{\ell}) - \BFp{\ell}}^2\right.\\
					&\quad\left. -2\alpha_t\inpro{\x_t^{\ell} - \Bxt{\ell}, \nabla \Fp{\ell}(\x_t^{\ell}) - \BFp{\ell}}}\\
				&\leq \rho_w^2(1 + c)\E\brkn{\normn{\x_t^{\ell} - \Bxt{\ell}}^2}\\
				&\quad + \alpha_t^2\rho_w^2(1+c^{-1})\E\brkn{\normn{\nabla \Fp{\ell}(\x_t^{\ell})}^2}.
			\end{align*}

			The last  {step is due to} Cauchy's and Young's inequality  {and holds} for any $c>0$.  {Invoking} Lemma \ref{lem:Fell}, we obtain
			\begin{align*}
				& \E\brkn{\normn{\x_t^{\ell + 1} - \Bxt{\ell + 1}}^2}\\
				&\leq \rho_w^2\brk{(1 + c) + 6\alpha_t^2 L^2(1 + c^{-1})}\E\brkn{\normn{\x_t^{\ell} - \Bxt{\ell}}^2}\\
				&\quad + 3n\rho_w^2\alpha_t^2(1 + c^{-1})\prt{\sigma^2_* + 2L\svar} \\
				&\quad + 6\alpha_t^2nL^2\rho_w^2(1 + c^{-1})\E\brkn{\normn{\bx{t}{\ell} - \bxs{\ell}}^2}
			\end{align*}
			
			 {In order to guarantee a contractive behavior, we set} $c = {(1 - \rho_w^2)}/{4}$, then  {we have} $1 + c^{-1} \leq {5}/{(1-\rho_w^2)}$. 
			 {In the case} ${\ \alpha_t \leq \sqrt{\frac{2-\rho_w^2}{24\rho_w^2(5-\rho_w^2)}}\frac{1-\rho_w^2}{L}}$, we obtain the  {desired} result.
		\end{proof}

  \subsection{Proof of Lemma \ref{lem:lya}}
\label{app:lem_lya}
\begin{proof}
    \textbf{Step 1: Obtain a combined recursion $H_t^{\ell}$}.
    Combining Lemmas \ref{lem:xbar0} and \ref{lem:cons0}, we obtain
    \begin{align}
        &H_t^{\ell + 1} \leq \brk{\prt{1-\frac{\alpha_t\mu}{2}} + \frac{30n\alpha_t^2 L^2}{1-\rho_w^2}\omega_t}\E\brk{\norm{\bx{t}{\ell} - \bxs{\ell}}^2}\nonumber\\
        &\quad + \brk{\frac{2\alpha_t L^2}{n}\prt{\frac{1}{\mu} + \alpha_t} + \frac{1+\rho_w^2}{2}\omega_t}\E\brk{\norm{\x_t^{\ell} - \Bxt{\ell}}^2}\nonumber\\
        &\quad + 2\alpha_t\svar\prt{1 + \frac{15\alpha_tnL\rho_w^2}{1-\rho_w^2}\omega_t} + \frac{15n\rho_w^2\alpha_t^2\sigma_*^2}{1-\rho_w^2}\omega_t\label{eq:At_s1}
    \end{align}
    
    $\omega_t$ is chosen so that the following inequalities hold for all $t, \ell$, 
    \begin{subequations}
        \label{eq:wt_in}
        \begin{align}
            &\quad \prt{1-\frac{\alpha_t\mu}{2}} + \frac{30n\alpha_t^2 L^2}{1-\rho_w^2}\omega_t \leq 1 - \frac{\alpha_t\mu}{4}\label{eq:wt1}\\
            &\quad \frac{2\alpha_t L^2}{n}\prt{\frac{1}{\mu} + \alpha_t} + \frac{1+\rho_w^2}{2}\omega_t \leq \prt{1 - \frac{\alpha_t\mu}{4}}\omega_t\label{eq:wt2}
        \end{align}
    \end{subequations}
    
    We verify the choice of $\omega_t$ in \eqref{eq:wt}. Firstly, \eqref{eq:wt2} is equivalent to 
    \begin{equation}
        \label{eq:wt2_s1}
        \prt{\frac{1-\rho_w^2}{2} - \frac{\alpha_t\mu}{4}}\omega_t \geq \frac{2\alpha_t L^2}{n}\prt{\frac{1}{\mu} + \alpha_t}.
    \end{equation} 
    
    Noting $\alpha_t \leq \frac{1-\rho_w^2}{{   2}\mu}\leq \frac{1}{\mu}$, we have 
    \begin{align*}
        &\frac{1-\rho_w^2}{2} - \frac{\alpha_t\mu}{4} \geq \frac{1-\rho_w^2}{2} - \frac{1-\rho_w^2}{8} = \frac{3(1-\rho_w^2)}{8}, \\
        &\frac{2\alpha_t L^2}{n}\prt{\frac{1}{\mu} + \alpha_t} \leq \frac{4\alpha_t L^2}{n\mu}.
    \end{align*}
    
    Thus, it is sufficient for $\omega_t \geq \frac{16\alpha_t L^2}{n\mu(1-\rho_w^2)}$ to satisfy \eqref{eq:wt2_s1}. 
    
    Secondly, \eqref{eq:wt1} requires $\omega_t \leq \frac{(1-\rho_w^2)\mu}{120n L^2}\frac{1}{\alpha_t}$ or $\frac{16\alpha_t L^2}{n\mu(1-\rho_w^2)}\leq \frac{(1-\rho_w^2)\mu}{120n L^2}\frac{1}{\alpha_t}.$ It is sufficient that $\alpha_t \leq \frac{(1-\rho_w^2)\mu}{8\sqrt{30} L^2}$. We thus obtain a recursion for $H_t^{\ell}$ according to \eqref{eq:wt_in} and \eqref{eq:At_s1}: 
    \begin{equation}
        \label{eq:At_ell}
        \begin{aligned}
            H_t^{\ell + 1} &\leq \prt{1 - \frac{\alpha_t \mu}{4}} H_t^{\ell} + 2\alpha_t\svar\prt{1 + \frac{240\alpha_t^2\rho_w^2L^3}{\mu(1-\rho_w^2)^2}}\\
            & + \frac{240\alpha_t^3\rho_w^2 L^2}{\mu (1-\rho_w^2)^2}\sigma_*^2
        \end{aligned}
    \end{equation}
    
    \textbf{Step 2: Relate $H_t^{\ell}$ with the outer loop.} 
    Unroll \eqref{eq:At_ell} with respect to $\ell$ and notice $\alpha_t$ is unchanged for $\ell \geq 0$, we obtain, 
    \begin{align}
        \label{eq:Atm}
        H_t^m &\leq \prt{1 - \frac{\alpha_t\mu}{4}}^m H_t^0 + 2\brk{\alpha_t\svar\prt{1 + \frac{240\alpha_t^2\rho_w^2L^3}{\mu(1-\rho_w^2)^2}} \right.\nonumber\\
        &\left. + \frac{120\alpha_t^3\rho_w^2 L^2}{\mu (1-\rho_w^2)^2}\sigma_*^2}\brk{\sum_{k=0}^{m - 1}\prt{1 - \frac{\alpha_t\mu}{4}}^k}.
    \end{align}
    
    Note from Algorithm \ref{alg:DGD-RR} and \eqref{eq:limit_avg}, we have the following facts:
    \begin{equation*}
        x_{i,t}^m = x_{i, t+1},\quad x_{i,t}^0 = x_{i,t} = x_{i, t - 1}^m, \quad \bxs{0} = x^* = \bxs{m}.
    \end{equation*}
    
    Therefore, 
    \begin{equation}
        \label{eq:in_out}
        \begin{aligned}
            &\quad \bx{t}{m} - \bxs{m} = \bar{x}_{t + 1} - x^*, \quad \bx{t}{0} - \bxs{0} = \bar{x}_t - x^*,\\
            &\quad \x_t^m - \Bxt{m} = \x_{t + 1} - \1\bar{x}_{t + 1}^{\T},\\
            & \x_t^0 - \Bxt{0} = \x_t - \1\bar{x}_t^{\T}.
        \end{aligned}
    \end{equation}
    
    We then use $H_{t + 1}$ to denote $H_t^{m}, \forall t\geq 0$, i.e., 
    \begin{align}
        \label{eq:At_outer}
        & H_{t+1} := \E\brk{\norm{\bar{x}_t^m - x^*}^2} + \omega_t \E\brk{\norm{\x_t^m - \Bxt{m}}^2}\nonumber\\
        &= \E\brk{\norm{\bar{x}_{t + 1} - x^*}^2} + \omega_t \E\brk{\norm{\x_{t + 1} - \1\bar{x}_{t + 1}^{\T}}^2}.
    \end{align}
    
    Substitute $H_t$ into \eqref{eq:Atm}, we obtain the recursion for the outer loop $t$ : 
    \begin{align*}
        H_{t + 1} &\leq \prt{1 - \frac{\alpha_t\mu}{4}}^m H_t + 2\brk{\alpha_t\svar\prt{1 + \frac{240\alpha_t^2\rho_w^2L^3}{\mu(1-\rho_w^2)^2}}\right.\\
        &\left. + \frac{120\alpha_t^3\rho_w^2 L^2}{\mu (1-\rho_w^2)^2}\sigma_*^2}\brk{\sum_{k=0}^{m - 1}\prt{1 - \frac{\alpha_t\mu}{4}}^k}.
    \end{align*}
    
    \textbf{A uniform bound for $H_t^{\ell}, \forall \ell.$} \eqref{eq:Atm} in Step 2 gives a uniform bound for $H_t^{\ell},\forall \ell$, 
    \begin{align*}
        H_t^{\ell} &\leq \prt{1 - \frac{\alpha_t\mu}{4}}^{\ell} H_t^0 + 2\brk{\alpha_t\svar\prt{1 + \frac{240\alpha_t^2\rho_w^2L^3}{\mu(1-\rho_w^2)^2}} \right.\\
        &\left. + \frac{120\alpha_t^3\rho_w^2 L^2}{\mu (1-\rho_w^2)^2}\sigma_*^2}\brk{\sum_{k=0}^{\ell - 1}\prt{1 - \frac{\alpha_t\mu}{4}}^k}.
    \end{align*}
\end{proof}
		
		\subsection{Proof of Lemma \ref{lem:cons1}}
		\label{app:lem_cons1}
		\begin{proof}

			{\kh 
			When we use decreasing stepsizes $\crk{\alpha_t}$, we choose $K$ such that $\alpha_t^2 \leq \frac{(1-\rho_w^2)^2\mu^2}{1920 L^4} \leq \frac{(1-\rho_w^2)^2\mu }{240\rho_w^2 L^3}$, i.e., $K^2\geq \frac{1920L^4\theta^2}{(1-\rho_w^2)^2\mu^4 m^2}$, then 
			\begin{align*}
				\frac{28\theta^2L\sigma^2_*}{m^2\mu^3}\prt{m + \frac{240\rho_w^2L}{\mu(1-\rho_w^2)^2}}\frac{1}{(t+ K)^2}
				&\leq \frac{28(m\mu + L)\sigma^2_*}{\mu L^2}.
			\end{align*}

			If we use a constant stepsize $\alpha_t = \alpha$, we have $\alpha \leq \frac{(1-\rho_w^2)^2\mu^2}{1920 L^4} \leq \frac{(1-\rho_w^2)^2\mu }{240\rho_w^2 L^3}$, then 
			\begin{align*}
				\frac{8\alpha^2\sigma^2_*}{\mu} \prt{mL + \frac{240\rho_w^2L^2}{\mu(1 - \rho_w^2)^2}} \sigma_*^2\leq \frac{8(m\mu + L)}{\mu L^2}.
			\end{align*}

			Therefore, we have an upper bound for the term $H_t^\ell$ for both decreasing and constant stepsizes, 
			\begin{align}
				\label{eq:opt0_uni}
				\E\brk{\norm{\bx{t}{\ell} - \bxs{\ell}}^2}&\leq H_t^\ell \leq H_0 + \frac{28(m\mu + L)\sigma^2_*}{\mu L^2} = \hat{X}_0.
			\end{align}
			}

			Substitute \eqref{eq:opt0_uni} into Lemma \ref{lem:cons0} and let 
			\begin{equation}
				\label{eq:hatX1}
				\hat{X}_1 := \frac{30 nL^2}{1-\rho_w^2}\hat{X}_0 + \frac{15 n\rho_w^2}{1-\rho_w^2}\sigma^2_* + \frac{mn\mu(1-\rho_w^2)}{8L}\sigma^2_*,
			\end{equation} 
			we have,
			\begin{align}
				& \E\brk{\norm{\x_t^{\ell + 1} - \Bxt{\ell + 1}}^2}\nonumber\\
				& \leq \frac{1 + \rho_w^2}{2}\E\brk{\norm{\x_t^{\ell} - \Bxt{\ell}}^2} + \alpha_t^2 \hat{X}_1\nonumber\\
				&\leq \prt{\frac{1+\rho_w^2}{2}}^{\ell + 1} \E\brk{\norm{\x_t^{0} - \Bxt{0}}^2} + \frac{\alpha_t^2 \hat{X}_1}{1-\rho_w^2}\label{eq:cons1_s1}
			\end{align}
			
			Next, we derive recursion among epochs. From \eqref{eq:cons1_s1} and \eqref{eq:in_out} {    in Supplementary Material}, we have 
			\begin{align*}
				&\E\brk{\norm{\x_{t + 1}^0 - \1(\bar{x}_{t + 1}^0)^{\T}}^2}\\
				& \leq \prt{\frac{1+\rho_w^2}{2}}^{m} \E\brk{\norm{\x_t^0 - \Bxt{0}}^2} + \frac{\alpha_t^2 \hat{X}_1}{1-\rho_w^2}.
			\end{align*}
			
			Therefore, 
			\begin{align*}
				& \E\brk{\norm{\x_t^0 - \Bxt{0}}^2}\leq \prt{\frac{1+\rho_w^2}{2}}^{mt}\E\brk{\norm{\x_0^0 - \1(\bar{x}_0^0)^{\T}}^2}\\
				&\quad + \frac{\hat{X}_1}{1-\rho_w^2}\sum_{k=0}^{t - 1}\prt{\frac{1 + \rho_w^2}{2}}^{m(t- 1- k)}\alpha_k^2
			\end{align*}
			
			By similar induction of those in \cite{huang2021improving}, we obtain
			\begin{align}
				\label{eq:ind1}
				\sum_{k=0}^{t - 1}\prt{\frac{1 + \rho_w^2}{2}}^{m(t- 1- k)}\alpha_k^2 \leq \frac{\alpha_t^2}{\frac{\alpha_{t + 1}^2}{\alpha_t^2} - \prt{\frac{1 + \rho_w^2}{2}}^m}
			\end{align}
			
			Invoking the choice of $\alpha_t$ and $K\geq \frac{24}{1 - \rho_w^2}$, we obtain, 
			\begin{align*}
				\frac{\alpha_{t + 1}^2}{\alpha_t^2} - \prt{\frac{1 + \rho_w^2}{2}}^m &= \prt{1 - \frac{1}{t + K + 1}}^2 - \prt{\frac{1+\rho_w^2}{2}}^m\\
				&\geq \prt{1 - \frac{1}{K}}^2 - \prt{\frac{1+\rho_w^2}{2}}\geq \frac{1 - \rho_w^2}{4}.
			\end{align*}
			
			Combing the above leads to the result.
		\end{proof}

\bibliographystyle{IEEEtran}
\bibliography{references_all}

\begin{thebibliography}{10}
\providecommand{\url}[1]{#1}
\csname url@samestyle\endcsname
\providecommand{\newblock}{\relax}
\providecommand{\bibinfo}[2]{#2}
\providecommand{\BIBentrySTDinterwordspacing}{\spaceskip=0pt\relax}
\providecommand{\BIBentryALTinterwordstretchfactor}{4}
\providecommand{\BIBentryALTinterwordspacing}{\spaceskip=\fontdimen2\font plus
\BIBentryALTinterwordstretchfactor\fontdimen3\font minus
  \fontdimen4\font\relax}
\providecommand{\BIBforeignlanguage}[2]{{%
\expandafter\ifx\csname l@#1\endcsname\relax
\typeout{** WARNING: IEEEtran.bst: No hyphenation pattern has been}%
\typeout{** loaded for the language `#1'. Using the pattern for}%
\typeout{** the default language instead.}%
\else
\language=\csname l@#1\endcsname
\fi
#2}}
\providecommand{\BIBdecl}{\relax}
\BIBdecl

\bibitem{nedic2009distributed}
A.~Nedic and A.~Ozdaglar, ``Distributed subgradient methods for multi-agent
  optimization,'' \emph{IEEE Transactions on Automatic Control}, vol.~54,
  no.~1, pp. 48--61, 2009.

\bibitem{assran2018stochastic}
M.~Assran, N.~Loizou, N.~Ballas, and M.~Rabbat, ``Stochastic gradient push for
  distributed deep learning,'' \emph{arXiv preprint arXiv:1811.10792}, 2018.

\bibitem{lu2021optimal}
Y.~Lu and C.~De~Sa, ``Optimal complexity in decentralized training,'' in
  \emph{International Conference on Machine Learning}.\hskip 1em plus 0.5em
  minus 0.4em\relax PMLR, 2021, pp. 7111--7123.

\bibitem{nedic2018network}
A.~Nedi{\'c}, A.~Olshevsky, and M.~G. Rabbat, ``Network topology and
  communication-computation tradeoffs in decentralized optimization,''
  \emph{Proceedings of the IEEE}, vol. 106, no.~5, pp. 953--976, 2018.

\bibitem{pu2021distributed}
S.~Pu and A.~Nedi{\'c}, ``Distributed stochastic gradient tracking methods,''
  \emph{Mathematical Programming}, vol. 187, no.~1, pp. 409--457, 2021.

\bibitem{pu2021sharp}
S.~Pu, A.~Olshevsky, and I.~C. Paschalidis, ``A sharp estimate on the transient
  time of distributed stochastic gradient descent,'' \emph{IEEE Transactions on
  Automatic Control}, 2021.

\bibitem{tang2018d}
H.~Tang, X.~Lian, M.~Yan, C.~Zhang, and J.~Liu, ``D2: Decentralized training
  over decentralized data,'' in \emph{International Conference on Machine
  Learning}, 2018, pp. 4848--4856.

\bibitem{yuan2019performance}
K.~Yuan, S.~A. Alghunaim, B.~Ying, and A.~H. Sayed, ``On the performance of
  exact diffusion over adaptive networks,'' in \emph{2019 IEEE 58th Conference
  on Decision and Control (CDC)}.\hskip 1em plus 0.5em minus 0.4em\relax IEEE,
  2019, pp. 4898--4903.

\bibitem{lian2017can}
X.~Lian, C.~Zhang, H.~Zhang, C.-J. Hsieh, W.~Zhang, and J.~Liu, ``Can
  decentralized algorithms outperform centralized algorithms? a case study for
  decentralized parallel stochastic gradient descent,'' in \emph{NIPS}, 2017,
  pp. 5336--5346.

\bibitem{chen2015learning2}
J.~Chen and A.~H. Sayed, ``On the learning behavior of adaptive networks—part
  ii: Performance analysis,'' \emph{IEEE Transactions on Information Theory},
  vol.~61, no.~6, pp. 3518--3548, 2015.

\bibitem{yuan2020influence}
K.~Yuan, S.~A. Alghunaim, B.~Ying, and A.~H. Sayed, ``On the influence of
  bias-correction on distributed stochastic optimization,'' \emph{IEEE
  Transactions on Signal Processing}, vol.~68, pp. 4352--4367, 2020.

\bibitem{huang2021improving}
K.~Huang and S.~Pu, ``Improving the transient times for distributed stochastic
  gradient methods,'' \emph{IEEE Transactions on Automatic Control}, 2022.

\bibitem{pu2020asymptotic}
S.~Pu, A.~Olshevsky, and I.~C. Paschalidis, ``Asymptotic network independence
  in distributed stochastic optimization for machine learning: Examining
  distributed and centralized stochastic gradient descent,'' \emph{IEEE signal
  processing magazine}, vol.~37, no.~3, pp. 114--122, 2020.

\bibitem{xin2019variance}
R.~Xin, U.~A. Khan, and S.~Kar, ``Variance-reduced decentralized stochastic
  optimization with gradient tracking,'' \emph{arXiv preprint
  arXiv:1909.11774}, 2019.

\bibitem{xin2021fast}
------, ``A fast randomized incremental gradient method for decentralized
  non-convex optimization,'' \emph{IEEE Transactions on Automatic Control}, pp.
  1--1, 2021.

\bibitem{bertsekas2011}
D.~P. Bertsekas, ``Incremental proximal methods for large scale convex
  optimization,'' \emph{Math. Program.}, vol. 129, no.~2, p. 163, 2011.

\bibitem{bottou2009}
L.~Bottou, ``Curiously fast convergence of some stochastic gradient descent
  algorithms,'' in \emph{Proceedings of the Symposium on Learning and Data
  Science, Paris}, 2009.

\bibitem{bottou2012}
------, ``Stochastic gradient descent tricks,'' in \emph{Neural Networks:
  Tricks of the Trade}.\hskip 1em plus 0.5em minus 0.4em\relax Springer, 2012,
  pp. 421--436.

\bibitem{ying2018stochastic}
B.~Ying, K.~Yuan, S.~Vlaski, and A.~H. Sayed, ``Stochastic learning under
  random reshuffling with constant step-sizes,'' \emph{IEEE Transactions on
  Signal Processing}, vol.~67, no.~2, pp. 474--489, 2018.

\bibitem{gurbu2019}
M.~G{\"u}rb{\"u}zbalaban, A.~Ozdaglar, and P.~Parrilo, ``Why random reshuffling
  beats stochastic gradient descent,'' \emph{Math. Program.}, vol. 186, no.
  1-2, pp. 49--84, 2021.

\bibitem{haochen2019}
J.~Z. {HaoChen} and S.~Sra, ``Random shuffling beats {SGD} after finite
  epochs,'' in \emph{International Conference on Machine Learning}, 2019, pp.
  2624--2633.

\bibitem{nguyen2020unified}
L.~M. Nguyen, Q.~Tran-Dinh, D.~T. Phan, P.~H. Nguyen, and M.~van Dijk, ``A
  unified convergence analysis for shuffling-type gradient methods,''
  \emph{Journal of Machine Learning Research}, vol.~22, no. 207, pp. 1--44,
  2021.

\bibitem{mishchenko2020random}
K.~Mishchenko, A.~Khaled Ragab~Bayoumi, and P.~Richt{\'a}rik, ``Random
  reshuffling: Simple analysis with vast improvements,'' \emph{Advances in
  Neural Information Processing Systems}, vol.~33, 2020.

\bibitem{tsitsiklis1986distributed}
J.~Tsitsiklis, D.~Bertsekas, and M.~Athans, ``Distributed asynchronous
  deterministic and stochastic gradient optimization algorithms,'' \emph{IEEE
  transactions on automatic control}, vol.~31, no.~9, pp. 803--812, 1986.

\bibitem{nedic2010constrained}
A.~Nedi{\'c}, A.~Ozdaglar, and P.~A. Parrilo, ``Constrained consensus and
  optimization in multi-agent networks,'' \emph{IEEE Transactions on Automatic
  Control}, vol.~55, no.~4, pp. 922--938, 2010.

\bibitem{lobel2011distributed}
I.~Lobel, A.~Ozdaglar, and D.~Feijer, ``Distributed multi-agent optimization
  with state-dependent communication,'' \emph{Mathematical programming}, vol.
  129, no.~2, pp. 255--284, 2011.

\bibitem{jakovetic2014fast}
D.~Jakoveti{\'c}, J.~Xavier, and J.~M. Moura, ``Fast distributed gradient
  methods,'' \emph{IEEE Transactions on Automatic Control}, vol.~59, no.~5, pp.
  1131--1146, 2014.

\bibitem{xu2015augmented}
J.~Xu, S.~Zhu, Y.~C. Soh, and L.~Xie, ``Augmented distributed gradient methods
  for multi-agent optimization under uncoordinated constant stepsizes,'' in
  \emph{2015 54th IEEE Conference on Decision and Control (CDC)}.\hskip 1em
  plus 0.5em minus 0.4em\relax IEEE, 2015, pp. 2055--2060.

\bibitem{kia2015distributed}
S.~S. Kia, J.~Cort{\'e}s, and S.~Mart{\'\i}nez, ``Distributed convex
  optimization via continuous-time coordination algorithms with discrete-time
  communication,'' \emph{Automatica}, vol.~55, pp. 254--264, 2015.

\bibitem{shi2015extra}
W.~Shi, Q.~Ling, G.~Wu, and W.~Yin, ``Extra: An exact first-order algorithm for
  decentralized consensus optimization,'' \emph{SIAM Journal on Optimization},
  vol.~25, no.~2, pp. 944--966, 2015.

\bibitem{di2016next}
P.~Di~Lorenzo and G.~Scutari, ``Next: In-network nonconvex optimization,''
  \emph{IEEE Transactions on Signal and Information Processing over Networks},
  vol.~2, no.~2, pp. 120--136, 2016.

\bibitem{qu2017harnessing}
G.~Qu and N.~Li, ``Harnessing smoothness to accelerate distributed
  optimization,'' \emph{IEEE Transactions on Control of Network Systems}, 2017.

\bibitem{nedic2017achieving}
A.~Nedi{\'c}, A.~Olshevsky, and W.~Shi, ``Achieving geometric convergence for
  distributed optimization over time-varying graphs,'' \emph{SIAM Journal on
  Optimization}, vol.~27, no.~4, pp. 2597--2633, 2017.

\bibitem{xu2017convergence}
J.~Xu, S.~Zhu, Y.~C. Soh, and L.~Xie, ``Convergence of asynchronous distributed
  gradient methods over stochastic networks,'' \emph{IEEE Transactions on
  Automatic Control}, vol.~63, no.~2, pp. 434--448, 2017.

\bibitem{pu2020push}
S.~Pu, W.~Shi, J.~Xu, and A.~Nedi{\'c}, ``Push-pull gradient methods for
  distributed optimization in networks,'' \emph{IEEE Transactions on Automatic
  Control}, 2020.

\bibitem{chen2012limiting}
J.~Chen and A.~H. Sayed, ``On the limiting behavior of distributed optimization
  strategies,'' in \emph{2012 50th Annual Allerton Conference on Communication,
  Control, and Computing (Allerton)}.\hskip 1em plus 0.5em minus 0.4em\relax
  IEEE, 2012, pp. 1535--1542.

\bibitem{chen2015learning}
------, ``On the learning behavior of adaptive networks—part i: Transient
  analysis,'' \emph{IEEE Transactions on Information Theory}, vol.~61, no.~6,
  pp. 3487--3517, 2015.

\bibitem{yuan2021removing}
K.~Yuan and S.~A. Alghunaim, ``Removing data heterogeneity influence enhances
  network topology dependence of decentralized sgd,'' 2021.

\bibitem{xin2021improved}
R.~Xin, U.~A. Khan, and S.~Kar, ``An improved convergence analysis for
  decentralized online stochastic non-convex optimization,'' \emph{IEEE
  Transactions on Signal Processing}, vol.~69, pp. 1842--1858, 2021.

\bibitem{alghunaim2021unified}
S.~A. Alghunaim and K.~Yuan, ``A unified and refined convergence analysis for
  non-convex decentralized learning,'' \emph{IEEE Transactions on Signal
  Processing}, vol.~70, pp. 3264--3279, 2022.

\bibitem{bottou2009curiously}
L.~Bottou, ``Curiously fast convergence of some stochastic gradient descent
  algorithms,'' in \emph{Proceedings of the symposium on learning and data
  science, Paris}, vol.~8, 2009, pp. 2624--2633.

\bibitem{bottou2012stochastic}
------, ``Stochastic gradient descent tricks,'' in \emph{Neural networks:
  Tricks of the trade}.\hskip 1em plus 0.5em minus 0.4em\relax Springer, 2012,
  pp. 421--436.

\bibitem{nagaraj2019sgd}
D.~Nagaraj, P.~Jain, and P.~Netrapalli, ``Sgd without replacement: Sharper
  rates for general smooth convex functions,'' in \emph{International
  Conference on Machine Learning}.\hskip 1em plus 0.5em minus 0.4em\relax PMLR,
  2019, pp. 4703--4711.

\bibitem{tran2021smg}
T.~H. Tran, L.~M. Nguyen, and Q.~Tran-Dinh, ``Smg: A shuffling gradient-based
  method with momentum,'' in \emph{International Conference on Machine
  Learning}.\hskip 1em plus 0.5em minus 0.4em\relax PMLR, 2021, pp.
  10\,379--10\,389.

\bibitem{lix2021convergence}
X.~Li, A.~Milzarek, and J.~Qiu, ``Convergence of random reshuffling under the
  {K}urdyka-{L}ojasiewicz inequality,'' \emph{arXiv preprint arXiv:2110.04926},
  2021.

\bibitem{yuan2018variance}
K.~Yuan, B.~Ying, J.~Liu, and A.~H. Sayed, ``Variance-reduced stochastic
  learning by networked agents under random reshuffling,'' \emph{IEEE
  Transactions on Signal Processing}, vol.~67, no.~2, pp. 351--366, 2018.

\bibitem{jiang2021distributed}
X.~Jiang, X.~Zeng, J.~Sun, J.~Chen, and L.~Xie, ``Distributed stochastic
  proximal algorithm with random reshuffling for non-smooth finite-sum
  optimization,'' \emph{arXiv preprint arXiv:2111.03820}, 2021.

\bibitem{nesterov2003introductory}
Y.~Nesterov, \emph{Introductory lectures on convex optimization: A basic
  course}.\hskip 1em plus 0.5em minus 0.4em\relax Springer Science \& Business
  Media, 2003, vol.~87.

\bibitem{zhang2021fedpd}
X.~Zhang, M.~Hong, S.~Dhople, W.~Yin, and Y.~Liu, ``Fedpd: A federated learning
  framework with adaptivity to non-iid data,'' \emph{IEEE Transactions on
  Signal Processing}, vol.~69, pp. 6055--6070, 2021.

\bibitem{morral2017success}
G.~Morral, P.~Bianchi, and G.~Fort, ``Success and failure of
  adaptation-diffusion algorithms with decaying step size in multiagent
  networks,'' \emph{IEEE Transactions on Signal Processing}, vol.~65, no.~11,
  pp. 2798--2813, 2017.

\bibitem{zeng2018nonconvex}
J.~Zeng and W.~Yin, ``On nonconvex decentralized gradient descent,'' \emph{IEEE
  Transactions on signal processing}, vol.~66, no.~11, pp. 2834--2848, 2018.

\bibitem{qureshi2020s}
M.~I. Qureshi, R.~Xin, S.~Kar, and U.~A. Khan, ``S-addopt: Decentralized
  stochastic first-order optimization over directed graphs,'' \emph{IEEE
  Control Systems Letters}, vol.~5, no.~3, pp. 953--958, 2020.

\bibitem{mnist}
Y.~{Lecun}, L.~{Bottou}, Y.~{Bengio}, and P.~{Haffner}, ``Gradient-based
  learning applied to document recognition,'' \emph{Proceedings of the IEEE},
  vol.~86, no.~11, pp. 2278--2324, 1998.

\bibitem{krizhevsky2009learning}
A.~Krizhevsky, G.~Hinton \emph{et~al.}, ``Learning multiple layers of features
  from tiny images,'' 2009.

\end{thebibliography}

\begin{IEEEbiography}[{\includegraphics[width=1in,height=1.25in,clip,keepaspectratio]{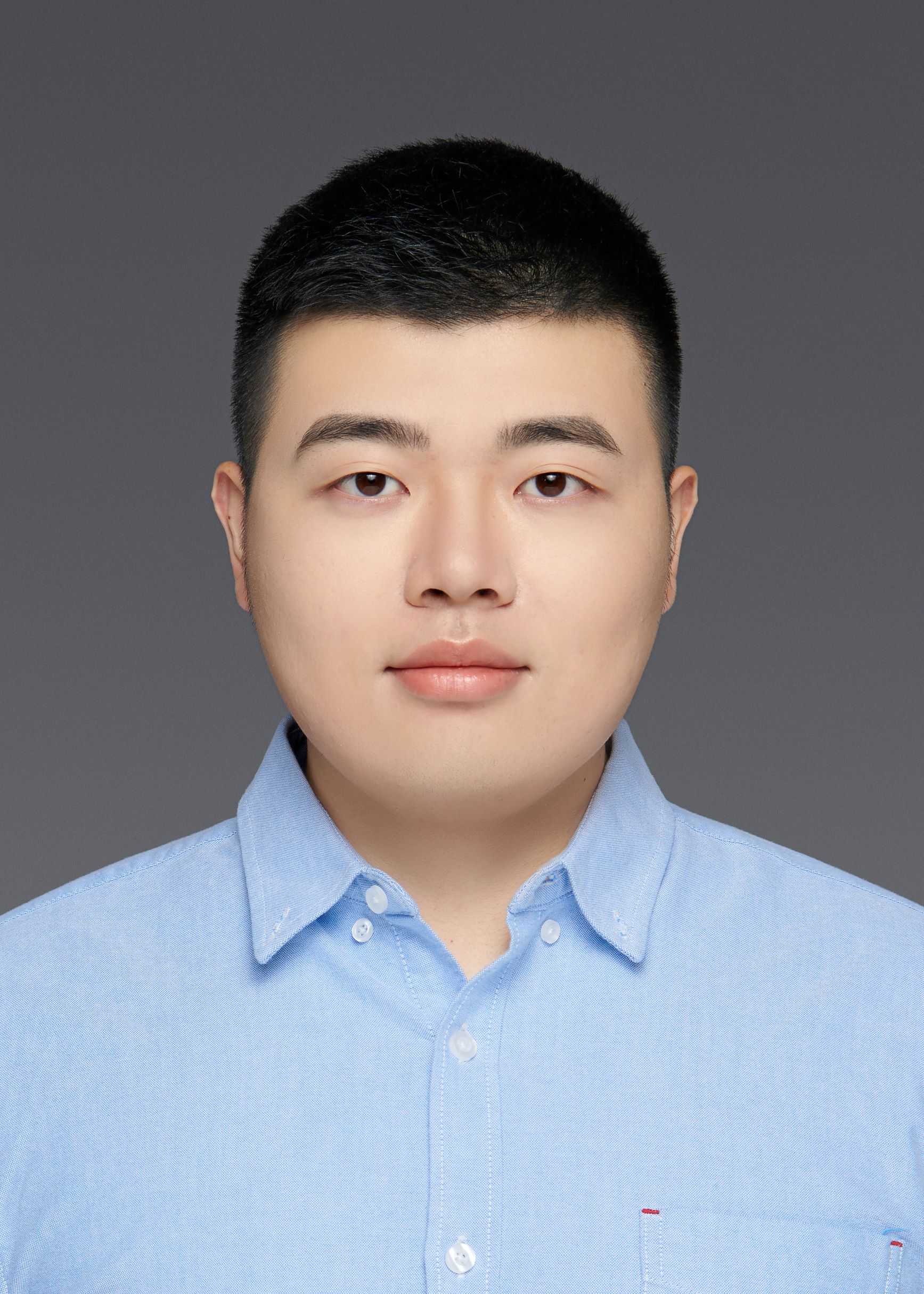}}]
{Kun Huang} is currently a Ph.D. student in data science at the School of Data Science, The Chinese University of Hong Kong, Shenzhen, China. He obtained a B.S. degree in Applied Mathematics from Tongji University in 2018, and an M.S. degree in Statistics from the University of Connecticut in 2020. His research interests primarily lie in the fields of distributed optimization and machine learning.
\end{IEEEbiography}

\begin{IEEEbiography}[{\includegraphics[width=1in,height=1.25in,clip,keepaspectratio]{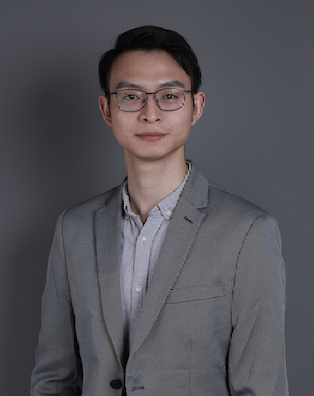}}]
  {Xiao Li} is an assistant professor at the School of Data Science at the Chinese University of Hong Kong, Shenzhen. He received his Ph.D. degree from the Chinese University of Hong Kong in 2020 and his B.Eng. degree from Zhejiang University of Technology in 2016. He was a visiting scholar at the University of Southern California from October 2018 to April 2019. His research focuses on stochastic, nonsmooth, and nonconvex optimization with applications to machine learning and signal processing.
\end{IEEEbiography}

\begin{IEEEbiography}[{\includegraphics[width=1in,height=1.25in,clip,keepaspectratio]{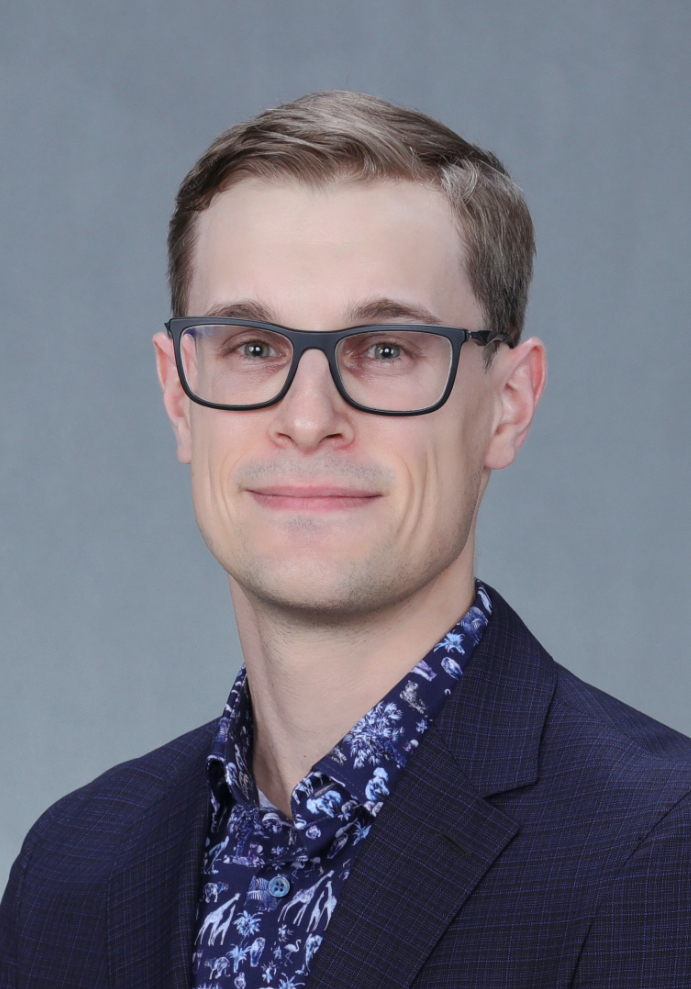}}]
  {Andre Milzarek}
  is currently an assistant professor in the School of Data Science at The Chinese University of Hong Kong, Shenzhen, China. He is also affiliated with the Shenzhen Research Institute of Big Data. He received his B.S., M.S., and Ph.D. degree in Mathematics from the Technical University of Munich in 2010, 2013, and 2016, respectively.  He was a postdoctoral researcher at the Technical University of Munich and at the Beijing International Center for Mathematical Research at Peking University from 2016 to 2019. His research interests include nonsmooth optimization, large-scale and stochastic optimization, and second order methods and theory.
\end{IEEEbiography}

  \begin{IEEEbiography}[{\includegraphics[width=1in,height=1.25in,clip,keepaspectratio]{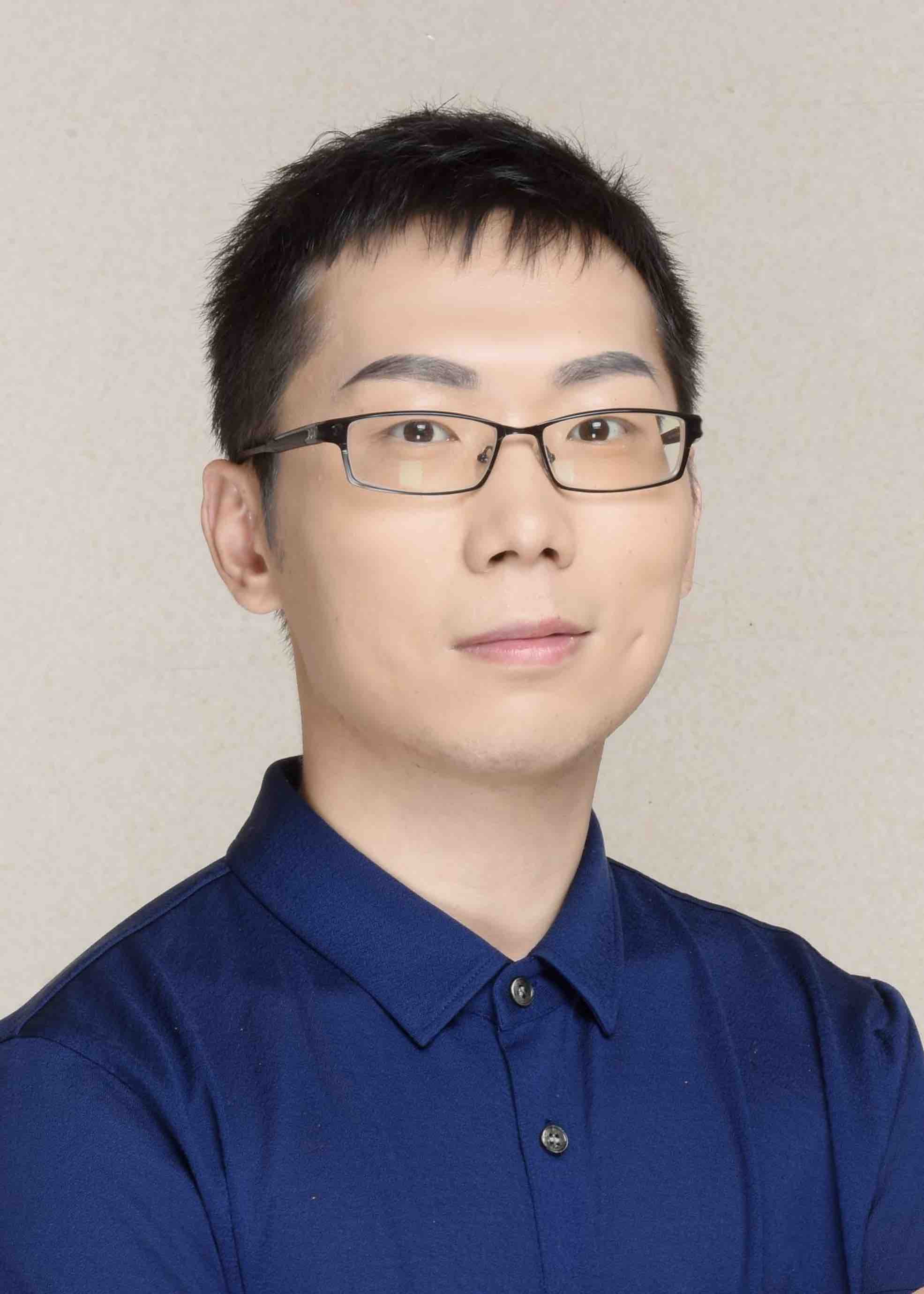}}]
    {Shi Pu}
		is currently an assistant professor in the School of Data Science, The Chinese University of Hong Kong, Shenzhen, China. He is also affiliated with Shenzhen Research Institute of Big Data. He received a B.S. Degree from Peking University, in 2012, and a Ph.D. Degree in Systems Engineering from the University of Virginia, in 2016. He was a postdoctoral associate at the University of Florida, Arizona State University and Boston University, respectively from 2016 to 2019. His research
		interests include distributed optimization, network science, machine learning, and game theory.
	\end{IEEEbiography}

\begin{IEEEbiography}[{\includegraphics[width=1in,height=1.25in,clip,keepaspectratio]{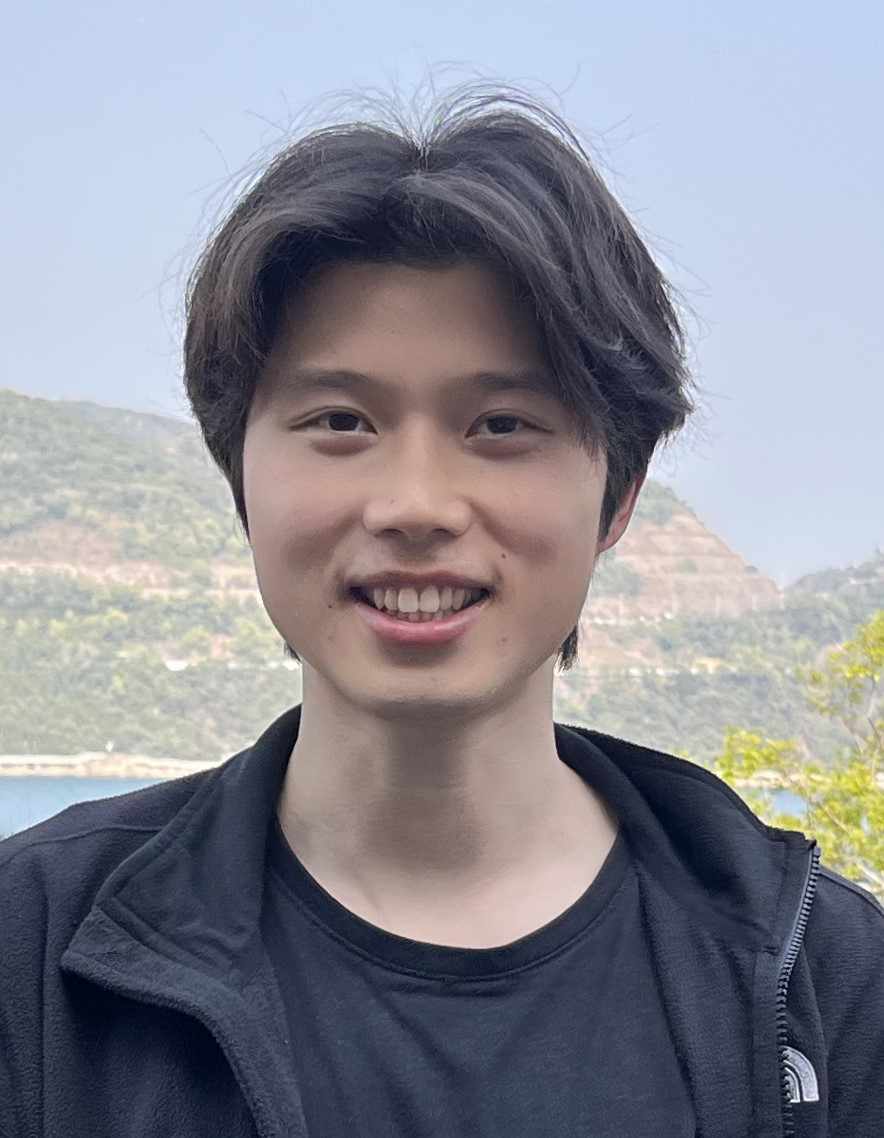}}]
  {Junwen Qiu} is currently a PhD candidate in Data Science at the School of Data Science, The Chinese University of Hong Kong, Shenzhen. He received his Bachelor's degree in the Department of Mathematics from Jinan University in 2019. His research focuses on stochastic optimization and nonconvex nonsmooth optimization.
\end{IEEEbiography}

\vfill

\newpage
\section*{Supplementary Material}

\setcounter{section}{1}
\section{Parts of Proofs for strongly convex case}

\subsection{Proof of Lemma \ref{lem:rhow}} 
\label{app:lem_rhow}
\begin{proof}
    Based on Assumption \ref{ass:W}, we have 
    \begin{align*}
        \norm{W\bs{\omega} - \1\bar{\omega}^{\T}} &= \norm{\prt{W-\frac{1}{n}\1\1^{\T}}\prt{\bs{\omega} - \1\bar{\omega}^{\T}}}\\
        &\leq\norm{W - \frac{1}{n}\1\1^{\T}}\norm{(I - \frac{1}{n}\1\1^{\T})\bs{\omega}}.
    \end{align*}
\end{proof}

\subsection{Proof of Lemma \ref{lem:sigma}}
\label{app:lem_sigma}
\begin{proof}
     {Using \eqref{eq:muL} and \eqref{eq:limit_avg}, it follows}
    \begin{align}
        & \E\brkn{D_{\bar{s}_\ell}(\bxs{\ell}, x^*)}
        \leq \frac{L}{2}\E\brkn{\normn{\bxs{\ell} - x^*}^2}\nonumber\\
        &= \frac{L {\alpha_t^2}}{2}\E\brk{\norm{\sum_{k=0}^{\ell - 1}\prt{\frac{1}{n}\sumn \nabla \fp{i}{k}(x^*)}}^2}\nonumber\\
        &\leq \frac{L {\alpha_t^2}}{2n}\sumn \E\brk{\norm{\sum_{k=0}^{\ell - 1}\nabla\fp{i}{k}(x^*)}^2} \nonumber\\
        &=  {\frac{L{\alpha_t^2}}{2n}
        \sumn\frac{\ell(m-\ell)}{(m-1)m}\sum_{j=1}^{m}\normn{\nabla f_{i,j}(x^*)}^2} \leq \frac{L\alpha_t^2 m}{4}\sigma_*^2\label{eq:sigma_s2},
    \end{align}
    where 
    \eqref{eq:sigma_s2} holds according to Proposition 1 in \cite{mishchenko2020random}. The lower bound can be derived similarly.
\end{proof}

\subsection{Proof of Lemma \ref{lem:At_ds}}
\label{app:lem_At_ds}
\begin{proof}
    Unroll \eqref{eq:At} in Lemma \ref{lem:lya} and invoke the relation between $\svar$ and $\sigma_*^2$ according to Lemma \ref{lem:sigma}, we get 
    \small
    \begin{align}
        & H_t \leq \prod_{j=0}^{t-1}\prt{1 - \frac{\alpha_j\mu}{4}}^m H_0 + \sigma_*^2\sum_{j=0}^{t-1}\crk{\brk{\prod_{q = j+1}^{t-1} \prt{1 - \frac{\alpha_q\mu}{4}}^m} \alpha_j^3\nonumber\right.\\ 
        &\left.  \brk{\sum_{k=0}^{m-1}\prt{1 - \frac{\alpha_j\mu}{4}}^k}
        \brk{\frac{Lm}{2} \prt{1 + \frac{240\alpha_j^2\rho_w^2L^3}{\mu(1-\rho_w^2)^2}} + \frac{240\rho_w^2L^2}{\mu(1 - \rho_w^2)^2}} }\nonumber\\
        &\leq \prod_{j=0}^{t-1}\prt{1 - \frac{\alpha_j\mu}{4}}^m H_0 + \sum_{j=0}^{t-1}\crk{\brk{\prod_{q = j+1}^{t-1} \prt{1 - \frac{\alpha_q\mu}{4}}^m} \alpha_j^3\nonumber\right.\\ 
        &\left.\brk{\sum_{k=0}^{m-1}\prt{1 - \frac{\alpha_j\mu}{4}}^k}}\prt{mL + \frac{240\rho_w^2L^2}{\mu(1 - \rho_w^2)^2}} \sigma_*^2\label{eq:At0}
    \end{align}\normalsize
    
    The last inequality holds for $\alpha_j^2 \leq \frac{(1-\rho_w^2)^2\mu}{240\rho_w^2 L^3}$ which is necessary for $\alpha_j\leq \frac{(1-\rho_w^2)\mu}{8\sqrt{30} L^2}$. {\kh Let $\alpha_t = \alpha$ in \eqref{eq:At0}. We obtain 
	\small
	\begin{equation}
		\label{eq:Ht_contant}
		\begin{aligned}
			& H_t \leq \prt{1 - \frac{\alpha\mu}{4}}^{mt} H_0 + \sum_{j=0}^{t-1}\crk{\brk{\prt{1 - \frac{\alpha\mu}{4}}^{m(t - 1-j)}}\right.\\
			&\left. \alpha^3\brk{\sum_{k=0}^{m-1}\prt{1 - \frac{\alpha\mu}{4}}^k}}
			\prt{mL + \frac{240\rho_w^2L^2}{\mu(1 - \rho_w^2)^2}} \sigma_*^2\\
			&= \prt{1 - \frac{\alpha\mu}{4}}^{mt} H_0\\
			& + \sum_{j=0}^{t-1}\sum_{k=0}^{m-1}\brk{\prt{1 - \frac{\alpha\mu}{4}}^{m(t - 1-j)+k}\alpha^3}\prt{mL + \frac{240\rho_w^2L^2}{\mu(1 - \rho_w^2)^2}} \sigma_*^2\\
			&\leq \prt{1 - \frac{\alpha\mu}{4}}^{mt} H_0 + \frac{4\alpha^2}{\mu} \prt{mL + \frac{240\rho_w^2L^2}{\mu(1 - \rho_w^2)^2}} \sigma_*^2.
		\end{aligned}
	\end{equation}\normalsize
	}
	Substitute $\alpha_t = \frac{\theta}{{    m}\mu(t +K)}$ into \eqref{eq:At0}, we obtain
    \small
    \begin{align}
        & H_t \leq \prod_{j=0}^{t-1}\prt{1 - \frac{\theta}{4{    m}(j + K)}}^m H_0\nonumber\\
        & + \sum_{j=0}^{t - 1}\crk{\brk{\prod_{q = j + 1}^{t - 1}\prt{1 - \frac{\theta}{4{    m}(q + K)}}^m} \frac{1}{(j + K)^3}\nonumber\right.\\
        &\quad\left. \brk{\sum_{k=0}^{m - 1}\prt{1 - \frac{\theta}{4{    m}(j + K)}}^k}}\prt{mL + \frac{240\rho_w^2L^2}{\mu(1 - \rho_w^2)^2}} \frac{\theta^3\sigma^2_*}{{    m^3}\mu^3}\nonumber\\
        &\leq \prt{\frac{K}{t + K}}^{\frac{\theta }{4}} H_0 + \prt{mL + \frac{240\rho_w^2L^2}{\mu(1 - \rho_w^2)^2}} \frac{\theta^3\sigma^2_*}{{    m^3}\mu^3 (t + K)^{\theta  / 4}}\nonumber\\
        &\quad \cdot \sum_{j=0}^{t-1}\frac{( j + K + 1)^{\theta  / 4}m}{(j+K)^3}\label{eq:ds_s1}\\
        &\leq \prt{\frac{K}{t + K}}^{\frac{\theta }{4}} H_0 + \biggl[\prt{mL + \frac{240\rho_w^2L^2}{\mu(1 - \rho_w^2)^2}}\nonumber\\
        &\quad\cdot \frac{{   8}\theta^3 \sigma^2_*}{{    m^2}\mu^3(\theta -8)}\frac{1}{(t+K)^2}\biggr],\label{eq:ds_s2}
    \end{align}\normalsize	
    where \eqref{eq:ds_s1} holds for Lemma \ref{lem:prod}{  by letting $K\geq \frac{\theta}{2m}$} and $\sum_{k=0}^{m - 1}\prt{1 - \frac{\theta}{4{    m}(j + K)}}^k \leq m$. \eqref{eq:ds_s2} holds because{  when $K\geq \frac{\theta}{2}\geq \frac{\theta}{2m}$,
    
    \small 
    \begin{align*}
        \prt{\frac{j + K + 1}{j + K}}^{\frac{\theta}{4}} &\leq \prt{1 + \frac{1}{K}}^{\frac{\theta}{4}}\leq \prt{1 + \frac{2}{\theta}}^{\frac{\theta}{4}}\leq \exp\prt{\frac{1}{2}}\leq 2,
    \end{align*}}\normalsize
    and when $\theta  > 12$,
    \small
    \begin{align*}
        \sum_{j=0}^{t - 1}(j + K)^{\theta /4 - 3} &\leq \int_0^t (x + K)^{\theta / 4 - 3} \mathrm{dx}\leq \frac{4}{\theta  - 8}(t + K)^{\theta  / 4 -2}.
    \end{align*}\normalsize
\end{proof}

\subsection{Proof of Lemma \ref{lem:Atl_ds}}
\label{app:lem_Atl_ds}
\begin{proof}

	{\sp Lemma \ref{lem:At_ds} provides two upper bounds in the form of $H_t\leq \mathcal{H}(\alpha_t)$ for some $\mathcal{H}(\alpha_t)$ depending on the choice of the stepsize policy. Applying the upper bounds to \eqref{eq:Atl} in Lemma \ref{lem:lya} and invoking Lemma \ref{lem:sigma}, we obtain
	\begin{equation}
		\label{eq:Htl_s1}
		\begin{aligned}
			& H_t^{\ell} \leq \prt{1 - \frac{\alpha_t\mu}{4}}^{\ell}\mathcal{H}(\alpha_t) + 2\brk{\frac{\alpha_t^3Lm\sigma^2_*}{4}\prt{1 + \frac{240\alpha_t^2\rho_w^2L^3}{\mu(1-\rho_w^2)^2}}\right.\\
			&\quad \left. + \frac{120\alpha_t^3\rho_w^2 L^2}{\mu (1-\rho_w^2)^2}\sigma_*^2}{  \frac{4}{\alpha_t\mu}}.
		\end{aligned}
	\end{equation}
	}
	
	{\kh 
	If we use diminishing stepsizes $\alpha_t = \frac{\theta}{m\mu(t + K)}$, relation \eqref{eq:Htl_s1} and Lemma \ref{lem:At_ds} yield
	\small 
	\begin{align*}
		H_t^{\ell}&\leq \prt{\frac{K}{t + K}}^{\frac{\theta }{4}} H_0 + {  \frac{4\theta^2 L\sigma^2_*}{m^2\mu^3}\prt{m + \frac{240\rho_w^2 L}{\mu(1-\rho_w^2)^2}}\frac{1}{(t + K)^2}}\\
		&\quad + \prt{mL + \frac{240\rho_w^2L^2}{\mu(1 - \rho_w^2)^2}} \frac{{  8}\theta^3 \sigma^2_*}{{    m^2}\mu^3(\theta -8)}\frac{1}{(t + K)^2}.
	\end{align*}\normalsize

	If we use a constant stepsize $\alpha_t = \alpha$, then \eqref{eq:Htl_s1} and Lemma \ref{lem:sigma} lead to 
	\small
	\begin{align*}
		H_t^{\ell} &\leq \prt{1 - \frac{\alpha\mu}{4}}^{mt} H_0 + \frac{4\alpha^2}{\mu} \prt{mL + \frac{240\rho_w^2L^2}{\mu(1 - \rho_w^2)^2}} \sigma_*^2\\
		&\quad + \frac{4\alpha^2 L}{\mu}\prt{ m + \frac{240\rho_w^2 L}{\mu(1-\rho_w^2)^2}}\sigma^2_*.
	\end{align*}\normalsize
	}
    
    {\sp Rearranging the terms and noting that $\theta > 12$}, we obtain the desired result.
\end{proof}

\subsection{Proof of Lemma \ref{lem:opt1}}
\label{app:lem_opt1}
\begin{proof}
	Apply Lemma \ref{lem:cons1} into Lemma \ref{lem:xbar0} and note $\alpha_t \leq \frac{1}{\mu}$, we have 
	\begin{align*}
		&\E\brk{\norm{\bx{t}{\ell + 1} - \bxs{\ell + 1}}^2}
		\leq (1-\frac{\alpha_t\mu}{2}) \E\brk{\norm{\bx{t}{\ell} - \bxs{\ell}}^2}\\
		&\quad + 2\alpha_t\svar + \frac{2\alpha_t L^2}{n}\prt{\frac{{   1}}{\mu} + \alpha_t}\crk{\prt{\frac{1+\rho_w^2}{2}}^{mt}
		\right.\\
		&\left. \quad \cdot\norm{\x_0^0 - \1(\bar{x}_0^0)^{\T}}^2
		+ \frac{4\hat{X}_1}{(1-\rho_w^2)^2}\alpha_t^2}
	\end{align*}
	
	Invoke relation \eqref{eq:in_out} {    and Lemma \ref{lem:sigma}}, we obtain
	\begin{align*}
		&\E\brk{\norm{\bar{x}_{t + 1}^0 - x^*}^2}
		\leq (1-\frac{\alpha_t\mu}{2})^{m}\E\brk{\norm{\bar{x}_t^0 - x^*}^2}\\
		&+ {  \frac{\alpha_t^3Lm\sigma^2_*}{2}}\brk{\sum_{k=0}^{m-1} \prt{1 - \frac{\alpha_t\mu}{2}}^k} + \frac{{   4}\alpha_tL^2}{n\mu^2}\crk{\prt{\frac{1+\rho_w^2}{2}}^{mt}\right.\\
		&\left. \cdot\norm{\x_0^0 - \1(\bar{x}_0^0)^{\T}}^2 + \frac{4\hat{X}_1}{(1-\rho_w^2)^2}\alpha_t^2}\brk{\sum_{k=0}^{m-1} \prt{1 - \frac{\alpha_t\mu}{2}}^k}
	\end{align*}
	
	Unroll the above, we have 
	\small
	\begin{align*}
		& \E\brk{\norm{\bar{x}_{t}^0 - x^*}^2}
		\leq \prod_{j = 0}^{t - 1} \prt{1 - \frac{\alpha_j\mu}{2}}^m {\norm{\bar{x}_0^0 - x^*}^2}\\
		&+ {  \frac{mL\sigma^2_*}{2}}\sum_{j=0}^{t - 1}\crk{\brk{\prod_{q = j + 1}^{t - 1} \prt{1 - \frac{\alpha_q\mu}{2}}^m} {   \alpha_j^3}\brk{\sum_{k=0}^{m-1} \prt{1 - \frac{\alpha_j\mu}{2}}^k}}\\ 
		&\quad + \sum_{j=0}^{t - 1}\biggr\{\brk{\prod_{q = j + 1}^{t - 1} \prt{1 - \frac{\alpha_q\mu}{2}}^m} \frac{{   4}\alpha_jL^2}{n\mu^2}\brk{\prt{\frac{1+\rho_w^2}{2}}^{mj}\right. \\
		&\left.\quad \cdot\norm{\x_0 - \1\bar{x}_0^{\T}}^2 + \frac{4\hat{X}_1}{(1-\rho_w^2)^2}\alpha_j^2}\brk{\sum_{k=0}^{m-1} \prt{1 - \frac{\alpha_j\mu}{2}}^k}\biggr\}
	\end{align*}\normalsize
	
	Similar to those in deriving Lemma \ref{lem:At_ds} and by $K\geq\theta$, we obtain
	\small
	\begin{align*}
		& \E\brk{\norm{\bar{x}_{t}^0 - x^*}^2} \leq \prt{\frac{K}{t + K}}^{\frac{\theta}{2}}{\norm{\bar{x}_0^0 - x^*}^2}\\
		&\quad + {  \frac{2\theta^3L\sigma^2_*}{m\mu^3(\theta - 4)}\frac{1}{(t + K)^2}} + \frac{{   96} \theta^3 L^2\hat{X}_1}{nm^2\mu^5(1-\rho_w^2)^2(\theta - 4)}\frac{1}{(t + K)^2}\\
		&\quad + \prt{\frac{K}{t + K}}^{\frac{\theta}{2}} \frac{{   96} L^2 }{n\mu^3(1-\rho_w^2)}{\norm{\x_0^0 - \1(\bar{x}_0^0)^{\T}}^2}.
	\end{align*}\normalsize
	
	The last term is because $(t + K)^{\theta /2 - 1}q_0^t$ is decreasing in $t$ with $q_0 := \prt{\frac{1 + \rho_w^2}{2}}^m$, then 
	\begin{align*}
		&\quad \frac{1}{\ln q_0} d\prt{(t + K)^{\frac{\theta }{2} - 1}q_0^t}\\
		&= \frac{1}{\ln q_0} \prt{\frac{\theta  }{2} - 1}\prt{t + K}^{\frac{\theta  }{2} - 2}q_0^t \mathrm{dt} + (t + K)^{\frac{\theta  }{2} -1}q_0^t \mathrm{dt}\\
		&\geq \frac{1}{2}(t + K)^{\frac{\theta  }{2} - 1}q_0^t dt.
	\end{align*}
	
	Thus, 
	\begin{align*}
		& \sum_{j = 0}^{t - 1}(j + K)^{\frac{\theta  }{2} - 1}q_0^j \leq \int_0^{\infty} (t + K)^{\theta /2 - 1}q_0^t \mathrm{dt}\\    
		&\leq \frac{2}{\ln q_0}\int_K^{\infty} \frac{\mathrm{d}\prt{t^{\frac{\theta  }{2} - 1} q_0^{t - K}}}{\mathrm{dt}}  = -\frac{2 K^{\frac{\theta  }{2} - 1}}{\ln q_0}\leq \frac{4 K^{\frac{\theta  }{2} - 1}}{m(1 - \rho_w^2)}.
	\end{align*}
\end{proof}

\subsection{Proof of Corollary \ref{cor:scvx_complexity}}
\label{app:cor_scvx_complexity}
\begin{proof}
    {  
    Note for $\alpha$ satisfying \eqref{eq:alphat}, 
    \begin{equation*}
        H_0 \ =\order{\frac{n\norm{\bar{x}_0^0 - x^*}^2 + \norm{\x_0^0 - \1(\bar{x}_0^0)^{\T}}^2}{n}}.
    \end{equation*}
    }
    
    Rearrange the result in {Theorem} \ref{cor:combined1} and notice $(1-x)\leq \exp(-x)$. We obtain 
		\begin{align}
			& \frac{1}{n}\sumn\E\brk{\norm{x_{i,T}^0 - x^*}^2} \leq \exp\prt{-\frac{\alpha\mu m T}{4}} H_0 + \frac{4\alpha^2mL\sigma^2_*}{\mu} \nonumber\\
			&\quad + \frac{4\hat{X}_1}{n(1-\rho_w^2)^2}\alpha^2 + \frac{960\rho_w^2 L^2\sigma^2_*}{\mu^2(1-\rho_w^2)^2}\alpha^2\nonumber\\
			&\quad + \exp\prt{-\frac{1-\rho_w^2}{2}mT}\frac{\norm{\x_0^0 - \1(\bar{x}_0^0)^{\T}}^2}{n}\label{eq:scvx_c}.
		\end{align}

		Denoting the stepsize satisfying \eqref{eq:alphat} as $\bar{\alpha}$ and 
	    using $\tilde{\cO}(\cdot)$ to hide the logarithm factors, it is sufficient to discuss the following two cases:

		\textbf{Case I:} If $\alpha = \frac{4}{{    L} m T}\log\frac{H_0\mu^2{    m}T^2}{\kappa \sigma^2_*}\leq \bar{\alpha}$, then we substitute such an $\alpha{  \leq \frac{4}{\mu m T}\log\frac{H_0\mu^2{    m}T^2}{\kappa \sigma^2_*}}$ into \eqref{eq:scvx_c} and get
		\begin{align*}
			& \frac{1}{n}\sumn\E\brk{\norm{x_{i,T}^0 - x^*}^2} = \tilde{\cO}\prt{\frac{\kappa \sigma^2_*}{\mu^2 m T^2}}\\
			&\quad + \exp\prt{-\frac{1-\rho_w^2}{2}mT}\frac{\norm{\x_0^0 - \1(\bar{x}_0^0)^{\T}}^2}{n}\\
			&\quad + \tilde{\cO}\prt{\frac{\hat{X}_1}{{    L^2}nm^2(1-\rho_w^2)^2 T^2}} + \tilde{\cO}\prt{\frac{\sigma^2_*}{{   \mu^2}(1-\rho_w^2)^2m^2 T^2}}.
		\end{align*}  

		\textbf{Case II:} If $\bar{\alpha} < \frac{4}{{    L} m T}\log\frac{H_0\mu^2{    m}T^2}{\kappa \sigma^2_*}{  \leq \frac{4}{\mu m T}\log\frac{H_0\mu^2{    m}T^2}{\kappa \sigma^2_*}}$, {  we substitute $\bar\alpha$ into \eqref{eq:scvx_c} and obtain},
		\begin{align*}
			& \frac{1}{n}\sumn\E\brk{\norm{x_{i,T}^0 - x^*}^2} = \exp\prt{-\frac{\bar\alpha\mu m T}{4}} H_0\\
			&\quad + \exp\prt{-\frac{1-\rho_w^2}{2}mT}\frac{\norm{\x_0^0 - \1(\bar{x}_0^0)^{\T}}^2}{n} + \tilde{\cO}\prt{\frac{\kappa\sigma^2_*}{\mu^2mT^2}}  \\
			&\quad + \tilde{\cO}\prt{\frac{\hat{X}_1}{nm^2{    L^2}(1-\rho_w^2)^2T^2}} + \tilde{\cO}\prt{\frac{\sigma^2_*}{(1-\rho_w^2)^2\mu^2m^2T^2}}.
		\end{align*}

		We finish our proof by noting that
		\small
		{  
		\begin{align*}
			\hat{X}_1 = &\order{\frac{\prt{n\norm{\bar{x}_0^0 - x^*}^2 + \norm{\x_0^0 - \1(\bar{x}_0^{0})^{\T}}^2}L^2}{1-\rho_w^2}+ \frac{mn\sigma^2_*}{1-\rho_w^2}}.
		\end{align*}
		}\normalsize

\end{proof}

\section{Proofs for the Nonconvex Case}
            {\kh~
            \subsection{Proof of Lemma \ref{as:bounded_var}}
            \label{app:lem_bounded_var}
            \begin{proof}
                The proof is similar to \cite[Proposition 2]{mishchenko2020random}, we present it here for completeness. By Assumption \ref{as:comp_fun}, we have 
                \begin{align*}
                    \norm{\nabla f_{i,\ell}(x)} \leq 2L \prt{f_{i,\ell}(x) - \bar{f}_{i,\ell}}, \ \forall x\in\R^p, \ i,\ell.
                \end{align*}
                Therefore, 
                \begin{align*}
                    &\frac{1}{mn}\sumn\sum_{\ell=1}^m \norm{\nabla f_{i,\ell}(x) - \nabla f(x)}^2
                    \leq \frac{1}{mn}\sumn\sum_{\ell=1}^m\norm{\nabla f_{i,\ell}(x)}^2\\
                    &\leq \frac{2L}{mn}\sumn\sum_{\ell=1}^m\prt{f_{i,\ell}(x) - \bar{f}_{i,\ell}}\\
                    &= 2L(f(x) - \bar{f}) + 2L\prt{\bar{f} - \frac{1}{mn}\sumn\sum_{\ell = 1}^m \bar{f}_{i,\ell}}.
                \end{align*}

                We have $A= 2L > 0$ and $B^2 = 2L\prt{\bar{f} - \frac{1}{mn}\sumn\sum_{\ell = 1}^m \bar{f}_{i,\ell}}\geq 0$.
            \end{proof}
            }
		\subsection{Proof of Lemma \ref{lem:descent-property}}
		\label{app:lem_descent-property}
		
		\begin{proof}
			By $L$-smoothness of the objective function $f$, we obtain
			\begin{equation}\small
				\label{eq:lem-desecent-1}
				\begin{aligned}
					& f(\bx{t+1}{0}) \leq f(\bx{t}{0}) - \inpro{\nabla f(\bx{t}{0}), \bx{t}{0}-\bx{t+1}{0}} + \frac{L}{2}\norm{\bx{t+1}{0}-\bx{t}{0}}^2\\
					&{\kh = f(\bx{t}{0}) - \alpha m \inpro{\nabla f(\bx{t}{0}), \frac{1}{mn} \sum_{\ell=0}^{m-1}\sum_{i=1}^{n} \nabla f_{i,\pi_\ell^i}(x_{i,t}^\ell)}}\\
					&\quad{\kh + \frac{L}{2}\norm{\frac{\alpha}{n} \sum_{\ell=0}^{m-1}\sum_{i=1}^{n} \nabla f_{i,\pi_\ell^i}(x_{i,t}^\ell)}^2}\\
					& = f(\bx{t}{0}) - \frac{\alpha{\kh m}}{2}\norm{\nabla f(\bx{t}{0})}^2\\
					&\quad - \frac{{\kh \alpha m}}{2}\prt{{\kh 1 - \alpha m L}}\norm{{\kh~\frac{1}{mn} \sum_{\ell=0}^{m-1}\sum_{i=1}^{n} \nabla f_{i,\pi_\ell^i}(x_{i,t}^\ell)}}^2 \\
					&\quad + \frac{\alpha {\kh m}}{2}\norm{{\kh~\frac{1}{mn} \sum_{\ell=0}^{m-1}\sum_{i=1}^{n} \nabla f_{i,\pi_\ell^i}(x_{i,t}^\ell)}- \nabla f(\bx{t}{0})}^2,
				\end{aligned}
			\end{equation}\normalsize
			where {\kh we invoke} $\inpro{a,b}=(\norm{a}^2+\norm{b}^2-\norm{a-b}^2)/2$. 
			

			Then, utilizing $L$-smoothness of each component function yields
			\begin{equation}
				\label{eq:lem-desecent-4}
				\begin{aligned}
					&{\kh~\norm{\frac{1}{mn} \sum_{\ell=0}^{m-1}\sum_{i=1}^{n} \nabla f_{i,\pi_\ell^i}(x_{i,t}^\ell)- \nabla f(\bx{t}{0})}^2}\\
					&\leq \frac{L^2}{mn}\sum_{\ell = 0}^{m-1}\sum_{i=1}^n \norm{x_{i,t}^{\ell} - \bx{t}{0}}^2\\
					& \leq \frac{2L^2}{mn} \sum_{\ell=0}^{m-1}\norm{\x_t^\ell-  \1(\bx{t}{\ell})^\T}^2 + \frac{2L^2}{m} \sum_{\ell=0}^{m-1}\norm{\bx{t}{\ell}-\bx{t}{0}}^2.
				\end{aligned}
			\end{equation}	

			Inserting \eqref{eq:lem-desecent-4} into \eqref{eq:lem-desecent-1} and using the condition $\alpha\leq \frac{1}{mL}$ finish the proof. 
		\end{proof}
		
		\subsection{Proof of Lemma \ref{lem:noncvx-con-err}}
		\label{app:lem_noncvx-con-err}
		{\kh~
		\begin{proof}
			By the update \eqref{eq:comp} and Young's inequality for $q = \frac{1 + \rho_w^2}{2\rho_w^2}$, we have 
			\small
			\begin{equation}
				\label{eq:noncvx-lem-1}
				\begin{aligned}
					&\norm{\x_t^{\ell + 1} - \1(\bx{t}{\ell + 1})^{\T}}^2 \leq \norm{\prt{W - \frac{\1\1^{\T}}{n}}\prt{\x_t^{\ell} - \alpha \nabla \Fp{\ell}(\x_t^{\ell})}}^2\\
					&\leq \frac{1 + \rho_w^2}{2}\norm{\x_t^{\ell} - \1(\bx{t}{\ell})^{\T}}^2 + \frac{2\alpha^2}{1-\rho_w^2}\norm{\nabla \Fp{\ell}(\x_t^{\ell})}^2.
 				\end{aligned}
			\end{equation}\normalsize

			{\sp It follows that}
			\begin{equation}
				\label{eq:noncvx-lem-2}
				\begin{aligned}
					& \sum_{\ell = 0}^{m-1} \norm{\x_t^{\ell} - \1(\bx{t}{\ell})^{\T}}^2 \leq \frac{2}{1-\rho_w^2} \norm{\x_t^0 - \1(\bx{t}{0})^{\T}}^2\\
					&\quad +  \frac{2\alpha^2}{1-\rho_w^2}\sum_{\ell = 0}^{m-1}\sum_{j=0}^{\ell - 1}\prt{\frac{1 + \rho_w^2}{2}}^{\ell - j - 1}\sum_{i=1}^n\norm{\nabla \fp{i}{j}(x_{i,t}^j)}^2.
				\end{aligned}
			\end{equation}

			Next consider the last term in \eqref{eq:noncvx-lem-2}. We have 
			\small
			\begin{align}
				&\frac{1}{3}\sum_{\ell = 0}^{m-1}\sum_{j=0}^{\ell - 1}\prt{\frac{1 + \rho_w^2}{2}}^{\ell - j - 1}\sum_{i=1}^n\norm{\nabla \fp{i}{j}(x_{i,t}^j)}^2\nonumber\\
				&\leq \sum_{\ell = 0}^{m-1}\sum_{j=0}^{\ell - 1}\prt{\frac{1 + \rho_w^2}{2}}^{\ell - j - 1}\sum_{i=1}^n\crk{\norm{\nabla \fp{i}{j}(x_{i,t}^j) - \nabla \fp{i}{j}(\bx{t}{0})}^2\right.\nonumber\\
				&\quad\left.  + \norm{\nabla \fp{i}{j}(\bx{t}{0}) - \nabla f(\bx{t}{0})}^2 + \norm{\nabla f(\bx{t}{0})}^2}\nonumber\\
				&\leq \sum_{\ell = 0}^{m-1}\sum_{j=0}^{\ell - 1}\prt{\frac{1 + \rho_w^2}{2}}^{\ell - j - 1}\sum_{i=1}^n\crk{L^2\norm{x_{i,t}^j - \bx{t}{0}}^2\right.\nonumber\\
				&\quad\left. + \norm{\nabla \fp{i}{j}(\bx{t}{0}) - \nabla f(\bx{t}{0})}^2 + \norm{\nabla f(\bx{t}{0})}^2}\label{eq:noncvx-cons-s1}\\
				&\leq \frac{2L^2}{1-\rho_w^2}\sum_{\ell=0}^{m-1}\sum_{i=1}^n\norm{x_{i,t}^{\ell} - \bx{t}{0}}^2 + \frac{2mn}{1-\rho_w^2}\norm{\nabla f(\bx{t}{0})}^2\nonumber\\
				&\quad + \frac{2m^2n}{1-\rho_w^2}\brk{2A\prt{f(\bx{t}{0}) - \bar{f}} + B^2},\label{eq:noncvx-cons-s2}
			\end{align}\normalsize
			where \eqref{eq:noncvx-cons-s1} comes from Assumption \ref{as:comp_fun}, and \eqref{eq:noncvx-cons-s2} holds by Lemma \ref{as:bounded_var}. Combining \eqref{eq:lem-desecent-4}, \eqref{eq:noncvx-lem-2} and \eqref{eq:noncvx-cons-s2} leads to
			\small
			\begin{equation}
				\label{eq:noncvx-cons-ell}
				\begin{aligned}
					&\sum_{\ell=0}^{m-1}\norm{\x_t^{\ell} - \1(\bx{t}{\ell})^{\T}}^2\leq \frac{2}{1-\rho_w^2}\norm{\x_t^0 - \1(\bx{t}{0})^{\T}}^2 + \frac{12B^2\alpha^2 m^2n}{(1-\rho_w^2)^2}\nonumber\\
					&\quad + \frac{12mn\alpha^2}{(1-\rho_w^2)^2}\norm{\nabla f(\bx{t}{0})}^2 + \frac{24\alpha^2nL^2}{(1-\rho_w^2)^2}\sum_{\ell = 0}^{m-1}\norm{\bx{t}{\ell} - \bx{t}{0}}^2\nonumber\\
					&\quad + \frac{24\alpha^2 L^2}{(1-\rho_w^2)^2}\sum_{\ell = 0}^{m-1}\norm{\x_t^\ell - \1(\bx{t}{\ell})^{\T}}^2\\
					&\quad + \frac{24 A\alpha^2 m^2n}{(1-\rho_w^2)^2}\prt{f(\bx{t}{0}) - \bar{f}}.
				\end{aligned}
			\end{equation}\normalsize

			Similar to \eqref{eq:noncvx-cons-s2}, we have 
			\small
			\begin{equation}
				\label{eq:noncvx-cons-bar}
				\begin{aligned}
					&\sum_{\ell = 0}^{m-1}\norm{\bx{t}{\ell} - \bx{t}{0}}^2 = m^2 \alpha^2 \sum_{\ell = 0}^{m-1}\norm{\frac{1}{mn}\sum_{j=0}^{\ell - 1}\sum_{i=1}^n\nabla \fp{i}{j}(x_{i,t}^j)}^2\\
					&\leq \frac{m\alpha^2}{n}\sum_{\ell = 0}^{m-1}\sum_{j=0}^{\ell-1}\sum_{i=1}^n \norm{\nabla \fp{i}{j}(x_{i,t}^j)}^2\\
					&\leq \frac{6m^2\alpha^2L^2}{n}\sum_{\ell=0}^{m-1}\norm{\x_t^\ell - \1(\bx{t}{\ell})^{\T}}^2 + 6m^2\alpha^2L^2\sum_{\ell=0}^{m-1}\norm{\bx{t}{\ell} - \bx{t}{0}}^2\\
					&\quad + 3m^3\alpha^2\brk{2A\prt{f(\bx{t}{0}) - \bar{f}} + B^2} + 3m^3\alpha^2\norm{\nabla f(\bx{t}{0})}^2.
				\end{aligned}
			\end{equation}\normalsize
			
			 {\sp Since $\cL_t=\frac{1}{n}\sum_{\ell=0}^{m-1}\norm{\x_t^{\ell} - \1(\bx{t}{\ell})^{\T}}^2 + \sum_{\ell = 0}^{m-1}\norm{\bx{t}{\ell} - \bx{t}{0}}^2$}, we have 
			\small
			\begin{align*}
				&\prt{1 - 6m^2\alpha^2L^2 - \frac{24\alpha^2L^2}{(1-\rho_w^2)^2}}\cL_t \leq \frac{2}{n(1-\rho_w^2)}\norm{\x_t^0 - \1(\bx{t}{0})^{\T}}^2\\
				&\quad + \frac{3m^2\alpha^2B^2(m + 4)}{(1-\rho_w^2)^2} + \frac{3\alpha^2m(4 + m^2)}{(1-\rho_w^2)^2}\norm{\nabla f(\bx{t}{0})}^2\\
				&\quad + \frac{6m^2\alpha^2 A(4 + m)}{(1-\rho_w^2)}\prt{f(\bx{t}{0})  - \bar{f}}.
			\end{align*}\normalsize
			
			Let 
			\begin{equation*}
				\alpha\leq \min\crk{\frac{1}{2\sqrt{6}mL}, \frac{1-\rho_w^2}{4\sqrt{6}L}}.
			\end{equation*}
			We obtain the desired result for $\cL^t$. {\sp Then}, from \eqref{eq:noncvx-lem-1}, we have 
			\small
			\begin{equation}
				\label{eq:noncvx-cons-s3}
				\begin{aligned}
					&\norm{\x_{t+1}^0 - \1(\bx{t+1}{0})^{\T}}^2 = \norm{\x_{t}^m - \1(\bx{t}{m})^{\T}}^2\\
					&\leq \prt{\frac{1 + \rho_w^2}{2}}^m\norm{\x_t^0 - \1(\bx{t}{0})^{\T}}^2\\
					&\quad + \frac{2\alpha^2}{1-\rho_w^2}\sum_{j=0}^{m-1}\prt{\frac{1 + \rho_w^2}{2}}^{m-j-1}\norm{\nabla \Fp{j}(\x_t^j)}^2\\
					&\leq \prt{\frac{1 + \rho_w^2}{2}}^m\norm{\x_t^0 - \1(\bx{t}{0})^{\T}}^2 + \frac{2\alpha^2}{1-\rho_w^2}\sum_{j=0}^{m-1}\norm{\nabla \Fp{j}(\x_t^j)}^2\\
					&\leq \prt{\frac{1 + \rho_w^2}{2}}^m\norm{\x_t^0 - \1(\bx{t}{0})^{\T}}^2 + \frac{12\alpha^2 L^2}{1-\rho_w^2}\sum_{j = 0}^{m-1}\norm{\x_t^{j} - \1(\bx{t}{j})^{\T}}^2\\
					&\quad + \frac{12\alpha^2 L^2n}{1-\rho_w^2}\sum_{j=0}^{m-1}\norm{\bx{t}{j}-\bx{t}{0}}^2 +  \frac{6\alpha^2mn}{1-\rho_w^2}\norm{\nabla f(\bx{t}{0})}^2\\
					&\quad + \frac{6\alpha^2mn}{1-\rho_w^2}\brk{2A\prt{f(\bx{t}{0}) - \bar{f}} + B^2}.
				\end{aligned}
			\end{equation}\normalsize

		\end{proof}}

	{\kh~
	\subsection{Proof of Lemma \ref{lem:Qt}}
	\label{app:lem_Qt}
	\begin{proof}
		Substituting $\cL_t$ into Lemma \ref{lem:descent-property} yields
		\begin{equation}
			\label{eq:ncvx_descent_new}
			\begin{aligned}
				&f(\bx{t + 1}{0})\leq f(\bx{t}{0}) - \frac{m\alpha}{2}\norm{\nabla f(\bx{t}{0})}^2 + \alpha L^2\cL_t\\
				&\leq f(\bx{t}{0}) - \frac{m\alpha}{2}\norm{\nabla f(\bx{t}{0})}^2 + \frac{4\alpha L^2}{n(1-\rho_w^2)}\norm{\x_t^0 - \1(\bx{t}{0})^{\T}}^2\\
				&\quad + \frac{6m^2\alpha^3B^2(m + 4)L^2}{(1-\rho_w^2)^2} + \frac{6\alpha^3m(4 + m^2)L^2}{(1-\rho_w^2)^2}\norm{\nabla f(\bx{t}{0})}^2\\
				&\quad + \frac{12m^2\alpha^3L^2 A(4 + m)}{(1-\rho_w^2)}\prt{f(\bx{t}{0})  - \bar{f}}.
			\end{aligned}
		\end{equation}

		Combining the two recursions in Lemma \ref{lem:noncvx-con-err} and letting $\alpha \leq \prt{\frac{(1-\rho_w^2)^3}{192L^2}}^{1/2}$, we have 
		\begin{equation}
			\label{eq:ncvx_cons_t}
			\begin{aligned}
				&\norm{\x_{t+1}^0 - \1(\bx{t+1}{0})^{\T}}^2 \leq \frac{3 + \rho_w^2}{4}\norm{\x_t^0 - \1(\bx{t}{0})^{\T}}^2\\
				&\quad +  \frac{6\alpha^2mn B^2}{1-\rho_w^2}\prt{1 + \frac{12\alpha^2L^2 m(m + 4)}{(1-\rho_w^2)^2}}\\
				&\quad  +  \frac{6\alpha^2mn}{1-\rho_w^2}\prt{1 + \frac{12\alpha^2 L^2(m^2 + 4)}{(1-\rho_w^2)^2}}\norm{\nabla f(\bx{t}{0})}^2\\
				&\quad + \frac{12A\alpha^2mn}{1-\rho_w^2}\prt{1 + \frac{12\alpha^2L^2 m(m+4)}{(1-\rho_w^2)^2}}\prt{f(\bx{t}{0}) - \bar{f}}\\
				&\leq  \frac{3 + \rho_w^2}{4}\norm{\x_t^0 - \1(\bx{t}{0})^{\T}}^2 +  \frac{12\alpha^2mn B^2}{1-\rho_w^2}\\
				&\quad  +  \frac{12\alpha^2mn}{1-\rho_w^2}\norm{\nabla f(\bx{t}{0})}^2 + \frac{24A\alpha^2mn}{1-\rho_w^2}\prt{f(\bx{t}{0}) - \bar{f}},
			\end{aligned}
		\end{equation}
		where the last inequality holds by letting $\alpha\leq \frac{1-\rho_w^2}{2\sqrt{3}(m + 2)L}$. Combining \eqref{eq:ncvx_descent_new} and \eqref{eq:ncvx_cons_t} yields
		\begin{equation}\small 
			\label{eq:ncvx_fcons}
			\begin{aligned}
				&f(\bx{t + 1}{0}) - \bar{f} + \frac{16\alpha L^2}{n(1-\rho_w^2)^2}\norm{\x_{t + 1}^0 - \1(\bx{t + 1}{0})^{\T}}^2\\
				&\leq \brk{1 + \frac{12m^2\alpha^3L^2 A(4 + m)}{(1-\rho_w^2)} + \frac{384A\alpha^3L^2m}{(1-\rho_w^2)^3}}\prt{f(\bx{t}{0}) - \bar{f}}\\
				&\quad + \frac{16\alpha L^2}{n(1-\rho_w^2)^2}\norm{\x_{t}^0 - \1(\bx{t}{0})^{\T}}^2 + \frac{6m\alpha^3 L^2B^2[m(m + 4) + 32]}{(1-\rho_w^2)^3}\\
				&\quad - \frac{m\alpha}{2}\prt{1 - \frac{12\alpha^2(4 + m^2)L^2}{(1-\rho_w^2)^2} - \frac{384\alpha^2L^2}{(1-\rho_w^2)^3}}\norm{\nabla f(\bx{t}{0})}^2.
			\end{aligned}
		\end{equation}\normalsize

		{\sp Since $Q_t$ is defined as}
		\begin{align*}
			Q_t := f(\bx{t}{0}) - \bar{f} + \frac{16\alpha L^2}{n(1-\rho_w^2)^2}\norm{\x_{t}^0 - \1(\bx{t}{0})^{\T}}^2,
		\end{align*}
		relation \eqref{eq:ncvx_fcons} becomes 
		\begin{equation}\small 
			\label{eq:ncvx_fcons1}
			\begin{aligned}
				&Q_{t + 1} \leq \brk{1 + \frac{12m^2\alpha^3L^2 A(4 + m)}{(1-\rho_w^2)} + \frac{384A\alpha^3L^2m}{(1-\rho_w^2)^3}} Q_t \\
				&\quad - \frac{m\alpha}{4}\norm{\nabla f(\bx{t}{0})}^2 + \frac{6m\alpha^3 L^2B^2[m(m + 4) + 32]}{(1-\rho_w^2)^3},
			\end{aligned}
		\end{equation}
		where we invoke $\alpha\leq \min\crk{\frac{1-\rho_w^2}{4\sqrt{3}L(m+2)}, \frac{(1-\rho_w^2)^{3/2}}{16\sqrt{6}L}}$.

	\end{proof}
	}

\end{document}